\documentclass{article}

\usepackage{arxiv}
\usepackage[applemac]{inputenc} 		
\usepackage[T1]{fontenc}    		
\usepackage[colorlinks]{hyperref}      
\usepackage{url}            			
\usepackage{amsthm}
\usepackage{booktabs}       		
\usepackage{multirow}
\usepackage{tabularx}
\usepackage{longtable}
\usepackage{amsfonts}       		
\usepackage{amsmath}
\usepackage{nicefrac}       		
\usepackage{microtype}      		
\usepackage{lipsum}				
\usepackage[square,numbers]{natbib}
\usepackage{mathtools}
\usepackage{algorithm}
\usepackage{algorithmicx}
\usepackage{algpseudocode}
\usepackage{graphicx}
\usepackage{indentfirst,latexsym,bm}
\usepackage{amsmath}
\usepackage{subfigure}
\usepackage{amssymb}
\usepackage{xcolor}
\usepackage{comment}
\usepackage{enumitem}

\usepackage{tikz}
\usetikzlibrary{arrows.meta,positioning}
\usepackage{pgfplots}
\usetikzlibrary{shapes,decorations}
\usetikzlibrary{fit}

\colorlet{color1}{blue}
\colorlet{color2}{red!50!black}

\definecolor{ivory}{RGB}{218,215,203}

\definecolor{cuhkp}{RGB}{98,56,105} 	
\definecolor{cuhkpl}{RGB}{152,24,147} 	
\definecolor{cuhkb}{RGB}{219,160,1} 	
\definecolor{cuhkbd}{RGB}{178,129,0} 	
\definecolor{cuhkr}{RGB}{88,35,155}  	

\definecolor{blackp}{RGB}{0,0,0} 
\definecolor{redp}{RGB}{255,0,0}
\definecolor{orangep}{RGB}{255,128,0}
\definecolor{brownp}{RGB}{128,77,0}
\definecolor{yellowp}{RGB}{255,230,0}
\definecolor{greenp}{RGB}{128,230,0}
\definecolor{bluep}{RGB}{0,128,255}
\definecolor{purplep}{RGB}{152,24,147}
\definecolor{pinkp}{RGB}{230,0,128}        

\usepackage{hyperref}[6.83]
\hypersetup{
  colorlinks=true,
  frenchlinks=false,
  pdfborder={0 0 0},
  naturalnames=false,
  hypertexnames=false,
  breaklinks,
  allcolors = cuhkbd,
  urlcolor = color2,
}

\RequirePackage[capitalize,nameinlink]{cleveref}[0.19]

\crefname{section}{section}{sections}
\crefname{subsection}{subsection}{subsections}

\Crefname{figure}{Figure}{Figures}

\crefformat{equation}{\textup{#2(#1)#3}}
\crefrangeformat{equation}{\textup{#3(#1)#4--#5(#2)#6}}
\crefmultiformat{equation}{\textup{#2(#1)#3}}{ and \textup{#2(#1)#3}}
{, \textup{#2(#1)#3}}{, and \textup{#2(#1)#3}}
\crefrangemultiformat{equation}{\textup{#3(#1)#4--#5(#2)#6}}%
{ and \textup{#3(#1)#4--#5(#2)#6}}{, \textup{#3(#1)#4--#5(#2)#6}}{, and \textup{#3(#1)#4--#5(#2)#6}}

\Crefformat{equation}{#2Equation~\textup{(#1)}#3}
\Crefrangeformat{equation}{Equations~\textup{#3(#1)#4--#5(#2)#6}}
\Crefmultiformat{equation}{Equations~\textup{#2(#1)#3}}{ and \textup{#2(#1)#3}}
{, \textup{#2(#1)#3}}{, and \textup{#2(#1)#3}}
\Crefrangemultiformat{equation}{Equations~\textup{#3(#1)#4--#5(#2)#6}}%
{ and \textup{#3(#1)#4--#5(#2)#6}}{, \textup{#3(#1)#4--#5(#2)#6}}{, and \textup{#3(#1)#4--#5(#2)#6}}

\crefdefaultlabelformat{#2\textup{#1}#3}

\theoremstyle{plain}

\newtheorem{thm}{Theorem}[section]
\newtheorem{lemma}[thm]{Lemma}
\newtheorem{corollary}[thm]{Corollary}
\newtheorem{proposition}[thm]{Proposition}
\newtheorem{remark}[thm]{Remark}
\newtheorem{assumption}[thm]{Assumption}
\theoremstyle{plain}
\newtheorem{defn}[thm]{Definition}

\DeclareMathOperator*{\argmin}{argmin}

\newcommand{\rmn}[1]{\textup{\textrm{#1}}}

\newcommand{\R}{\mathbb{R}}
\newcommand{\N}{\mathbb{N}}
\newcommand{\Rn}{\mathbb{R}^n}
\newcommand{\Rex}{(-\infty,\infty]}
\newcommand{\Rexx}{[-\infty,\infty]}
\newcommand{\Spp}{\mathbb{S}^n_{++}}

\newcommand{\vp}{\varphi}
\newcommand{\veps}{\epsilon}
\newcommand{\dist}{\mathrm{dist}}
\newcommand{\crit}{\mathrm{crit}}
\newcommand{\cS}{\mathcal{S}}
\newcommand{\cL}{\mathcal{L}}

\newcommand{\half}{\frac{1}{2}}
\newcommand{\iprod}[2]{\langle #1, #2 \rangle}
\newcommand{\mer}{H_\tau}

\newcommand{\be}{\begin{equation}}
\newcommand{\ee}{\end{equation}}

\newcommand\prox[1]{\mathrm{prox}_{\lambda\varphi}(#1)}

\newcommand\proxs{\mathrm{prox}_{\lambda\varphi}}

\newcommand\Fnor[1]{F^{\lambda}_{\mathrm{nor}}(#1)}
\newcommand\Fnors{F^{\lambda}_{\mathrm{nor}}}
\newcommand\oFnor{F^{\lambda}_{\mathrm{nor}}}
\newcommand\Fnat[1]{F^{\lambda}_{\mathrm{nat}}(#1)}
\newcommand\oFnat{F^{\lambda}_{\mathrm{nat}}}
\newcommand\dom[1]{\mathrm{dom\\}(#1)}

\newcommand{\aff}{\mathrm{aff}}

\newcommand\oprox{\mathrm{prox}_{\lambda\varphi}}
\newcommand\env[1]{\mathrm{env}_{\lambda\varphi}(#1)}

\newcommand\envs{\mathrm{env}_{\lambda\varphi}}
\newcommand\oenv{\mathrm{env}_{\lambda\varphi}}

\newcommand\tr{\mathrm{tr}}
\newcommand\diag{\mathrm{diag}}

\newcolumntype{C}[1]{>{\centering\arraybackslash}p{#1}}

\title{A trust region-type normal map-based semismooth Newton method for nonsmooth nonconvex composite optimization}

\author{
     Wenqing Ouyang\\
 The Chinese University of Hong Kong, Shenzhen \\
 School of Data Science (SDS) \\
 Shenzhen Research Institute of Big Data (SRIBD) \\
  Shenzhen, Guangdong, China \\
  \texttt{wenqingouyang1@link.cuhk.edu.cn} \\
  \And
   Andre Milzarek\thanks{A. Milzarek is partly supported by the Fundamental Research Fund -- Shenzhen Research Institute of Big Data (SRIBD) Startup Fund JCYJ-AM20190601 and by the Shenzhen Institute of Artificial Intelligence and Robotics for Society (AIRS).} \\
  The Chinese University of Hong Kong, Shenzhen \\ School of Data Science (SDS) \\
  Shenzhen Research Institute of Big Data (SRIBD) \\
  Shenzhen Institute of Artificial Intelligence \\ and Robotics for Society (AIRS) \\
  Shenzhen, Guangdong, China \\
  \texttt{andremilzarek@cuhk.edu.cn} \\
}

\begin{document}
\maketitle

\begin{abstract}
  We propose a novel trust region method for solving a class of nonsmooth, nonconvex composite-type optimization problems. The approach embeds inexact semismooth Newton steps for finding zeros of a normal map-based stationarity measure for the problem in a trust region framework. Based on a new merit function and acceptance mechanism, global convergence and transition to fast local q-superlinear convergence are established under standard conditions. In addition, we verify that the proposed trust region globalization is compatible with the Kurdyka-{\L}ojasiewicz inequality yielding finer convergence results. We further derive new normal map-based representations of the associated second-order optimality conditions that have direct connections to the local assumptions required for fast convergence. Finally, we study the behavior of our algorithm when the Hessian matrix of the smooth part of the objective function is approximated by BFGS updates. We successfully link the KL theory, properties of the BFGS approximations, and a Dennis-Mor{\'e}-type condition to show superlinear convergence of the quasi-Newton version of our method. Numerical experiments on sparse logistic regression, image compression, and a constrained log-determinant problem illustrate the efficiency of the proposed algorithm.
\end{abstract}


\section{Introduction}

In this paper, we develop and analyze a novel normal-map based second-order approach for the composite optimization problem
\begin{align} \label{eq1-1}
\min\limits_{x\in\mathbb{R}^n}~\psi(x):=f(x)+\varphi(x),
\end{align}
where $f:\mathbb{R}^n\rightarrow\mathbb{R}$ is a continuously differentiable (not necessarily convex) function and $\varphi:\mathbb{R}^n\rightarrow(-\infty,\infty]$ is a convex, lower semicontinuous (lsc), and proper (not necessarily smooth) mapping. 

Composite-type problems of the form \eqref{eq1-1} have become a ubiquitous tool in optimization to model a large variety of applications, including, e.g., sparse $\ell_1$-regularized problems \cite{tibshirani1996regression,shevade2003simple,donoho2006compressed}, group sparse problems \cite{cotter2005sparse,yuan2006model,meier2008group}, structured dictionary learning \cite{mairal2009online,bach2011optimization}, matrix completion \cite{candes2009exact,cai2010singular}, and machine learning tasks \cite{Bis06,shalev2014understanding,BotCurNoc18}.

Many common and recent algorithmic approaches for solving \eqref{eq1-1} are based on classical forward-backward splitting techniques or proximal gradient steps \cite{FukMin81,chen1997convergence,ComWaj05,parikh2014proximal}. Specifically, at iteration $k$, the traditional forward-backward splitting method performs a gradient descent step for the smooth function $f$ followed by a proximal ``backward'' step for the nonsmooth mapping $\vp$, 
\be \label{eq:fbs-intro} x_{k+1} = \prox{x_k - \lambda\nabla f(x_k)}, \ee
where $\oprox : \Rn \to \Rn$, $\prox{x} := \argmin_{y \in \Rn} \varphi(y) + \frac{1}{2\lambda} \|x-y\|^2$ denotes the well-known proximity operator of $\varphi$, \cite{Mor65}, and $\lambda > 0$ is a positive parameter. Alternatively, the forward-backward scheme in \eqref{eq:fbs-intro} can also be interpreted as a fixed-point procedure applied to the nonsmooth equation:
\begin{align} \label{eq:natural-residual}
\oFnat : \Rn \to \Rn, \quad \Fnat{x} := x - \prox{x - \lambda\nabla f(x)} = 0.          
\end{align}
Here, the \textit{natural residual} $\oFnat$ represents the first-order necessary optimality conditions of problem \eqref{eq1-1} (see, e.g., \cref{sec:foc}) which stresses the fundamental role of $\oFnat$ and of the proximal updates \eqref{eq:fbs-intro} in the design of methodologies for solving \eqref{eq1-1}. The second-order approach investigated in this work is based on a different characterization of the associated optimality conditions of problem \eqref{eq1-1} using the so-called \textit{normal map}
\begin{align} \label{eq:normal-map}
\oFnor : \Rn \to \Rn, \quad \Fnor{z}:=\nabla f(\prox{z})+{\lambda^{-1}}(z-\prox{z})=0.
\end{align}
The normal map $\oFnor$ was initially introduced by Robinson in \cite{robinson1992normal} and has been primarily used in the context of classical variational inequalities (VI) and generalized equations for the special case where the proximity operator $\oprox$ reduces to the projection $\mathcal P_C$ onto a closed, convex set $C \subset \Rn$. We refer to \cite{facchinei2007finite} for further background. Similar to the observations in \cite{robinson1992normal,facchinei2007finite} and since the range of $\oprox$ coincides with $\dom{\partial{\varphi}}\subseteq \dom{{\varphi}}$, (see \cite[section 24]{rockafellar1970convex}), the normal map remains well-defined if $\nabla f$ is only defined on the effective domain $\dom{\vp}$. This attractive feature is a distinctive advantage of the normal map and one of the main motivations for developing normal map-based algorithms.

\subsection{Contributions}

Our basic algorithmic idea is to apply a semismooth Newton method, \cite{QiSun93,qi1993convergence}, to solve the nonsmooth equation
\be \label{eq:intro-fnor} \Fnor{z}=0. \ee
In particular, we combine approximate semismooth Newton steps for \eqref{eq:intro-fnor}, generated via an inexact CG-solver, and a globalization technique that is based on a trust region-like mechanism to control the acceptance of the Newton steps. In this way, the resulting trust region-type algorithm can be guaranteed to converge globally and locally at a q-superlinear rate. 

Our algorithmic framework follows a normal map-based approach proposed in \cite{pieper2015finite} by Pieper.
The normal map scheme developed in \cite{pieper2015finite} also uses a trust region-type strategy and has already been successfully applied in a heuristic manner in several other works, see  \cite{kunisch2016time,boulanger2017sparse,RunAigKunSto18a,ManRun20,mannelhybrid}. However, only limited convergence results are available so far. Our goal in this work is to design a modified version of this original normal map-based method that allows to establish full global and local convergence results while maintaining the favorable performance reported in \cite{pieper2015finite,kunisch2016time,boulanger2017sparse,RunAigKunSto18a,ManRun20,mannelhybrid}. We now summarize our main contributions:
\begin{itemize}
\item We propose a new normal map-based merit function for problem \eqref{eq1-1} that has more tractable descent properties than the original objective function $\psi$ and that is compatible with our trust region strategy. By incorporating this novel merit function in the trust region acceptance mechanism, we are able to globalize the semismooth Newton method for \eqref{eq:intro-fnor} and, {in contrast to \cite{pieper2015finite,ManRun20,mannelhybrid}}, a full unified convergence analysis is possible. In particular, we establish regular global convergence in terms of the normal map 
\[ \|\Fnor{z_k}\| \to 0 \quad \text{as} \quad k \to \infty, \]
where $\{z_k\}$ denotes a sequence of iterates generated by the proposed approach. In addition, we verify that the celebrated 
Kurdyka-{\L}ojasiewicz (KL) framework is applicable yielding finer global convergence results. 
\item Under mild local assumptions, we prove that our algorithm can locally turn into a pure (inexact) semismooth Newton method and that a fast q-superlinear rate of convergence can be achieved. This shows that the proposed normal map-based approach does not suffer from a \textit{Maratos-type effect}. Our analysis is based on the local behavior of the merit function allowing to make strong connections to the utilized trust region models.
\item We derive new normal map-based representations of the second-order optimality conditions of problem \cref{eq1-1}. Specifically, we show that many second-order concepts, such as the second-order sufficient conditions, strong metric subregularity, and quadratic growth conditions, have an equivalent (and often simpler) characterization using the normal map which underlines the strong link between the natural residual and $\oFnor$. We further show that some of our assumptions required for local convergence are directly connected to second-order conditions.
 
\item We study the convergence properties of a practical quasi-Newton variant of our method using BFGS updates to approximate the Hessian of $f$. After refining some classical results for BFGS updates, we show that BFGS techniques are fully compatible with the KL theory and some of our earlier convergence results. This allows us to derive a Dennis-Mor{\'e}-type condition yielding local superlinear convergence. In contrast to other nonsmooth quasi-Newton methods and techniques, \cite{ip1992local,CheYam92,Qi97,LewOve13,SteThePat17,TheStePat18}, we do not need to assume differentiability of the normal map $\oFnor$. Instead, since we only approximate the curvature information of the smooth function $f$, we will work with a slightly stronger second-order-type condition involving $\nabla^2 f$.   
 
To the best of our knowledge, this is the first work which fully links the KL theory, boundedness of the BFGS updates, and a Dennis-Mor{\'e}-type condition to establish fast q-superlinear convergence of the semismooth quasi-Newton method.

\item Finally, numerical experiments are conducted on a sparse logistic regression, a nonconvex image compression, and a constrained log-determinant optimization problem which demonstrate the favorable performance of the normal map-based semismooth Newton method.  
\end{itemize}

\subsection{Related Work}
The importance and popularity of semismoothness, \cite{Mif77,QiSun93,qi1993convergence}, and of the semismooth Newton method, \cite{QiSun93,qi1993convergence,QiSun99}, stem from the fact that nonsmooth versions of Newton's method applied to a nonlinear, nonsmooth equation 
\[ F(x) = 0 \] 
are well-defined and can be shown to converge locally at least q-superlinearly under suitable conditions if the mapping $F$ is semismooth. 
While the local convergence of the semismooth Newton method can be established in a broad and universal context, globalization techniques and global convergence results are typically more tailored to the considered application and can depend on the specific problem structure. 
In the last decades, a variety of globalization schemes for semismooth Newton methods have been proposed for different problems. 
This includes line-search based globalization techniques \cite{HanPanRan92,qi1993convergence,MarQi95} (on suitable merit functions, such as, e.g., $\frac12\|F(x)\|^2$), specialized globalization schemes for complementarity problems and KKT systems \cite{DeLFacKan96,FerKanMun99,KanQi99,MunFacFerFisKan01}, projection methods for monotone equations \cite{SolSva01,xiao2018regularized}, and lesser studied trust region-type globalization mechanisms \cite{ulbrich2001nonmonotone}. There is also a vast amount of literature on specialized semismooth Newton methods that are based on the natural residual. For $\ell_1$-problems with $\varphi \equiv \mu \|\cdot\|_1$, $\mu > 0$, various semismooth Newton schemes have been proposed in \cite{GriLor08,milzarek2014semismooth,HanRaa15,ByrChiNocOzt16}. An extension of the work \cite{milzarek2014semismooth} to general composite-type problems can also be found in \cite{milzarek2016numerical}. Moreover, in \cite{PatBem13,PatSteBem14,SteThePat17}, the authors introduce the so-called forward-backward envelope (FBE) as a smooth merit function for \eqref{eq1-1} and different semismooth Newton methods with line search-type globalization are analyzed. 
 
As mentioned, Robinson's normal map has been mainly used for classical variational inequalities and generalized equations. In particular, it is the basis of the path search damped Newton methods investigated in \cite{ralph1994global,dirkse1995path} and of a projected gradient hybrid scheme for nonlinear complementarity problems proposed by Ferris and Ralph in \cite{FerRal95}. In \cite{HanSun97}, Han and Sun discuss the convergence properties of a Newton and quasi-Newton method applied to a normal map formulation of VI problems on polyhedral sets. 
In \cite{ZhoTohSun03}, Zhou, Toh, and Sun propose a normal map-based smoothing Newton method for an $\ell_2$-norm regression problem. 
For more background on the normal map and additional classical normal map-based approaches, let us again refer to \cite{facchinei2007finite}. 

To the best of our knowledge, there are only few works that directly utilize the normal map to solve general composite-type problems of the form \eqref{eq1-1}. Specifically, besides Pieper's PhD thesis \cite{pieper2015finite}, this mainly includes the works \cite{ManRun20,mannelhybrid} by Mannel and Rund where local properties of a quasi-Newton variant of Pieper's normal map-based trust region method using Broyden-like updates are established in a Banach space setting. 
%
%

The semismooth Newton method is also an integral component of various related classes of algorithms for solving \eqref{eq1-1}. In proximal Newton approaches \cite{lee2014proximal}, the semismooth Newton method is used as a subproblem solver to compute the proximal Newton steps. 
We refer to \cite{BecFadOch19,KanLec21} for recent applications of this technique. Furthermore, the semismooth Newton method is the core of several augmented Lagrangian and proximal point algorithms for semidefinite programming and nuclear and spectral norm or Lasso-type  problems, \cite{ZhaSunToh10, JiaSunToh14,YanSunToh15,CheLiuSunToh16,LiSunToh18}. 

Finally, we note that various types of nonsmooth trust region methods have been studied and analyzed for the general optimization problem $\min_x\,\psi(x)$ during the last decades. A majority of these approaches are based on abstract model functions that are often not further specified. In \cite{DenLiTap95}, a nonsmooth trust region method is proposed for $\min_x\,\psi(x)$ under the assumption that $\psi$ is regular. A nonsmooth trust region algorithm with an abstract first-order model is investigated in \cite{QiSun94}. Global convergence properties of nonsmooth trust region methods are typically shown under strong assumptions on the accuracy of the model and rely on the concept of a ``strict model'' introduced by Noll in \cite{Nol10}. This can limit the direct applicability and numerical tractability of nonsmooth trust region approaches. In \cite{ChrDLRMey20}, Christof, De Los Reyes, and Meyer propose a hybrid method that combines simpler quadratic trust region models and a more complicated second model to overcome some of the practical limitations of strict models. More related to our work, Chen, Milzarek, and Wen, \cite{chen2020trust}, propose a normal map-based trust region framework for composite problems. The approach in \cite{chen2020trust} is based on steepest descent-type directions and truncations to control the accuracy of the utilized quadratic models. As a consequence, the convergence analysis in \cite{chen2020trust} requires relatively strong assumptions and is closer to the analyses of some of the other trust region methods mentioned here. A recent nonsmooth trust region scheme with proximal quasi-Newton models is presented in \cite{AraBarOrb21}. 

A more detailed discussion of related literature concerning nonsmooth second-order theory and (nonsmooth) quasi-Newton schemes can be found in \cref{sec:sop} and \cref{sec:qnm}. 
\subsection{Organization}

In \cref{sec:foc}, we introduce several first-order optimality conditions and the normal map and we list required concepts from nonsmooth analysis. In \cref{sec:alf}, we motivate our algorithmic framework. Specifically, we introduce a novel merit function for \eqref{eq1-1} and a new reduction ratio controlling the acceptance of trust region steps. Basic global convergence properties of our algorithm are derived in \cref{sec2}. In \cref{sec:KL_conv}, we investigate convergence of the approach under the KL inequality. 
In \cref{sec:loc-super-conv}, we discuss local convergence properties and transition to fast local q-superlinear convergence. 
In \cref{sec:sop}, we derive a novel representation of the second-order optimality conditions for problem \cref{eq1-1} using the normal map perspective. 
In \cref{sec:qnm}, we  present an in-depth study of a BFGS-type version of our method. 
Finally, in \cref{sec:ne} we illustrate and discuss the numerical performance of our algorithm.

\begin{table}[t]
  \centering
   \begin{tabular}{cll}  
 \cmidrule[1pt](){1-3} 
    Notation & \multicolumn{2}{c}{Description and Reference} \\[0.5ex]  
  \cmidrule(){1-3} \\[-2.5ex]
   $\oFnor$ & normal map & $\oFnor(z) = \nabla f(\prox{z}) + \frac{1}{\lambda}(z-\prox{z})$  \\[0.5ex]
   $H_\tau$ & merit function & $H_\tau(z) = \psi(\oprox(z)) + \frac{\tau\lambda}{2}\|\oFnor(z)\|^2$  \\[0.5ex]
   $\chi$ & criticality measure & $\chi(z) = \|\oFnor(z)\|$ \\[0.5ex]
   $\bar q_k$; $\bar s_k$; $s_k$ & \multicolumn{2}{l}{approximate CG-solution of \eqref{eq3-4}; lifted and rescaled step \eqref{eq:def-sk-bar-sk}}  \\[0.5ex]
   $\mathrm{pred}$ & \multicolumn{2}{l}{predicted reduction $\mathrm{pred}(z,s,\Delta,\nu)$ \eqref{eq:pred}} \\[0.5ex]
   $B$; $D$ & \multicolumn{2}{l}{(Approximation of) $\nabla^2 f$; generalized derivative of $\oprox$} \\[0.5ex]
   $\eta_1$; $\eta_2$ & \multicolumn{2}{l}{trust region parameter}   \\[0.5ex]
   $\gamma_0$; $\gamma_1$; $\gamma_2$; $\Delta_{\min}$ & \multicolumn{2}{l}{trust region parameter} \\[0.5ex]
   $\mathcal S$; $n_{\mathcal S}$ & \multicolumn{2}{l}{successful iterations; $n_{\mathcal S}(k) := \vert \mathcal S \cap \{0,1,...,k-1\}\vert$} \\[0.5ex]
 \cmidrule[1pt](){1-3} \\[-1.5ex]
  \end{tabular}
  \caption{List of Variables, Parameters, and Functions.}
  \label{table_nota}
  \end{table}

\subsection{Notation}
Our notation is standard and follows \cite{rockafellar1970convex,rockafellar2009variational,clarke1990optimization}. 
By $\iprod{\cdot}{\cdot}$ and $\|\cdot\|$ we denote the standard Euclidean inner product and norm. For matrices, the norm $\|\cdot \|$ is the standard spectral norm. The sets of symmetric and symmetric positive definite $n \times n$ matrices are denoted by $\mathbb S^n$ and $\Spp$, respectively. For a given matrix $A \in \Spp$, we define the norm $\|x\|_A := \sqrt{\iprod{x}{Ax}}$. For two matrices $A, B \in \mathbb S^n$, we write $A\succeq B$ if $A-B$ is positive semidefinite. We use $\mathrm{ri}(S)$ to denote the relative interior of a convex set $S \subset \Rn$.

The effective domain of a function $\theta : \Rn \to \Rex$ is defined as $\dom{\theta}=\{x \in \Rn : \theta(x)<\infty\}$. Let $x \in \dom{\theta}$ be given. The {lower directional epi-derivative} or {lower subderivative} of $\theta$ at $x$ in the direction $h \in \Rn$ is defined as follows 
\[ \theta^\downarrow_{-}(x;h) := \liminf_{t \downarrow 0, \, \tilde h \to h}~\Delta_t \;\! \theta(x)(\tilde h), \quad \Delta_t \;\! \theta(x)(h) := \frac{\theta(x+th)-\theta(x)}{t}. \]
We say that $\theta$ is {directionally epi-differentiable} at $x$ in the direction $h \in \Rn$ with {epi-derivative} $\theta^\downarrow(x;h)$ if and only if for every sequence $(t_k)_k$, $t_k \downarrow 0$, it holds that
\be \label{eq:def-epi} \left[ \begin{array}{ll} \displaystyle\liminf_{k \to \infty}~\Delta_{t_k} \;\! \theta(x)(h^k) \,\, \geq \theta^\downarrow(x;h) & \text{for every sequence } h^k \to h, \\ \displaystyle\limsup_{k \to \infty}~\Delta_{t_k} \;\! \theta(x)(h^k) \leq \theta^\downarrow(x;h) & \text{for some sequence } h^k \to h. \end{array} \right. \ee
The function $\theta$ is called {directionally differentiable} at $x$ in the direction $h$ if the limit $\theta^\prime(x;h) = \lim_{t \downarrow 0} \Delta_t \;\! \theta(x)(h)$ exists. Moreover, we say that $\theta$ is {semidifferentiable} or {directionally differentiable in the sense of Hadamard} at $x$ in the direction $h \in \Rn$ if the limit $\lim_{t \downarrow 0, \tilde h \to h} \Delta_t \;\! \theta(x)(\tilde h)$ exists. In this case, we will also use the term $\theta^\prime(x;h) $ to denote its limit. Let us note that the latter two definitions do also make sense for mappings $F : \Rn \to \R^m$. The set $\mathcal{R}(F):=\{y \in \R^m: \exists~x \in \Rn \, \text{with} \, y=F(x)\}$ is the range of the mapping $F : \Rn \to \R^m$. In this paper, $\partial \theta$ denotes Clarke's subdifferential for extended-valued functions or for locally Lipschitz continuous mappings $\theta : \Rn \to \R^m$, see, e.g., \cite[section 8.J]{rockafellar2009variational} or \cite[section 2.4]{clarke1990optimization}. 

Throughout this work, we assume that $f:\mathbb{R}^n\rightarrow \mathbb{R}$ is continuously differentiable and $\varphi:\mathbb{R}^n\rightarrow(-\infty,\infty]$ is a convex, lower semicontinuous, and proper mapping.

\section{First-Order Optimality and Preliminaries} \label{sec:foc}

%
%

 {A point $\bar{x} \in \dom{\vp}$ is a stationary point of \eqref{eq1-1} if $0 \in \partial\psi(\bar x)=\nabla f(\bar{x})+\partial\varphi(\bar{x})$ and we use $\mathrm{crit}(\psi)$ to denote the set of all stationary points of $\psi$. Here, $\partial\vp$ is the standard subdifferential for convex functions.} The optimality condition $0 \in \partial\psi(\bar x)$ can be equivalently represented as a nonsmooth equation:
\[ x \in \mathrm{crit}(\psi) \; \iff \; \Fnat{x} := x - \prox{x - \lambda \nabla f(x)} = 0. \]
As mentioned, $\oprox$ denotes the proximity operator of $\varphi$ with respect to the scalar $\lambda$. 
The proximity operator is a {firmly nonexpansive mapping}, i.e.,
\be \label{eq:prox-nonexp} \|\prox{x}-\prox{y}\|^2\leq\langle x-y,\prox{x}-\prox{y} \rangle \quad \forall~x,y \in \Rn. \ee
In particular, $\oprox$ is globally Lipschitz continuous with constant 1. Moreover, the proximity operator can be characterized by the associated optimality conditions of its underlying optimization problem:
\begin{align}\label{eq1-3}
\prox{x}\in x-\lambda \partial\varphi(\prox{x}). 
\end{align}
We will also work with the Moreau envelope $\env{x}:=\min_y\varphi(y)+\frac{1}{2\lambda}\|x-y\|^2$. It is well known that $\oenv$ is convex and continuously differentiable and its gradient is given by $\nabla\env{x}=\frac{1}{\lambda}(x-\prox{x})$. Let us further note that due to the convexity of $\vp$ and differentiability of $f$, Clarke's subdifferential $\partial \psi$ coincides with the regular and limiting subdifferential of $\psi$. We refer to \cite{Mor65,rockafellar2009variational,BauCom11} for additional details and background. 

Next, we summarize the different stationarity concepts for problem \eqref{eq1-1} and connect them to Robinson's normal map $\oFnor$.
Specifically, we show that every solution $\bar z$ of the nonsmooth equation
\begin{align} \label{normal_eq} \Fnor{z} := \nabla f(\prox{z})+{\lambda^{-1}}(z-\prox{z}) =0 \end{align}
corresponds to a stationary point of the problem \eqref{eq1-1} via $\bar{x}=\prox{\bar{z}}$.

\begin{lemma}
\label{lemma1-1}
Let $\lambda > 0$ be given. The following conditions are mutually equivalent:
\begin{itemize}
\item[\rmn{(i)}] It holds that $0\in\nabla f(\bar{x})+\partial\varphi(\bar{x})$.
\item[\rmn{(ii)}] The point $\bar{x}$ is a solution of the fixed-point type equation $\Fnat{\bar{x}}=0$. 
\end{itemize}
Furthermore, if $\bar x$ is a stationary point of \eqref{eq1-1}, then $\bar{z}=\bar{x}-\lambda \nabla f(\bar{x})$ is a zero of $\oFnor$. Conversely, if $\bar{z}$ is a zero of the normal map $\oFnor$, it holds that $\prox{\bar{z}} \in \crit{(\psi)}$.
\end{lemma}
\begin{proof} The inclusion in (i) is equivalent to $\bar{x}\in\bar{x}-\lambda \nabla f(\bar{x})-\lambda \partial\varphi(\bar{x})$. Since the proximity operator is uniquely determined by \eqref{eq1-3}, this implies that condition (i) is equivalent to $\bar{x}=\prox{\bar{x}-\lambda \nabla f(\bar{x})}$. 
Next, let us suppose that $\bar{x}$ is a zero of the natural residual $\oFnat$. Then, setting $\bar{z}=\bar{x}-\lambda \nabla f(\bar{x})$, it follows 
\[ \Fnor{\bar{z}}= \nabla f(\prox{\bar x - \lambda \nabla f(\bar x)}) + {\lambda^{-1}}(\Fnat{\bar x} - \lambda \nabla f(\bar x)) = 0. \] 
Conversely, let $\bar z$ be a solution of \eqref{normal_eq} and let us set $\bar{x}=\prox{\bar{z}}$. Rearranging the terms in $\oFnor$, this yields $\bar z = \bar x - \lambda \nabla f(\bar x)$ and $\Fnat{\bar{x}} = \bar x - \prox{\bar z} =0$. \end{proof}

The next result establishes a subtler connection between the natural residual and the normal map. 

\begin{lemma} \label{lemma:conn-nat-nor} Let $\lambda > 0$ and $x \in \dom{\partial\vp}$ be given. Then, it holds that
\[ \frac{1}{\lambda} \|\Fnat{x}\| \leq \inf_{v \in \partial\psi(x)} \|v\| = \mathrm{dist}(0,\partial\psi(x)) = \inf_z~\{\|\Fnor{z}\| : x = \prox{z}\}. \]
\end{lemma}
\begin{proof} The first inequality is well known, see, e.g., \cite[Theorem 3.5]{DruLew18}. A full proof of \cref{lemma:conn-nat-nor} can be found in \cite[Lemma 4.1.6]{milzarek2016numerical} and will be omitted here. 
\end{proof}

 
Let us recall the definition of semismoothness. Following \cite{QiSun93,ulbrich2011semismooth}, a mapping $F : \Rn \to \Rn$ is said to be semismooth at $x$ if $F$ is Lipschitz continuous in a neighborhood of $x$, directionally differentiable at $x$, and 
\begin{align}
\label{eq2-1}
\sup_{M\in\partial F(x+h)}\| F(x+h)-F(x)-Mh  \| = o(\|h\|)  \quad \text{as} \quad h\rightarrow 0.
\end{align}
If \eqref{eq2-1} holds for all $M\in \mathcal M(x+h)$, where $\mathcal M : \Rn \rightrightarrows \R^{n\times n}$ is a set-valued mapping (that can be different from Clarke's subdifferential $\partial F$), then $F$ is called semismooth at $x$ with respect to $\mathcal M$. 

\begin{lemma}
\label{prop2-6}
Suppose that  $f$ is twice continuously differentiable in a neighborhood of $\prox{\bar{z}}$ and let us assume that $\proxs$ is semismooth at $\bar{z}$. We  define the following set-valued mapping $ \mathcal M^\lambda : \Rn \rightrightarrows \R^{n \times n}$:
\be \label{eq:gen-deriv} \mathcal{M}^{\lambda}(z) := \{M = \nabla^2f(\prox{z})D+\tfrac{1}{\lambda}(I-D): \; D\in\partial\prox{z}\}. \ee
Then, $\oFnor$ is semismooth at $\bar z$ with respect to $\mathcal{M}^{\lambda}$.
\end{lemma}
%
%
%

 {\cref{prop2-6} readily follows from existing chain rules for semismooth functions, see, e.g., \cite[Theorem 7.5.17]{facchinei2007finite} or \cite[Proposition 3.8]{ulbrich2011semismooth}. Finally, we state several structural properties of the generalized derivatives of the proximity operator which are used in the subsequent sections. We refer to \cite[Proposition 1]{MenSunZha05} and \cite[Lemma 3.3.5]{milzarek2016numerical} for a detailed derivation of \cref{prop2-5}.}

\begin{lemma} \label{prop2-5} 
Let $D\in\partial\prox{x} \subseteq \R^{n \times n}$, $x \in \Rn$, be an arbitrary generalized derivative. Then, it holds that: 
\begin{itemize}
\item[\rmn{(i)}] Both $D$ and $I-D$ are symmetric and positive semidefinite matrices.
\item[\rmn{(ii)}] The matrix $D(I-D)$ is positive semidefinite.
\end{itemize}
\end{lemma}
%


\section{Algorithmic Framework}
\label{sec:alf}
We now develop and motivate our algorithmic approach. 
We split and organize our discussion according to the different main  components of the algorithm. 

\subsection{Semismooth Newton Steps}
Following the original normal map-based approach proposed in \cite{pieper2015finite}, our core idea is to apply the semismooth Newton method, \citep{QiSun93,qi1993convergence}, in order to solve the nonsmooth equation $\Fnor{z}=0$. Specifically, at iteration $k\in\mathbb{N}$, we consider semismooth Newton steps of the form:
\begin{align}\label{eq3-1}
M_ks_k=-\Fnor{z_k}, \quad z_{k+1}=z_k+s_k, \quad M_k\in\mathcal{M}^{\lambda}(z_k),
\end{align}
where the set-valued mapping of generalized derivatives $\mathcal{M}^{\lambda}:\mathbb{R}^n\rightrightarrows\mathbb{R}^{n\times n}$ is given as in \eqref{eq:gen-deriv}. {The method described in \cite{pieper2015finite} embeds this basic step in a trust region-like framework to elegantly unify regularization schemes and inexact solution methods for the linear system of equations \eqref{eq3-1}. 
In this paper, we provide a detailed convergence theory for a modified version of the method developed in \cite{pieper2015finite}. In particular, we utilize a novel merit function and acceptance mechanism which ultimately allows us to derive some of the first full convergence results for this type of methodology. 
 
Let $M_k = \nabla^2 f(\prox{z_k})D_k + \frac{1}{\lambda}(I-D_k)$ with $D_k \in \partial \prox{z_k}$ be given. We first notice that the matrix $M_k$ is typically not symmetric. However, multiplying the linear equation in \eqref{eq3-1} with the symmetric matrix $D_k = D_k^\top$ from the left, we can obtain the following symmetric linear system:
\begin{align}\label{eq3-2}
D_k M_ks=-D_k\Fnor{z_k}.
\end{align}
Since $D_kM_k$ is symmetric, standard approaches for solving \eqref{eq3-2} can be applied. Furthermore, it is often possible to exploit the structure of the generalized derivative $D_k$ to reduce the dimension of the linear system \eqref{eq3-2}, see \cref{lemma3-8}. Following \cite{pieper2015finite}, we integrate this lower dimensional system in a trust region framework and to use the Steihaug-CG method, \cite{steihaug1983conjugate}, to solve it inexactly. 

The following lemma reveals that there is a close relationship between the linear systems in \eqref{eq3-1} and \eqref{eq3-2}. 
\begin{lemma}
\label{lemma2-1} Let $B, D \in \mathbb{S}^n$ be symmetric matrices and let $r \in \Rn$ and $\epsilon \geq 0$ be given. 
Let us set $M := B D + \frac{1}{\lambda}(I-D)$ and assume that $y$ satisfies the condition $\|D(My + r)\| \leq \epsilon$. Then, setting $x = y - \lambda (My + r)$, it follows $\|Mx + r\| \leq \|I-\lambda B\| \epsilon$. 
\end{lemma}
\begin{proof}
Applying the definition of $M$, we directly obtain
\begin{align}\label{eq3-3}
\lambda M =\lambda BD +I-D=(\lambda B -I)D+I
\end{align}
and
%
$ \|Mx+r\|=\|My-\lambda M (My+r)+r\| =\|(I-\lambda B)D(My+r)\| \leq  \|I - \lambda B \| \epsilon$. 
\end{proof}

Hence, solutions of the full system \eqref{eq3-1} can be recovered by solving the reduced and symmetric system \eqref{eq3-2}. 

\subsection{Trust Region Globalization} \label{sec:tr-glob}

We now develop a trust region framework in order to control the quality of the generated inexact semismooth Newton steps and to ensure global convergence of the approach. Formally, we can design a trust region subproblem associated with the linear system \eqref{eq3-2} as follows: 
%
\begin{align}\label{eq3-4}
  \min_{q}~m_k(q):=\iprod{\Fnor{z_k}}{D_kq} + \frac{1}{2} \iprod{M_kq}{D_kq} \quad \text{s.t.} \quad \|q\|\leq\Delta_k,
\end{align}
where $\Delta_k$ is the trust region radius. 
Let $\bar q_k$ denote an approximate solution of the subproblem \eqref{eq3-4} returned by the Steihaug-CG method. Motivated by our previous discussion, we can then generate a lifted and rescaled step $s_k$ via
\be \label{eq:def-sk-bar-sk} \bar s_k = \bar q_k - \lambda  (\Fnor{z_k}+M_k\bar q_k) \quad \text{and} \quad s_k = \min\left\{1,{\Delta_k}/{\|\bar s_k\|}\right\} \bar s_k, \ee
%
which corresponds to an approximate step for the original system \eqref{eq3-1} that additionally satisfies the constraint $\|s_k\| \leq \Delta_k$. Next, we briefly discuss the model $m_k$ and present our acceptance mechanism for $s_k$.

 {\textit{The Trust Region Model $m_k$.} In order to motivate the model $m_k$, let us at this point assume that $f$ and the proximity operator $\proxs$ are sufficiently smooth. Let us consider the auxiliary function $z \mapsto \varrho(z) := (\psi \circ \proxs)(z)$ and let us set $D_k = D\prox{z_k}$. Due to $\psi(\prox{z}) = f(\prox{z}) + \env{z} - \frac{1}{2\lambda} \|z-\prox{z}\|^2$, we obtain
\begin{align*} \nabla \varrho(z_k) & = D_k \nabla f(\prox{z_k}) +\frac{1}{\lambda}(z_k-\prox{z_k}) - \frac{1}{\lambda}(I-D_k)(z_k-\prox{z_k}) = D_k\Fnor{z_k}, \\ \nabla^2 \varrho(z_k)[q,q] & = \iprod{\Fnor{z_k}}{D^2\prox{z_k}[q,q]} + \iprod{D_kq}{[\nabla^2 f(\prox{z_k})D_k + \frac{1}{\lambda}(I-D_k)]q} \quad \forall~q \in \Rn.\end{align*}
%
%

Thus, $m_k$ can be interpreted as a nonsmooth second-order model for the function $\varrho$ that omits the curvature term $q \mapsto \iprod{\Fnor{z_k}}{D^2\prox{z_k}[q,q]}$. Since this term vanishes at solutions of the nonsmooth equation \eqref{normal_eq}, this strategy shares similarities with the traditional Gauss-Newton method applied to nonlinear least-squares problems. In \cref{sec:loc-super-conv}, we will investigate this outlined connection between $\psi$, $\psi\circ\proxs$, and $m_k$ rigorously without requiring differentiability of the proximity operator. Our discussion demonstrates that $m_k$ is not just a by-product of our globalization but it indeed is a proper model for the minimization problem $\min_z\,(\psi\circ\proxs)(z)$. This feature will become important in our local convergence analysis.

\textit{Accepting Trust Region Steps.} As in classical trust region methods, we base the acceptance of the current trust region trial step $z_k + s_k$ on a reduction ratio test ``$\rho_k \geq \eta$'', $\eta > 0$, where 
\[ \rho_k = \frac{\mathrm{ared}(s_k)}{\mathrm{pred}(s_k)} \equiv \frac{\mathrm{ared}_k}{\mathrm{pred}_k}. \]
The ratio $\rho_k$ compares the actual reduction ``$\mathrm{ared}_k$'' (based on the objective function or a suitable merit function)  
with some model-based predicted reduction ``$\mathrm{pred}_k$''. 
We are specifically interested in an acceptance mechanism that can ensure the following global and local features:
\begin{itemize}
\item Accumulation points of a sequence generated by the normal map-based approach should be solutions of the equation $\Fnor{z} = 0$. 
\item The condition ``$\rho_k \geq \eta$'' can be satisfied locally under suitable assumptions and if the involved subproblems are solved with sufficiently high accuracy. This should allow transition to fast local convergence. 
\end{itemize}
%
{Based on the choice of $m_k$ and our previous discussion}, a first potential candidate for the reduction ratio is given by: 
\[ \tilde \rho_k := \frac{\psi(\prox{z_k}) - \psi(\prox{z_k + \bar q_k})}{-m_k(\bar q_k)}. \] 
The ratio $\tilde \rho_k$ compares the reduction of the auxiliary function $\psi\circ\proxs$ with the reduction predicted by $m_k$. Acceptance based on $\tilde \rho_k$ coincides with traditional trust region mechanisms for $\psi\circ\proxs$. In fact, the normal map approaches in \cite{pieper2015finite,kunisch2016time,ManRun20} utilize this criterion. Since $\tilde\rho_k$ only measures the quality of $\bar q_k$ and not of  $s_k$, this trust region globalization generally can not ensure that accumulation points of the generated iterates are solutions of \eqref{normal_eq}. In addition, the predicted model decrease $-m_k(\bar q_k)$ and $\tilde\rho_k$ are meaningless in certain situations, e.g., if $D_k = 0$. We refer to \cite[Remark 3.9]{pieper2015finite} for more comments. 

In this work, we propose to measure the actual reduction based on a novel merit function that combines the function $\psi\circ\oprox$ and the normal map $\oFnor$. 

\begin{defn} \label{def:mer}
Let $\lambda > 0$, $\tau \in (0,1)$ be given. We define the merit function
$$ \mer: \Rn \to \R, \quad \mer(z):=\psi(\prox{z})+\frac{\tau\lambda}{2}\|\Fnor{z}\|^2. $$
\end{defn}

The choice of the merit function is not trivial. In particular, $\mer$ needs to be compatible with the truncated semismooth Newton-type step $s_k$ and $\mer$ should possess certain descent properties that prevent stagnation or failure of the trust region process. 
Based on \cref{def:mer}, we define the actual reduction term $\mathrm{ared}_k$ via:  
\[ \mathrm{ared}_k := \mer(z_k) - \mer(z_k+s_k). \]

Our choice of $\mathrm{pred}_k$ is mainly motivated by the classical Cauchy decrease condition. In order to control the accuracy of the inexact solutions $\bar q_k$ of the subproblem \eqref{eq3-4}, we typically require $\bar q_k$ to satisfy a Cauchy decrease condition:
\[ -m_k(\bar q_k)\geq c \chi(z_k)\min\{1,\Delta_k,\chi(z_k)\}, \quad c > 0. \]
Here, $\chi:\mathbb{R}^n\rightarrow\mathbb{R}_{+}$ is a continuous criticality measure. Throughout this work, we will use with the criticality measure 
\[ \chi(z) := \|\Fnor{z}\|. \]

As the model $m_k$ itself does not capture the reduction achieved by the lifted step $s_k$, we directly define the reduction term $\mathrm{pred}_k := \mathrm{pred}(z_k,s_k,\Delta_k,\nu_k)$ via: 
\begin{align}  \label{eq:pred} \mathrm{pred}(z,s,\Delta,\nu) & := \frac{\tau\chi(z)}{2} \min\{\lambda,\Delta,\lambda\chi(z)\} \\ & \hspace{8ex} + \frac{\nu \chi(z)}{\min\{\Delta,\lambda\chi(z)\}} \|\prox{z+s} - \prox{z}\|^2, \nonumber \end{align}
where $\tau \in (0,1)$ and $\nu \in [0,1)$ are given. In the next sections, we show that this choice of $\rho_k$, $\mathrm{ared}_k$, and $\mathrm{pred}_k$ meets all the mentioned requirements. The full details of the method are presented in \cref{algo2}. Notice that we allow the usage of approximate Hessian information $B_k \approx \nabla^2 f(\prox{z_k})$. We call iteration $k$ successful if $z_k + s_k$ is accepted as new iterate, i.e., $z_{k+1}=z_k+s_k$.

 \begin{algorithm}[t]
        \caption{A Trust Region Normal Map Semismooth Newton Method}
         \label{algo2}
        \begin{algorithmic}[1]
            \Require Choose an initial point $z_0\in\mathbb{R}^n$, $B_0 \in \mathbb{S}^n$, $\lambda>0$, and $\{\epsilon_k\} \subset \R_{+}$. Set iteration $k=0$.
            \While{$F_{\text{nor}}^{\lambda}(z_k)\neq0$}
            \State Choose $D_k\in\partial\text{prox}_{\varphi}^{\lambda}(z_k)$ and set $M_k=B_kD_k+\frac{1}{\lambda}(I-D_k)$.
            \State Run the Steihaug-CG method with $S = D_k M_k$, $g = D_k \Fnor{z_k}$, $\Delta = \Delta_k$, and $\epsilon = \epsilon_k \geq 0$ returning $\bar q_k = q$;
            \State Set $\bar{s}_k=\bar q_k-\lambda (F_{\text{nor}}^{\lambda}(z_k)+M_k\bar q_k)$ and $s_k=\min\{1,\frac{\Delta_k}{\|\bar{s}_k\|}\}\bar{s}_k$;
             \If{$\rho_k=\frac{\mer(z_k)-\mer(z_k+s_k)}{\mathrm{pred}(z_k,s_k,\Delta_k,\nu_k)} < \eta_1$}
             \State Set $z_{k+1}=z_{k}$ and $B_{k+1} = B_k$;
             \Else
             \State Set $z_{k+1}=z_k+s_k$ and choose $B_{k+1} \approx \nabla^2 f(\prox{z_{k+1}})$;
             \EndIf
             \State Set $\Delta_{k+1}$ based on $\rho_k$ by invoking \cref{algo3};
             \State $k\gets k+1$;
            \EndWhile
        \end{algorithmic}
    \end{algorithm}

\textit{Updating the Trust Region Radius}. The trust region radius $\Delta_k$ is updated as usual based on the ratio $\rho_k$. Inspired by \cite{ulbrich2001nonmonotone}, we also consider a strategy that requires the updated trust region radius $\Delta_{k+1}$ to satisfy $\Delta_{k+1} \geq \Delta_{\min} \geq 0$ if the iteration $k$ was successful. In most of our results, we assume $\Delta_{\min}$ to be a positive (small) parameter. Our update scheme is summarized in \cref{algo3}.

\begin{algorithm}[t]
        \caption{Update of the Trust Region Radius}
        \label{algo3}
        \begin{algorithmic}[1]  
            \Require Input: $\Delta_k,\rho_k$. Let $\eta_1\leq\eta_2<1$ and $0<\gamma_0<\gamma_1<1<\gamma_2$, and $\Delta_{\min} \geq 0$ be fixed. \vspace{0.1ex}
           \State Select $\Delta_{k+1}\in \begin{cases} (\gamma_0\Delta_k,\gamma_1\Delta_k] & \text{if }  \rho_k < \eta_1, \\ [\gamma_1\Delta_k,\max\{\Delta_{\min},\Delta_k\}]\cap[\Delta_{\min},\infty) & \text{if } \rho_k\in[\eta_1,\eta_2)\\ (\Delta_k,\max\{\Delta_{\min},\gamma_2\Delta_k\}]\cap[\Delta_{\min},\infty) & \text{if } \rho_k \geq \eta_2 \end{cases}$;
           \State Return $\Delta_{k+1}$;
        \end{algorithmic}
\end{algorithm}

     \begin{algorithm}[th]
        \caption{The Steihaug-CG Method}
          \label{algo1}
        \begin{algorithmic}[1]  
            \Require Input: $S \in \R^{n\times n}$, $g \in \Rn$, $\epsilon, \Delta \geq 0$. Set $q_0=0$, $r_0=g$, $p_0 = -g$ and $i=0$.
            \If{$\|r_0\|<\epsilon$}
            \State return $q = q_0$;
            \EndIf
            \While{$i\leq n-1$}
                    \If{$\iprod{p_i}{Sp_i}\leq0$}
                        \State Compute $\alpha_{i}$ such that $\alpha_i=\argmin\limits_{\|q_i+\alpha_ip_i\|=\Delta}m(q_i+\alpha_ip_i)$ and return $q = q_i+\alpha_ip_i$;
                        \Else
                        \State Set $\alpha_i=\frac{\langle r_i,r_i \rangle}{\langle p_i,Sp_i\rangle}$ and $q_{i+1}=q_i+\alpha_ip_i$;
                        \If{$\|q_{i+1}\|\geq\Delta$}
                        \State Reset $\alpha_i$ such that $\alpha_i\geq0$ and $\|q_i+\alpha_ip_i\|=\Delta$ and return $q = q_i+\alpha_ip_i$;
                        \EndIf
                        \State Set $r_{i+1}=r_i+\alpha_iSp_i$;
                        \If{$\|r_{i+1}\|<\epsilon$}
                        \State Return $q = q_{i+1}$;
                        \EndIf
                        \State Set $\beta_{i+1}=\frac{\|r_{i+1}\|^2}{\|r_i\|^2}$ and $p_{i+1}=-r_{i+1}+\beta_{i+1}p_i$;
                    \EndIf
                    \State $i\gets i+1$
            \EndWhile
        \end{algorithmic}
      \end{algorithm}
    
\subsection{Properties of the CG-Method}

In this subsection, we collect several properties of the Steihaug-CG method and of the linear systems \eqref{eq3-1} and \eqref{eq3-2} and their respective solutions. We consider the general setting of \cref{lemma2-1}, i.e., let {$B, D \in \mathbb{S}^n$ be symmetric matrices} and let $\lambda>0$ be given. Let us then define $M := B D + \frac{1}{\lambda}(I-D)$.

\begin{lemma}
\label{lemma3-7} Suppose that $DM$ is positive semidefinite and $M$ is invertible. Let us set $g = D\Fnor{z}$ and assume $\iprod{g}{DMg} \leq 0$. Then, it holds that $ M \Fnor{z} = \lambda\Fnor{z}$. 
\end{lemma}

\begin{proof} Since $DM$ is positive semidefinite and symmetric, we have:
\begin{align*}
\iprod{g}{DMg} \leq 0 \; \iff \; \|[DM]^\frac12 g\| = 0 \; \iff \;  D M g = 0 \; \iff \; M^\top D g = 0. 
\end{align*}

Consequently, this implies $Dg = 0$ and $0 = \Fnor{z}^\top D g = \|g\|^2$. Hence, using \eqref{eq3-3}, we can infer $\lambda M\Fnor{z} = (\lambda B -I) D \Fnor{z} + \Fnor{z} = \Fnor{z}$. \end{proof}

In the following result, we summarize some of the core properties of \cref{algo1}. In particular, \cref{algo1} will always terminate after a maximum of $m = \mathrm{rank}(D) = \mathrm{dim}~\mathcal R(D)$ iterations. If the generalized derivative $D$ has low rank and satisfies $m \ll n$, this allows to significantly reduce the complexity of computing a semismooth Newton step. 
\begin{lemma}
\label{lemma3-8} Suppose that \cref{algo1} is run with $S = D M$ and $g = D\Fnor{z}$ and define $m = \mathrm{dim}~\mathcal R(D)$. Then, it holds that:
\begin{itemize} 
\item[\rmn{(i)}] \cref{algo1} stops after at most $m\leq n$ iterations with $\|q \| \leq \Delta$.
\item[\rmn{(ii)}] In addition, assume that $D M$ is positive semidefinite and $M$ is invertible and that we have $\|M^{-1}\Fnor{z}\| \leq \Delta$. \cref{algo1} then returns $q$ with $\|D (M q + \Fnor{z})\| \leq \epsilon$.
\end{itemize}
\end{lemma}
A proof of \cref{lemma3-8} is presented in \cref{sec:app:pf-38}.


\section{Global Convergence Analysis}
\label{sec2}

In this section, we investigate the global convergence properties of \cref{algo2}. We start with listing our (additional) assumptions on the functions $f$ and $\vp$. 
\begin{assumption}\label{assum2-1} We consider the conditions:
\begin{enumerate}[label=\textup{\textrm{(A.\arabic*)}},topsep=0pt,itemsep=0ex,partopsep=0ex,leftmargin=8ex]
  \item \label{A1} The gradient $\nabla f$ is Lipschitz continuous on $\dom{\vp}$ with modulus $L$.
  \item \label{A2} The objective function $\psi$ is lower bounded on $\dom{\partial\varphi}$.
\end{enumerate}
\end{assumption}
%

We continue with several assumptions on the choice of the parameters and constants utilized in \cref{algo2}. 
\begin{assumption}
\label{assum2-2}  We assume: 
\begin{enumerate}[label=\textup{\textrm{(B.\arabic*)}},topsep=0pt,itemsep=0ex,partopsep=0ex,leftmargin=8ex]
\item \label{B1} 
The parameters $\tau$ and $\nu_k$ satisfy the conditions: 
\[ \nu \in [0,1),\quad 0\leq\nu_k\leq\nu, \quad \text{and} \quad (L^2\lambda^2+2)\tau< 2(1-\nu). \]
\item \label{B2} It holds that $\sum_{k=0}^\infty ({1 + \|B_k\|})^{-1} = \infty$. 
\item \label{B3} There is $\kappa_B > 0$ such that $\|B_k\| \leq \kappa_B$ for all $k \in \N$.
\end{enumerate}
\end{assumption}
Condition \ref{B2} holds, e.g., if the matrices $\{B_k\}$ satisfy $\|B_k\| \leq c_{B_1} + c_{B_2} k$ for all $k$ and some $c_{B_1}, c_{B_2} > 0$. Hence, \ref{B2} is generally weaker than \ref{B3}. Let $\{j_i\}$ be an increasing sequence enumerating the indices of the accepted iterates in step 8 of \cref{algo2}. We then define the set of all successful iterations as $\mathcal S := \{j_i : i \geq 0\} = \{k \in \N : \rho_k \geq \eta_1\}$. We will also use the notations 
\be \label{eq:muk-xk} {\mu}_k := \nu_k \chi(z_k) \min\{\Delta_k,\lambda\chi(z_k)\}^{-1} \quad \text{and} \quad x_k := \prox{z_k}, \ee
where $\{z_k\}$ and $\{\Delta_k\}$ denote the sequences generated by \cref{algo2}. 

We first study the descent properties of the merit function $\mer$.

\begin{lemma}
\label{lemma2-5}
Suppose that \ref{A1} is satisfied and let $z, e \in \Rn$, $\alpha \in (0,1]$, {$\nu \in [0,1)$, and $\tau \in (0,1-\nu)$} be given. Setting $d:=-\Fnor{z}$ and $x := \prox{z}$, it follows
\begin{align*}
\mer(z+\alpha\lambda (d+e))-\mer(z) & \leq -\frac{\tau\lambda\alpha}{2}\|\Fnor{z}\|^2 - \frac{\nu}{\lambda\alpha}\|p_\alpha - x\|^2 + \frac{\tau\lambda\alpha}{2} \|e\|^2 \\ & \hspace{4ex} + \left[L\tau\lambda + 1 - \tau\right] \alpha \|e\|\|d+e\| + C(\alpha) \|p_\alpha-x\|^2,
\end{align*}
where $p_\alpha := \prox{z+\alpha\lambda (d+e)}$ and $C(\alpha) := (\frac{L^2\lambda^2+2}{2\lambda} \cdot \tau - \frac{1-\nu}{\lambda})\frac{1}{\alpha} + L \tau + \frac{L}{2}-\frac{\tau}{2\lambda}$.
\end{lemma}
\begin{proof} Applying the Lipschitz continuity of $\nabla f$ on $\dom{\vp}$ and $\nabla \env{z+\alpha \lambda (d+e)} = \frac{1}{\lambda}(z+\alpha \lambda  (d+e) - p_\alpha) \in \partial \vp(p_\alpha)$, see \eqref{eq1-3}, we have
\begin{align} \label{eq:here-esti1} 
\psi(p_\alpha) - \psi(x) & \leq  \iprod{\nabla f(x)}{p_\alpha - x}  + \frac{L}{2} \|p_\alpha - x\|^2 + \frac{1}{\lambda}\iprod{z+\alpha \lambda  (d+e) - p_\alpha}{p_\alpha - x} \nonumber \\ & = \iprod{\Fnor{z}+\alpha (d+e)}{p_\alpha - x} + \left[\frac{L}{2}-\frac{1}{\lambda}\right] \|p_\alpha - x\|^2 \nonumber \\ & = (1-\alpha) \iprod{\Fnor{z} - e}{p_\alpha - x} + \iprod{e}{p_\alpha - x} + \left[\frac{L}{2}-\frac{1}{\lambda}\right] \|p_\alpha - x\|^2. \end{align}
Moreover, it holds that
\begingroup
\allowdisplaybreaks
\begin{align*} \frac{\lambda}{2} \|\Fnor{z+\alpha\lambda (d+e)}\|^2 & = \frac{\lambda}{2} \|\nabla f(p_\alpha) + \lambda^{-1}(z+\alpha\lambda(d+e)-p_\alpha)\|^2 \\ & = \frac{\lambda}{2} \|(1-\alpha)\Fnor{z} + \nabla f(p_\alpha) - \nabla f(x) - {\lambda^{-1}}(p_\alpha - x) + \alpha e\|^2 \\ & = \frac{\lambda(1-\alpha)^2}{2} \|\Fnor{z}\|^2 + (1-\alpha)\iprod{\Fnor{z}}{\lambda [\nabla f(p_\alpha)- \nabla f(x)]-[p_\alpha-x]}  \\ & \hspace{4ex} + (1-\alpha)\alpha \iprod{\Fnor{z}}{\lambda e} +  \iprod{\nabla f(p_\alpha)-\nabla f(x)}{x-p_\alpha +  \lambda \alpha e}  \\ & \hspace{4ex} + \frac{\lambda}{2}\|\nabla f(p_\alpha)-\nabla f(x)\|^2 + \frac{\lambda\alpha^2}{2} \|e\|^2  - \alpha \iprod{e}{p_\alpha - x}  + \frac{1}{2\lambda} \|p_\alpha - x\|^2  \\ & \hspace{0ex} \leq \frac{\lambda(1-\alpha)^2}{2} \|\Fnor{z}\|^2 - (1-\alpha)\iprod{\Fnor{z} - e}{p_\alpha - x} - \iprod{e}{p_\alpha - x} \\ & \hspace{4ex}  + (1-\alpha)\alpha \iprod{\Fnor{z}}{\lambda e} + \iprod{(1-\alpha)\Fnor{z}+\alpha e}{\lambda [\nabla f(p_\alpha)-\nabla f(x)]}  \\ & \hspace{4ex} + \frac{\lambda\alpha^2}{2} \|e\|^2 + \left[ \frac{L^2\lambda}{2} + L + \frac{1}{2\lambda}\right] \|p_\alpha - x\|^2 \\ & \hspace{0ex} = \frac{\lambda}{2}(1-\alpha) \|\Fnor{z}\|^2 + (1-\alpha) \iprod{d+e}{p_\alpha - x} - \iprod{e}{p_\alpha - x} + \frac{\lambda\alpha}{2} \|e\|^2 \\ & \hspace{4ex} - \frac{\lambda\alpha}{2}(1-\alpha) \|d+e\|^2  + \left[ \frac{L^2\lambda}{2} + L + \frac{1}{2\lambda}\right] \|p_\alpha - x\|^2 \\ & \hspace{4ex} + \iprod{(1-\alpha)\Fnor{z}+\alpha e}{\lambda [\nabla f(p_\alpha)-\nabla f(x)]}, 
\end{align*}
\endgroup
where we used the Lipschitz continuity of $\nabla f$ and $2 \iprod{\Fnor{z}}{e} =  \|\Fnor{z}\|^2 + \|e\|^2 - \|d+e\|^2$ in the last step. Next, applying Young's inequality, it follows
\begin{align*} \vert\iprod{d+e}{\nabla f(p_\alpha) - \nabla f(x)}\vert & \leq \|d+e\| \cdot L \|p_\alpha - x\|  \leq \frac{\alpha}{2}  \|d+e\|^2 + \frac{L^2}{2\alpha} \|p_\alpha - x\|^2 \end{align*}
%
and we have $\|p_\alpha - x\|\leq \|z+\alpha\lambda (d+e)-z\| = \lambda\alpha \|d+e\|$, $\vert\iprod{e}{p_\alpha - x}\vert \leq \lambda\alpha \|e\|\|d+e\|$, and $\vert\iprod{e}{\nabla f(p_\alpha) - \nabla f(x)}\vert \leq  L\lambda\alpha \|e\| \|d+e\|$. Hence, upon writing $(1-\alpha)\Fnor{z}+\alpha e = -(1-\alpha)(d+e)+e$, this implies
\begin{align} \label{eq:here-esti2}
 \frac{\lambda}{2} \|\Fnor{z+\alpha\lambda (d+e)}\|^2 & \leq \frac{\lambda}{2}(1-\alpha) \|\Fnor{z}\|^2 + (1-\alpha) \iprod{d+e}{p_\alpha - x}  \\ & \hspace{4ex} - \iprod{e}{p_\alpha - x} + \frac{\lambda\alpha}{2} \|e\|^2  + \left[ \frac{L^2\lambda}{2\alpha} + L + \frac{1}{2\lambda}\right] \|p_\alpha - x\|^2 + L\lambda^2\alpha  \|e\|\|d+e\|. \nonumber
 \end{align} 
Furthermore, by the firm nonexpansiveness of the proximity operator, \eqref{eq:prox-nonexp}, we have
\begin{equation} \label{eq:here-esti3} -\iprod{d+e}{p_\alpha - x} = - \frac{1}{\lambda\alpha}\iprod{z + \lambda\alpha (d+e) - z}{p_\alpha - x} \leq - \frac{1}{\lambda\alpha} \|p_\alpha - x\|^2. \end{equation}
Combining the estimates \eqref{eq:here-esti1}, \eqref{eq:here-esti2}, and \eqref{eq:here-esti3}, we finally obtain
\begin{align*}
\mer(z+\alpha\lambda (d+e))-\mer(z) + \frac{\tau\lambda\alpha}{2}\|\Fnor{z}\|^2& \\ &  \hspace{-40ex} =\psi(p_\alpha) - \psi(x) +\frac{\tau\lambda}{2}\|\Fnor{z+\alpha\lambda (d+e)}\|^2-\frac{\tau\lambda}{2}(1-\alpha)\|\Fnor{z}\|^2 \\ & \hspace{-40ex} \leq (1-\tau)\langle e,p_\alpha-x\rangle +  L\tau\lambda^2\alpha \|e\|\|d+e\| +\frac{\tau\lambda\alpha}{2}\|e\|^2 \\ & \hspace{-36ex} + \left[ \frac{L^2\lambda \tau}{2\alpha} + L \tau + \frac{\tau}{2\lambda} + \frac{L}{2}-\frac{1}{\lambda} - \frac{(1-\tau)(1-\alpha)}{\lambda\alpha} \right]\|p_\alpha -x\|^2 \\ & \hspace{-40ex} \leq  - \left[\frac{\nu}{\lambda\alpha} - C(\alpha)\right] \|p_\alpha-x\|^2 + \left[L\tau\lambda + 1-\tau \right]\lambda\alpha \|e\|\|d+e\| + \frac{\tau\lambda\alpha}{2} \|e\|^2,
\end{align*}
as desired. \end{proof}

\begin{remark} \label{remark:lwb-step-size} In \cref{lemma2-5}, suppose that condition \ref{B1} is additionally satisfied. Then, the factor $\frac{L^2\lambda^2+2}{2\lambda} \cdot \tau - \frac{1-\nu}{\lambda}$ in $C(\alpha)$ is negative. In particular, defining
\be \label{eq:lwb-step-size} \bar\alpha = \bar\alpha(L,\lambda,\tau,\nu):= \begin{cases} 1 & \text{if $(1-2L\lambda)\tau \geq L\lambda$,} \\ \min\left\{1,\frac{2(1-\nu)-(L^2\lambda^2+2)\tau}{(1+2\tau)L\lambda+\tau}\right\} & \text{otherwise,} \end{cases} \ee
and for all $\alpha \in [0,\bar\alpha]$, we can infer
\begin{align*} \mer(z+\alpha\lambda (d+e))-\mer(z) & \leq -\frac{\tau\lambda\alpha}{2} \chi(z)^2 - \frac{\nu}{\lambda\alpha} \|p_\alpha - x\|^2 + \frac{\tau\lambda\alpha}{2} \|e\|^2 \\ & \hspace{4ex} +\left[L\tau\lambda  + 1 - \tau \right] \lambda\alpha \|e\| \|d+e\|  \end{align*}
\end{remark}

We now derive a first result that gives insight on the occurrence of successful steps. Specifically, we show that the trial point $z_k + s_k$ is always successful and accepted if the trust region radius $\Delta_k$ is sufficiently small. 

\begin{lemma} \label{lemma2-6}
Let $\{z_k\}$, $\{s_k\}$, and $\{\Delta_k\}$ be generated by \cref{algo2} and suppose that the assumptions \ref{A1} and \ref{B1} are satisfied. Then, there exists a constant $\hat c > 0$ that only depends on $\bar \alpha$, $\eta_2$, $\tau$, $\lambda$, $L$ such that every iteration $k$ with 
\be \label{eq:bd-del} \Delta_k \leq \hat c_k \chi(z_k), \quad \hat c_k := \frac{\hat c}{1+\lambda\|B_k\|}, \ee
is very successful, i.e., the condition \eqref{eq:bd-del} implies $\rho_k \geq \eta_2$ and $k \in \mathcal S$.  
\end{lemma}
\begin{proof} First, by \cref{prop2-5}, we have $\|D\|\leq 1$ for all $D \in \partial \prox{z}$ and $z \in \Rn$. Hence, using \eqref{eq3-3}, we obtain
\begin{equation} 
\|I-\lambda M_k\|=\|(I-\lambda B_k)D_k\|\leq 1+\lambda\|B_k\|. \label{eq4-16}
\end{equation}
We now define the constants $c_1 ={\lambda\bar{\alpha}}/{(1+\bar{\alpha})}$, $c_2 = (1-\eta_2)\lambda$, 
\[  c_3 = \frac{\tau\lambda(1-\eta_2)}{8(1+\lambda)(L\tau\lambda +1-\tau)}, \quad c_4 = \frac{\lambda\sqrt{1-\eta_2}}{2}, \quad \text{and} \quad c_5=\frac{(1-\eta_2)\lambda}{2\eta_2}, \]
where $\bar{\alpha}$ was introduced in \cref{remark:lwb-step-size}. Furthermore, let us set 
\be \label{eq:hat-ck} \hat c := \min\left\{\min_{1\leq i \leq 5} c_i,1\right\} \quad \text{and} \quad \hat c_k := \frac{\hat c}{1+\lambda\|B_k\|}. \ee
We consider an iteration $k$ with $\Delta_k \leq \hat c_k \chi(z_k)$. Let us define $s_k^d := -\lambda d = \lambda  \Fnor{z_k}$. We have $\|\bar{s}_k\| \geq\|s^d_k\|-\|(I-\lambda B_k)D_k\bar q_k\|$ and by the algorithmic construction, it holds that $\|\bar q_k\|\leq\Delta_k$. Then, applying $\|s_k^d\| = \lambda\chi(z_k)$, \eqref{eq4-16}, and \eqref{eq:hat-ck}, it follows
\begin{align*}
\|\bar{s}_k\| & \geq\|s_k^d\| - (1+\lambda\|B_k\|)\Delta_k \geq \left[ \frac{\lambda}{\hat c_k}- (1+\lambda\|B_k\|) \right] \Delta_k \geq \left[ \frac{\lambda}{c_1} - 1 \right] (1+\lambda\|B_k\|) \Delta_k \geq \frac{1}{\bar\alpha} \Delta_k, \\ \|\bar s_k\| & \geq \lambda\chi(z_k) - (1+\lambda\|B_k\|)\hat c_k \chi(z_k) \geq [\lambda - c_2] \chi(z_k)  = \lambda\eta_2\chi(z_k),
\end{align*}
and $\frac{\Delta_k}{\|\bar{s}_k\|}\leq \min\{\bar{\alpha},  \frac{\Delta_k}{\lambda\eta_2\chi(z_k)} \} \leq1$. This yields $s_k=\frac{\Delta_k}{\|\bar{s}_k\|}\lambda (-\Fnor{z_k}+(\frac{1}{\lambda} I-M_k)\bar q_k)$. Setting  $\alpha={\Delta_k}/\|\bar{s}_k\| \leq\bar{\alpha}$, $e=(\frac{1}{\lambda} I-M_k)\bar q_k$, and using $\lambda(d + e)=\bar s_k$ (cf. \eqref{eq:def-sk-bar-sk}), \cref{lemma2-5} and \cref{remark:lwb-step-size} imply
\begin{align} \label{eq4-17} \nonumber
\mer(z_k+s_k)-\mer(z_k) & \leq -\frac{\tau\lambda\alpha}{2}\cdot \chi(z_k)^2 - \frac{\nu}{\lambda\alpha} \|\prox{z_k+s_k} - x_k\|^2  \\ & \hspace{9ex} + \left[L \tau\lambda + 1-\tau\right] \lambda\alpha \|e\| \|\bar s_k\| + \frac{\tau\lambda\alpha}{2} \|e\|^2.
\end{align}
Next, we provide additional estimates for $\|e\|$ and $\|\bar s_k\|$. Invoking \eqref{eq4-16}, it holds that $\|e\| = \|(\frac{1}{\lambda} I-B_k)D_k\bar q_k\|\leq \frac{1}{\lambda}(1+\lambda\|B_k\|) \Delta_k\leq \frac{1}{\lambda} \hat c\chi(z_k)$ and
\be \label{eq:upp-sk}
\|\bar s_k\| \leq \|s_k^d\| + \lambda\|e\| \leq ( \lambda + \hat c) \chi(z_k). 
\ee
Thus, we obtain
\begin{align*} 
 \left[ L\tau\lambda  + 1-\tau\right]\lambda\|e\|\|\bar s_k\| \leq \left[ L\tau\lambda  + 1-\tau\right]c_3(\lambda+1)\chi(z_k)^2  = \frac{\tau\lambda}{2} \frac{1-\eta_2}{4} \cdot \chi(z_k)^2 
\end{align*} 
and $\|e\|^2 \leq \frac{c_4^2}{\lambda^2} \chi(z_k)^2 = \frac14{(1-\eta_2)} \chi(z_k)^2$. Using these estimates in \eqref{eq4-17}, we have
\begin{align*} \mer(z_k+s_k)-\mer(z_k) \leq -\frac{\tau\lambda\alpha}{2}\frac{1+\eta_2}{2} \chi(z_k)^2 - \frac{\nu}{\lambda\alpha} \|\prox{z_k+s_k} - x_k\|^2. \end{align*}
Due to $\Delta_k \leq c_2 \chi(z_k) \leq \lambda\chi(z_k)$ and $\nu_k \leq \nu$, we can further infer 
\[ -\frac{\nu}{\lambda\alpha} \leq -\frac{\nu_k}{\lambda}\cdot\frac{\eta_2 \lambda\chi(z_k)}{\Delta_k} = - \frac{\eta_2 \cdot \nu_k\chi(z_k)}{\min\{\Delta_k,\lambda\chi(z_k)\}} = -\eta_2\mu_k. \] 
%
%
By \eqref{eq:upp-sk}, we also have $\|\bar s_k\|\leq (\lambda + c_5) \chi(z_k) = (2\eta_2)^{-1}(1+\eta_2)\lambda \chi(z_k)$ which establishes $-\alpha = - \Delta_k/\|\bar s_k\| \leq - 2\eta_2 \Delta_k((1+\eta_2)\lambda\chi(z_k))^{-1}$. Combining the last steps, it follows
\begin{align*} \mer(z_k+s_k)-\mer(z_k) & \leq-\frac{\eta_2\tau}{2} \Delta_k\chi(z_k) - {\eta_2{\mu}_k} \|\prox{z_k+s_k}-x_k\|^2 \leq - \eta_2 \mathrm{pred}(z_k,s_k,\Delta_k,\nu_k). \end{align*}
Consequently, we have $\rho_k\geq\eta_2$ which concludes the proof. \end{proof}

\begin{remark}
\label{remark:bound_chi_delta} \cref{lemma2-6} implies that the condition $\Delta_k \leq \hat c_k \chi(z_k)$ can not hold for unsuccessful iterations $k \notin \mathcal S$. Let us further consider an iterate $z_k$ with $k\in\mathcal{S}$. If $k-1\in\mathcal{S}$, then we obtain $\Delta_k\geq \Delta_{\min}$. Otherwise, if $k-1\notin\mathcal{S}$, we have $B_k = B_{k-1}$ and we can infer $\Delta_{k-1}>\hat c_{k-1}\chi(z_{k-1})=\hat c_k\chi(z_k)$. Thus, by the algorithmic construction, it follows $\Delta_{k}\geq\gamma_0 \Delta_{k-1}\geq \gamma_0\hat c_k\chi(z_k)$ and for all $k \in \mathcal S$, we have 
\[ \Delta_k\geq\min\{\gamma_0\hat c_k\chi(z_k),\Delta_{\min}\}.  \]
%
\end{remark}

\cref{lemma2-6} allows us to prove that a sequence generated by \cref{algo2} contains infinitely many successful iterates. 
\begin{lemma}
\label{lemma2-7}
Under the assumptions \ref{A1} and \ref{B1}, \cref{algo2} either terminates after finitely many iterations or it generates infinitely many successful steps.
\end{lemma}
\begin{proof} Conversely, assume that \cref{algo2} generates an infinite sequence $\{z_k\}$ with only finitely many successful steps. Let $k'\in\mathbb{N}$ denote the last successful iteration, i.e., it holds that $z_k = z_{k'}$ for all $k\geq k'$. The update rule for the trust region radius then implies $\Delta_k\rightarrow 0$. Furthermore, since the matrices $B_k$ are no longer updated for all $k \geq k^\prime$, we obtain $\hat c_k = \hat c_{k^\prime}$ for all $k \geq k^\prime$ where the parameter $\hat c_k$ is defined in \eqref{eq:bd-del}. Due to $\Delta_k \to 0$ there then exists $\ell > k^\prime$ with $\Delta_{\ell} \leq \hat c_{k^\prime} \chi(z_{k^\prime}) =  \hat c_{\ell}\chi(z_{\ell})$ which, by \cref{lemma2-6}, yields ${\ell} \in \mathcal S$. However, this is a contradiction to our assumption. 
\end{proof}

We now present our main global convergence result. 
\begin{thm}
\label{thm:global_conv}
Let the conditions \ref{A1}--\ref{A2}, \ref{B1}--\ref{B2}, and $\Delta_{\min} > 0$ hold and assume that \cref{algo2} does not terminate after finitely many steps. Then, we have
\be \label{eq:liminf-conv} \lim_{k\rightarrow\infty}\chi(z_k)=0 \quad \text{and} \quad \sum_{k=0}^\infty {\mu}_k \|x_{k+1}-x_k\|^2 < \infty. \ee
\end{thm}
\begin{proof} \cref{lemma2-7} implies $\vert{\mathcal S}\vert=\infty$ and it holds that
\[ {\sum}_{k\in\mathcal{S}}\mer(z_{k})-\mer(z_{k+1})\geq{\sum}_{k\in\mathcal{S}}\eta_1 \mathrm{pred}(z_k,s_k,\Delta_k,\nu_k). \]
Furthermore, by assumption \ref{A2}, the merit function $\mer$ is bounded from below on $\Rn$. Hence, since the sequence $\{\mer(z_k)\}$ is non-increasing, there exists $\zeta \in \R$ with $\lim_{k \to \infty} \mer(z_k) = \zeta$ and it follows
\be \label{eq:crit-sum} \sum_{k \in \mathcal S} {\mu}_k \|x_{k+1}-x_k\|^2 < \infty \quad \text{and} \quad \sum_{k\in\mathcal{S}}\chi(z_k)\min\{\lambda,{\Delta_k},\lambda\chi(z_k)\}<\infty. \ee
We now proceed as in \cite{Pow84} and first show the weaker condition $\liminf_{k \to \infty} \chi(z_k) = 0$. {Assume that there is $\epsilon \in (0,\min\{1,\lambda^{-1}\Delta_{\min}\})$ with $\chi(z_k) \geq \epsilon$ for all $k$ sufficiently large}. Then, utilizing \cref{remark:bound_chi_delta} and \eqref{eq:crit-sum}, we obtain
\be \label{eq:p-sum-lw} \infty > \sum_{k \in \mathcal S} {\epsilon} \min\{\lambda,\Delta_k,\lambda\epsilon\} \geq {\epsilon} \sum_{k \in \mathcal S} \min\left\{\lambda\epsilon,\gamma_0\hat c_k\chi(z_k)\right\} \geq {\epsilon^2} \sum_{k \in \mathcal S} \frac{\min\{\lambda,\hat c \gamma_0\}}{1+\lambda\|B_k\|}. \ee 
Due to $\vert\mathcal S\vert = \infty$, we have $(1+\lambda\|B_k\|)^{-1} \to 0$ as $\mathcal S \ni k \to \infty$ and hence, it follows $\sum_{k \in \mathcal S} (1+\lambda\|B_k\|)^{-1} \leq \infty$. Let us recall the notation $\mathcal S = \{j_i : i \geq 0\}$ and let us consider an arbitrary unsuccessful iteration $ j_i+1 \leq k \leq j_{i+1}-1$. Using the trust region update mechanism and \cref{lemma2-6}, this yields
\[\frac{\hat c\epsilon}{1+\lambda\|B_k\|} \leq \hat c_k \chi(z_k) \leq \Delta_k \leq \gamma_1^{k-j_i} \Delta_{j_i}. \]
Summing this estimate, we obtain
\[ \sum_{k=j_i+1}^{j_{i+1}-1} \frac{1}{1+\lambda\|B_k\|} \leq \frac{\Delta_{j_i}}{\hat c \epsilon} \sum_{k=j_i+1}^{j_{i+1}-1} \gamma_1^{k-j_i} \leq \frac{\Delta_{j_i}}{\hat c \epsilon} \left[ \sum_{k=0}^{\infty} \gamma_1^{k} - 1 \right] = \frac{\gamma_1}{(1-\gamma_1)\hat c \epsilon} \cdot \Delta_{j_i}. \]
Summing this expression once more for all $i$, we can infer $\sum_{k \notin \mathcal S}(1+\lambda\|B_k\|^{-1}) \leq \frac{\gamma_1}{(1-\gamma_1)\hat c\epsilon} \sum_{k \in \mathcal S} \Delta_k$. As before, the summability condition in \eqref{eq:p-sum-lw} implies $\Delta_k \to 0$ as $\mathcal S \ni k \to \infty$ and hence, combining the last steps, it follows 
\[ \sum_{k=0}^\infty \frac{1}{1+\lambda\|B_k\|} < \infty \quad \implies \quad \sum_{k=0}^\infty \frac{1}{1+\|B_k\|}<\infty. \]
However, this contradicts \ref{B2} and thus, we can establish $\liminf_{k \to \infty} \chi(z_k) = 0$. 
%
Let us now suppose that $\{\chi(z_k)\}$ does not converge. Then there exist $\delta > 0$ and infinite, increasing sequences $\{t_i\}_i,\{\ell_i\}_i \subset \mathcal S$ such that $t_{i+1} \geq \ell_i > t_i$ for all $i \in \N$ and
\[ \chi(z_{t_i}) \geq 2\delta, \quad \chi(z_{\ell_i}) < \delta, \quad \text{and} \quad \chi(z_k) \geq \delta \quad k = t_i+1,...,\ell_i-1. \]
Let us define $\mathcal K := \{k \in \mathcal S: t_i \leq k < \ell_i\}$. Due to $\mathcal K \subseteq \mathcal S$, the second condition in \eqref{eq:crit-sum} implies 
\[ \infty > \sum_{k \in \mathcal K} \chi(z_k) \min\{\lambda,\Delta_k,\lambda\chi(z_k)\} \geq \sum_{k \in \mathcal K} \delta \min\{\lambda,\Delta_k,\lambda\delta\}. \]
Consequently, we have $\Delta_k \to 0$ as $\mathcal K \ni k \to \infty$ and utilizing $\rho_k \geq \eta_1$ for all $k \in \mathcal K$, it follows $\Delta_k \leq \frac{2}{\tau\eta_1\delta}[\mer(z_k)-\mer(z_{k+1})]$ for all $k \in \mathcal K$ sufficiently large. Next, the Lipschitz continuity of $\nabla f$ and $\oprox$ yield
\begin{align} \nonumber \vert\chi(w)-\chi(z)\vert & \leq \|\Fnor{w}-\Fnor{z}\| \\ \nonumber & \hspace{-12ex} \leq \|\nabla f(\prox{w})-\nabla f(\prox{z})\| + {\lambda^{-1}}\|w-\prox{w} - (z-\prox{z})\| \\ & \hspace{-12ex}  \leq L \|\prox{w}-\prox{z}\| + 2\lambda^{-1}\|w-z\| \leq (L+2\lambda^{-1}) \|w-z\| \label{eq:fnor-lipschitz} \end{align}
for all $w, z \in \Rn$. Thus, setting $L_F := L+2\lambda^{-1}$, combining the previous estimates, and using \eqref{eq:def-sk-bar-sk}, we obtain
\begin{align*} \delta & < \vert\chi(z_{t_i}) - \chi(z_{\ell_i})\vert \leq L_F {\|z_{\ell_i}-z_{t_i}\|} \\ & \leq L_F \sum_{k = t_i, \, k \in \mathcal K}^{\ell_i-1} \|s_k\| \leq L_F \sum_{k = t_i, \, k \in \mathcal K}^{\ell_i-1} \Delta_k \leq {\frac{2L_F}{\tau\eta_1\delta}}[\mer(z_{t_i}) - \mer(z_{\ell_i})]. \end{align*}
Due to $\mer(z_k) \to \zeta$, the right hand side of the last inequality has to converge to zero which is a contradiction.
\end{proof}

Related results for classical trust-region methods have been shown in \cite{Pow84,Yua85,Pow10,GraYuaYua15} and \cite[Section 8.4]{ConGouToi00}. Here, based on the Lipschitz assumption \ref{A1} and similar to \cite[Theorem 4.9]{ulbrich2001nonmonotone}, the special definition of our predicted reduction term $\mathrm{pred}(z_k,s_k,\Delta_k,\nu_k)$ allows us to obtain stronger results and convergence of the whole sequence $\{\chi(z_k)\}$ -- even if the matrices $\{B_k\}$ are not bounded.

\begin{remark} \label{remark:csq-stat} \cref{thm:global_conv} has an interesting consequence concerning stationarity properties of the sequence $\{x_k\}$. Let us consider the index set
\[ \mathcal T := \{k \in \N: x_{k+1} \neq x_k \} \]
and let us suppose $\vert\mathcal T\vert < \infty$. Then, there exist $\bar x$ and $k^\prime \in \N$ such that $x_k = \bar x$ for all $k \geq k^\prime$ and we have 
\[ \Fnor{z_k} = \nabla f(x_k) + {\lambda}^{-1} (z_k - x_k) = \nabla f(\bar x) -{\lambda}^{-1} \bar x + {\lambda}^{-1} z_k \quad \forall~k \geq k^\prime. \] 
Applying \cref{thm:global_conv}, we can infer $z_k \to \bar x - \lambda  \nabla f(\bar x) =: \bar z$ and $\Fnor{\bar z} = 0$ and hence, $\bar x$ is a stationary point of problem \eqref{eq1-1}. This observation can also be used algorithmically. In particular, when there is a successful iteration $k \in \mathcal S$ with $x_{k+1} = x_k$ or $x_{k+1} \approx x_k$, we can check if the natural stationarity criterion $\Fnat{x_{k+1}} = x_{k+1} - \prox{z_{k+1}-\lambda \Fnor{z_{k+1}}} \approx 0$ is satisfied to terminate earlier.
\end{remark}

Based on the proof of \cref{thm:global_conv} and using \cref{remark:bound_chi_delta} and the stronger assumption \ref{B3}, we can directly establish square summability of $\chi(z_k)$. 

\begin{corollary} \label{cor:sec04:conv-str} Let the conditions \ref{A1}--\ref{A2}, \ref{B1}, \ref{B3}, and $\Delta_{\min} > 0$ be satisfied and suppose that \cref{algo2} does not terminate after finitely many steps. Then, it holds that
\[ \sum_{k \in \mathcal S}^\infty \chi(z_k)^2 < \infty \quad \text{and} \quad \sum_{k=0}^\infty {\mu}_k \|x_{k+1}-x_k\|^2 < \infty. \]
\end{corollary}

 

\section{Convergence Properties Under the Kurdyka- {\L}ojasiewicz Inequality} \label{section:local-convergence}
\label{sec:KL_conv}
 
We now investigate additional convergence properties of \cref{algo2} utilizing the Kurdyka-{\L}ojasiewicz (KL) inequality. Our discussion is motivated by the general KL-framework provided in \cite{AttBol09,AttBolSva13,BolSabTeb14} and by the results for the forward- backward quasi-Newton method in \cite[section 3.2]{SteThePat17} and \cite{TheStePat18}. Specifically, we show how these techniques can be transferred to our nonsmooth trust-region method allowing us to establish convergence of the whole sequence $\{z_k\}$ and local rates of convergence.

\subsection{Definitions and Assumptions} In the following, we introduce the class of so-called {desingularizing functions} that will be used in the definition of the KL-property. By $\mathfrak{S}_\eta$ we denote the class of all continuous and concave functions $\varrho : [0,\eta) \to \R_+$ such that 
 \[ \varrho \in C^1((0,\eta)), \quad \varrho(0) = 0, \quad \text{and} \quad \varrho^\prime(x) > 0, \quad \forall~x \in (0,\eta).\]
We also consider the subclass of {{\L}ojasiewicz functions}
 $\mathfrak L := \{ \varrho : \R_+ \to \R_+ : \exists~c > 0, \, \theta \in [0,1): \varrho(x) = c x^{1-\theta} \}$.  
Obviously, it holds that $\mathfrak L \subset \mathfrak S_\eta$ for all $\eta > 0$. For the ease of exposition, we define the KL-property for functions of the form $\psi = f + \varphi$, where $f : \Rn \to \R$ is continuously differentiable and $\vp : \Rn \to \Rex$ is convex, lsc, and proper.
\begin{defn} Let $\psi = f + \vp$ be a proper, lsc function as specified above. We say that $\psi$ has the Kurdyka-{\L}ojasiewicz property at $\bar x \in \dom{\partial\vp}$ if there exist $\eta \in (0,\infty]$, a neighborhood $U$ of $\bar x$, and a function $\varrho \in \mathfrak{S}_\eta$ such that for all $x \in U \cap \{x \in \Rn : 0 < \psi(x)-\psi(\bar x) < \eta\}$ the KL-inequality holds, i.e.,
\be \label{eq:kl-ineq} \varrho^\prime(\psi(x) - \psi(\bar x)) \cdot \dist(0,\partial \psi(x)) \geq 1. \ee
If the mapping $\varrho$ can be chosen from $\mathfrak L$ and satisfies $\varrho(x) = c x^{1-\theta}$ for some $c > 0$ and $\theta \in [0,1)$, then we say that $\psi$ has the KL-property at $\bar x$ with exponent $\theta$. 
\end{defn}
{The KL-property is a powerful concept which is applicable to a vast range of problems. In particular, the KL-inequality holds for the ubiquitous class of subanalytic or semialgebraic functions, \cite{lojasiewicz1963,lojasiewicz1993,kurdyka1998,BolDanLew06}.} Let $\{z_k\}$, $\{x_k\}$, and $\{s_k\}$ be generated by \cref{algo2} and let us introduce the set of accumulation points 
\begin{align*} {\mathfrak A} := \{z \in \Rn: \exists~\text{a subsequence} \,\{k_\ell\}_\ell \, \text{with} \, {z}_{k_\ell} \to z, \ell \to \infty \}. \end{align*} 
Notice that $\mathfrak A$ is closed-valued by definition. Next, we formulate our main assumptions of this section.
\begin{assumption} \label{ass:kl} We consider the conditions: \vspace{.5ex}
\begin{enumerate}[label=\textup{\textrm{(C.\arabic*)}},topsep=0pt,itemsep=0ex,partopsep=0ex,leftmargin=8ex]
\item \label{C1} The merit function $\mer$ satisfies the following KL-type property on $\mathfrak A$: for all $\bar z \in \mathfrak A$ there exist $\eta \in (0,\infty]$, a neighborhood $V$ of $\bar z$, and a function $\varrho \in \mathfrak S_\eta$ such that we have \vspace{1ex}
\end{enumerate}
\be \label{eq:kl-mer} \varrho^\prime(\mer(z) - \mer(\bar z)) \cdot \chi(z) \geq 1 \quad \forall~z \in V \cap \{z \in \Rn: 0 < \mer(z)-\mer(\bar z) < \eta\}. \ee
\begin{enumerate}[label=\textup{\textrm{(C.\arabic*)}},topsep=0pt,itemsep=0ex,partopsep=0ex,leftmargin=8ex] \addtocounter{enumi}{1}
\item \label{C2} The sequence $\{z_k\}$ is bounded.
\item \label{C3} Setting $n_{\mathcal S}(k) := \vert \mathcal S \cap \{0,1,...,k-1\}\vert$, we assume that $\nu$ and $\nu_k$ satisfy
\[ \nu > 0 \quad \text{and} \quad \nu_k\geq \min\{\nu,a_{n_{\cS}(k)}^2\|\prox{z_k+s_k}-x_k\|^{2p}\}, \quad \forall~k,  \]
where $p>0$ is a constant and $\{a_k\} \subset \R_{++}$ is given with $\sum_{k=0}^{\infty}a_k^{-{1}/{p}}<\infty$.
\end{enumerate}
\end{assumption}

Condition \ref{C2} is a typical and ubiquitous prerequisite appearing in the application of the KL-framework, see, e.g., \cite{AttBol09,AttBolRedSou10,AttBolSva13,BolSabTeb14}. Assumption \ref{C3} specifies the growth behavior of the parameters $\{\nu_k\}$. The lower bound in \ref{C3} is motivated by the convergence analysis in \cref{thm:kl-convergence} and in \cref{sec:loc-super-conv} and it will allow us to establish full global-local results. Assumption \ref{C1} can be interpreted as a specialized variant of the usual KL-condition. 
In the following, we will show that this condition is satisfied at a point $\bar z \in \mathfrak A$ when $\psi$ has the KL-property at $\prox{\bar z}$ with exponent $\theta$. In this case, the desingularizing function $\varrho$ in \ref{C1} can also be chosen from the class $\mathfrak L$ with exponent $\max\{\theta,\frac12\}$. Hence, similar to the forward-backward envelope, the merit function $\mer$ can preserve the KL-properties of the original objective function. We refer to \cite{YuLiPon19,SteThePat17,TheStePat18} for comparison and further details.

\begin{lemma} \label{lemma:calc-kl-mer} Suppose that $\psi$ satisfies the KL-property at a stationary point $\bar x = \prox{\bar z} \in \crit(\psi)$ with exponent $\theta$. Then, the merit function $\mer$ satisfies the KL-type property defined in \ref{C1} at $\bar z$ with exponent $\max\{\theta,\frac12\}$. 
\end{lemma}

\begin{proof} Let $\varrho(x) = c x^{1-\theta}$ be the associated desingularizing function. By definition, there exist $\epsilon, \eta > 0$ such that
\begin{align}
    \label{eq:lemma5-3-1}
    [c(1-\theta)]^{\frac{1}{\theta}} \dist(0,\partial\psi(x))^\frac{1}{\theta} \geq \psi(x) - \psi(\bar x) 
\end{align} 
for all $x$ with $\|x - \bar x\| \leq \epsilon$ and $\psi(x) < \psi(\bar x) + \eta$. (The inequality is obviously true in the case $\psi(x) \leq \psi(\bar x)$). Applying \cite[Lemma 4.1.6]{milzarek2016numerical} (with $\Lambda = \frac{1}{\lambda}I$), we obtain
\be \label{eq:dist-nor} \dist(0,\partial\psi(\prox{z})) \leq \|\Fnor{z}\|=\chi(z) \quad \forall~z. \ee
We now choose $\delta \leq  \epsilon$ sufficiently small such that $\chi(z) \leq 1$ for all $z \in B_\delta(\bar z)$. Then, for all $z \in B_\delta(\bar z)$ with $\mer(\bar z) = \psi(\bar x) < \mer(z) < \mer(\bar z) + \eta$, we have $\|\prox{z} - \bar x\| \leq \epsilon$ and $\psi(\prox{z}) < \psi(\bar x) + \eta$ and thus, \eqref{eq:lemma5-3-1} is applicable. Setting $c_\theta := [c(1-\theta)]^{1/\theta}$ and combining \eqref{eq:lemma5-3-1} and \eqref{eq:dist-nor}, we can infer $c_\theta \chi(z)^{{1/\theta}} \geq \psi(\prox{z}) - \psi(\bar x)$ and
\[ \left[ c_\theta + \frac{\tau\lambda}{2} \right]\chi(z)^{\min\left\{\frac{1}{\theta},2\right\}} \geq c_\theta \chi(z)^{\frac{1}{\theta}} + \frac{\tau\lambda}{2} \|\Fnor{z}\|^{2} \geq \mer(z) - \mer(\bar z). \]
This shows that $\mer$ satisfies the KL-type inequality with exponent $\max\{\theta,\frac12\}$. \end{proof}

\subsection{Convergence Results}

In the following, we state properties of the set of accumulation points $\mathfrak A$. 

\begin{lemma} \label{lemma:A-prop} Let $\{z_k\}$ be generated by \cref{algo2} and suppose that the conditions \ref{A1}--\ref{A2}, \ref{B1}--\ref{B2}, \ref{C2}, and $\Delta_{\min} > 0$ are satisfied. Then, we have: 
\begin{itemize}
\item[{\rm(i)}] The set ${\mathfrak A}$ is compact and nonempty and satisfies $\prox{\mathfrak A} \subseteq {\crit}(\psi)$.
\item[{\rm(ii)}] We have $\lim_{k \to \infty} \dist(z_k,{\mathfrak A}) = 0$.
\item[{\rm(iii)}] The functions $\psi \circ \proxs$ and $\mer$ are constant and finite on $\mathfrak A$.
\end{itemize}
\end{lemma}


\cref{lemma:A-prop} can be seen as an analogue of \cite[Lemma 5]{BolSabTeb14}. The derivation of \cref{lemma:A-prop} is based on \cref{thm:global_conv} and on the continuity of $\psi \circ \proxs$ and $\mer$ and closely follows the original proof in \cite{BolSabTeb14}. We will omit details here.

The properties (i)--(iii) in \cref{lemma:A-prop} allow deriving a uniformized version of the KL-type inequality \eqref{eq:kl-mer}. Specifically, suppose that the KL-type property holds on $\mathfrak A$. 
Then, there are $\delta,\eta > 0$ and $\varrho \in \mathfrak S_\eta$ 
such that for all $\bar z \in \mathfrak A$ and $z  \in V_{\delta,\eta} := \{ z \in \Rn : \dist(z,\mathfrak A) < \delta \} \cap \{z \in \Rn : 0 < \mer(z)-\mer(\bar z) < \eta \}$, we have:
\[ \varrho^\prime(\mer(z) - \mer(\bar z)) \cdot \chi(z) \geq 1, \]
see, e.g.,  \cite[Lemma 6]{BolSabTeb14}. Next, we present the main result of this section. 
 
\begin{thm} \label{thm:kl-convergence}
  Let $\{z_k\}$ be generated by \cref{algo2} and assume that the conditions \ref{A1}, \ref{B1}--\ref{B2}, \ref{C1}--\ref{C3}, and $\Delta_{\min} > 0$ are satisfied. Then, it holds that:
  \begin{itemize}
  \item[\rmn{(i)}] The sequence $\{z_k\}$ converges to some $\bar z$ with $\bar x = \prox{\bar z} \in \crit(\psi)$. 
  \item[\rmn{(ii)}] {In addition, suppose that the KL-type inequality \cref{eq:kl-mer} holds at the limit point $\bar z$ of $\{z_k\}$ for some $\varrho \in \mathfrak L$ with exponent $\theta \in [0,1)$ and let \ref{B3} be satisfied.} 
  \begin{itemize}
  \item[$\bullet$] If $\theta \in [0,\frac12)$, the sequence $\{z_k\}$ converges within a finite number of steps. 
  \item[$\bullet$] If $\theta = \frac12$, $\{\chi(z_k)\}_{\mathcal S}$ and $\{z_k\}_{\mathcal{S}}$ converge r-linearly to $0$ and $\bar z$, respectively. 
  \item[$\bullet$] If $\theta \in (\frac12,1)$, then the sequences $\{\chi(z_k)\}_{\mathcal S}$ and $\{z_k\}_{\mathcal S}$ converge with the following rates for $\mathcal S \ni k \to \infty$:
   \[ \chi(z_k) =  O(n_{\mathcal S}(k)^{-\frac{1-\theta}{2\theta-1}}) \quad \text{and} \quad \|z_{k}-\bar z\| =  O(n_{\mathcal S}(k)^{-\frac{(1-\theta)^2}{(1+p)\theta(2\theta-1)}}).  \]
  Moreover, it holds $\liminf_{k\to\infty}n_\cS(k)^{1+q}\chi(z_k)=0$ for all $q \in (0,\frac{1-\theta}{2\theta-1})$.
  \end{itemize}
  \end{itemize}
  \end{thm}
  
  \begin{proof} Let $\delta, \eta > 0$ be the constants appearing in the definition of the uniformized KL-type inequality. The proof of \cref{thm:global_conv} implies that the sequence $\{\mer(z_k)\}$ is non-increasing and converges to some $\zeta \in \R$. Suppose now there exists $\ell \in \N$ with $\mer(z_\ell) = \zeta$. Then, we necessarily have $\mer(z_k) = \zeta$ for all $k \geq \ell$. By \cref{lemma2-7} there is $\mathcal S \ni k \geq \ell$ with $0 = \mer(z_k) - \mer(z_{k+1}) \geq 0.5\eta_1\tau\chi(z_k)\min\{\lambda,\Delta_k,\lambda\chi(z_k)\}$ which implies $\chi(z_k) = 0$ and hence, the algorithm would terminate after finitely many steps. As a consequence and using \cref{lemma:A-prop} (ii) and (iii), there exist $\hat z \in \mathfrak A$ with $\mer(\hat z) = \zeta$ and $k^\prime$ such that $\mer(z_k) > \mer(\hat z)$ and $z_k \in V_{\delta,\eta}$ for all $k \geq k^\prime$. Recall $\mathcal{S}=\{j_i: i \geq 0\}$ and let $\ell^{\prime}$ be the smallest index such that $j_{\ell^{\prime}}\geq k'$. To prove convergence of $\{z_k\}$, it suffices to prove convergence of $\{z_{j_i}\}$. Since the uniformized KL-type inequality is applicable for the latter iterates, we have
  \be \label{eq:kl-prf} \varrho^\prime(\mer(z_{j_i}) - \mer(\hat z)) \cdot \chi(z_{j_i}) \geq 1 \quad \forall~i \geq \ell^\prime. \ee
  Next, due to the concavity of $\varrho$ and setting $\delta_k := \varrho(\mer(z_k) - \mer(\hat z))$, we obtain
  \begin{align*} \delta_{j_i} - \delta_{j_{i+1}} & \geq \varrho^\prime(\mer(z_{j_i}) - \mer(\hat z)) [\mer(z_{j_i}) - \mer(z_{j_{i+1}})] \\ & \geq \frac{\mer(z_{j_i}) - \mer(z_{j_{i+1}})}{\chi(z_{j_i})}=\frac{\mer(z_{j_i})-\mer(z_{j_i+1})}{\chi(z_{j_i})} \end{align*}
  for all $i \geq \ell^\prime$, where we have used the fact $z_{j_i+1}=z_{j_{i+1}}$ for $j_i\in\mathcal{S}$. Summing this expression for $i \geq \ell^\prime$ {and by the definition of $\mathrm{pred}$, (cf. \eqref{eq:pred})}, it follows
  \be \delta_{j_{\ell^\prime}} =  \sum_{i=\ell^\prime}^{\infty} \delta_{j_i} - \delta_{j_{i+1}} \geq \eta_1\sum_{i=\ell^{\prime}}^{\infty} \left[ \frac{\tau}{2} \min\{\lambda,\Delta_{j_i},\lambda\chi(z_{j_i})\} + \frac{\nu_{j_i}\|x_{j_{i+1}}-x_{j_i}\|^2}{\min\{\Delta_{j_i},\lambda\chi(z_{j_i})\}}  \right] \label{eq:esti-nice} \ee
  where we applied $x_{j_i+1} = x_{j_{i+1}}$ for $j_i \in \mathcal S$, the continuity of $\varrho$, and $\varrho(0) = 0$. Due to Young's inequality, we have
  \begin{align*}  \sqrt{2\tau\nu_{j_i}}\|x_{j_{i+1}}-x_{j_i}\| & = \frac{ \sqrt{2\nu_{j_i}}\|x_{j_{i+1}}-x_{j_i}\|}{\sqrt{\min\{\Delta_{j_i},\lambda\chi(z_{j_i})\}}} \cdot\sqrt{\tau\min\{\Delta_{j_i},\lambda\chi(z_{j_i})\}} \\ & \leq \frac{\nu_{j_i}\|x_{j_{i+1}}-x_{j_i}\|^2}{\min\{\Delta_{j_i},\lambda\chi(z_{j_i})\}} + \frac{\tau}{2} \min\{\Delta_{j_i},\lambda\chi(z_{j_i})\}. \end{align*}
  Thus, utilizing $\lim_{k \to \infty} \chi(z_k) = 0$, we can infer 
  \be \label{eq:esti-xdel} \sqrt{2\tau}\eta_1 \sum_{i=\ell^{\prime\prime}}^{\infty}\sqrt{\nu_{j_i}}  \|x_{j_{i+1}}-x_{j_i}\| \leq \delta_{j_{\ell^{\prime\prime}}} \ee
  for some $\ell^{\prime\prime} \geq \ell^{\prime}$. Notice that the assumptions \ref{B1} and \ref{C3} imply $\nu_k = \nu$ if $a_{n_{\cS}(k)}^2 \|\prox{z_k+s_k}-x_k\|^{2p} \geq \nu$. Hence, introducing the index sets $\mathcal{I}_1 := \{i\geq \ell^{\prime\prime}: \nu_{j_i}=\nu\}$ and $\mathcal{I}_2:=\{i\geq \ell^{\prime\prime}: i\notin\mathcal{I}_1\}$, we obtain
  \[ {\sum}_{i\in\mathcal{I}_1}\|x_{j_{i+1}}-x_{j_i}\|<\infty \quad \text{and} \quad {\sum}_{i\in\mathcal{I}_2}a_{i}\|x_{j_{i+1}}-x_{j_i}\|^{1+p}<\infty, \]  
  where we used $n_{\mathcal S}(j_i) = \vert\{j_0,...,j_{i-1}\}\vert = i$. Next, in order to estimate the second term, we apply the reverse H\"older inequality
  \[ \infty> {\sum}_{i\in\mathcal{I}_2}a_i\|x_{j_{i+1}}-x_{j_i}\|^{1+p}\geq \left[{\sum}_{i\in\mathcal{I}_2} a_i^{-{1}/{p}}\right]^{-p}              
  \left[{\sum}_{i\in\mathcal{I}_2}\|x_{j_{i+1}}-x_{j_i}\|\right]^{1+p} \]
  and consequently, due to $\bar a := \sum_{i} a_i^{-{1}/{p}}<\infty$, we can infer $\sum_{i}\|x_{j_{i+1}}-x_{j_i}\|< \infty$.
  This implies that $\{x_k\}_{\mathcal{S}}$ and $\{x_k\}$ are Cauchy sequences that converge to the (same) limit $\bar x$. By \cref{thm:global_conv}, the convergence of $\{x_k\}$ also yields convergence of $\{z_k\}$:
  \[ z_k = x_k - \lambda \nabla f(x_k) + \lambda  \Fnor{z_k} \to \bar x - \lambda  \nabla f(\bar x) =: \bar z, \quad k \to \infty, \]
  and thus, we have $\bar x = \prox{\bar z} \in \crit(\psi)$. We now continue with the proof of the second part. Using $\varrho(x) = c x^{1-\theta}$ and the KL-type inequality \eqref{eq:kl-prf}, it holds that
  \begin{align} \delta_{j_i} = \varrho(\mer(z_{j_i}) - \mer(\bar z)) & = c \left[ \frac{c(1-\theta)}{\varrho^\prime(\mer(z_{j_i})-\mer(\bar z))} \right]^{\frac{1-\theta}{\theta}} \leq c(c(1-\theta))^{\frac{1-\theta}{\theta}} \chi(z_{j_i})^{\frac{1-\theta}{\theta}} \label{eq:del-kl-chi} \end{align}
  for all $i \geq \ell^{\prime\prime}$. Due to $\lim_{k\rightarrow\infty}\chi(z_{k})=0$, we can assume that $\ell^{\prime\prime}$ is chosen sufficiently large to guarantee $\chi(z_{j_i})< \Delta_{\min}/\lambda$ for all $i \geq \ell^{\prime\prime}$. Then, utilizing \cref{remark:bound_chi_delta}, \ref{B3}, and \eqref{eq:esti-nice}, we have
  \begin{align*}
  2 \delta_{j_\ell} \geq \eta_1  \tau {\sum}_{i=\ell}^{\infty} \min\{\lambda,\Delta_{j_i},\lambda\chi(z_{j_i})\} \geq \eta_1  \tau {\sum}_{i=\ell}^{\infty} \min\{\lambda,\gamma_0\hat{c}(1+\lambda\kappa_B)^{-1}\}\chi(z_{j_i})
  \end{align*}
  for every $\ell\geq\ell^{\prime\prime}$. Setting $C_\theta := 2c(c(1-\theta))^{\frac{1-\theta}{\theta}} [\eta_1\tau \min\{\lambda,\gamma_0\hat{c}(1+\lambda\kappa_B)^{-1}\}]^{-1}$ and $\Gamma_\ell := \sum_{i=\ell}^{\infty}\chi(z_{j_i})$, this yields
  \be \label{eq:tt-kl} {C_\theta}(\Gamma_{\ell}-\Gamma_{\ell+1})^\frac{1-\theta}{\theta} \geq \Gamma_{\ell}. \ee   
  Next, we discuss different cases depending on the KL-exponent $\theta \in [0,1)$. 
  
  In the case $\theta \in (0,\frac12)$, it follows $\frac{1}{\theta}-2 > 0$ and $C_\theta \chi(z_{j_\ell})^{\frac{1}{\theta}-2} \geq 1$. Due to $\chi(z_k) \to 0$, this condition can only hold for finitely many $\ell$ and hence, $\{z_k\}$ has to converge in finitely many steps. Notice that the KL-inequality \eqref{eq:kl-prf} reduces to $c \chi(z_{j_i}) \geq 1$ in the case $\theta = 0$ which again implies finite step convergence. In the case $\theta = \frac12$, if $C_\theta \leq 1$, we obtain $\Gamma_{\ell+1} \leq 0$ and finite step convergence. Otherwise, we have $\Gamma_{\ell+1} \leq [ 1 - {1}/{C_\theta}] \Gamma_\ell$ which proves q-linear convergence of $\{\Gamma_\ell\}$ and hence, $\{\chi(z_{j_i})\}$ converges r-linearly to zero.  
  Combining \eqref{eq:esti-xdel}, \eqref{eq:del-kl-chi}, and the previous estimates and using $\vert x+y \vert^{1+p} \leq 2^p(\vert x\vert^{1+p} + \vert y\vert^{1+p})$, $x,y \in \R$, we obtain
  \begin{align*} \frac{c^2}{2\sqrt{2\tau}\eta_1} \chi(z_{j_\ell}) & \geq \frac{\delta_{j_\ell}}{\sqrt{2\tau}\eta_1} \geq \sum_{i=\ell}^{\infty}\sqrt{\nu_{j_i}}  \|x_{j_{i+1}}-x_{j_i}\| \\ & \geq \sqrt{\nu}\,{\sum}_{i \in \mathcal I_1, i \geq \ell} \|x_{j_{i+1}}-x_{j_i}\| + \bar a^{-p} \left[ {\sum}_{i \in \mathcal I_2, i \geq \ell} \|x_{j_{i+1}}-x_{j_i}\| \right]^{1+p} \\ & \hspace{-0ex} \geq \frac{\min\{\sqrt{\nu},\bar a^{-p}\}}{2^p} \left [ {\sum}_{i \geq \ell} \|x_{j_{i+1}}-x_{j_i}\| \right]^{1+p}  \geq \frac{\min\{\sqrt{\nu},\bar a^{-p}\}}{2^p} \|x_{j_\ell} - \bar x\|^{1+p} \end{align*}
  for all $\ell \geq \ell^{\prime\prime}$ sufficiently large such that $ {\sum}_{i \in \mathcal I_1, i \geq \ell} \|x_{j_{i+1}}-x_{j_i}\|\leq 1$. This proves r-linear convergence of $\{x_{j_i}\}$ to $\bar{x}$. Moreover, we note that
  \begin{align*}  \|z_k-\bar{z}\| & =\|x_k-\bar{x}-\lambda (\nabla f(x_k)-\nabla f(\bar{x}))+\lambda \Fnor{z_k}\|  \leq (1+\lambda L)\|x_k-\bar{x}\|+\lambda\chi(z_k). \end{align*}
  Therefore, the r-linear convergence of $\{x_{j_i}\}$ and $\{\chi(z_{j_i})\}$ imply r-linear convergence of the sequence $\{z_{j_i}\}$. Finally, let us consider the case $\theta \in (\frac12,1)$. Rearranging the terms in \eqref{eq:tt-kl} and using the monotonicity of the function $t \mapsto t^\frac{\theta}{1-\theta}$, we have 
  \[ \Gamma_\ell^{-\frac{\theta}{1-\theta}} (\Gamma_\ell - \Gamma_{\ell+1}) \geq C_\theta^{-\frac{\theta}{1-\theta}}. \]
  Hence, since the mapping $\hat \varrho^\prime(t) := t^{-\frac{\theta}{1-\theta}}$ is monotonically decreasing, we can infer
  \begin{align*} \hat \varrho^\prime(C_\theta) \leq \hat \varrho^\prime(\Gamma_\ell) (\Gamma_\ell - \Gamma_{\ell+1}) & \leq \int_{\Gamma_{\ell+1}}^{\Gamma_\ell} \hat \varrho^\prime(t)\; \mathrm{d}t  = \hat \varrho(\Gamma_\ell) - \hat \varrho(\Gamma_{\ell+1}) = \frac{1-\theta}{1-2\theta}\left[ \Gamma_\ell^{\frac{1-2\theta}{1-\theta}} - \Gamma_{\ell+1}^{\frac{1-2\theta}{1-\theta}}\right].  \end{align*}
  Summing this estimate for $\ell^{\prime\prime} \leq \ell \leq m-1$ and noticing $1-2\theta < 0$, this yields
  \[ \frac{2\theta-1}{1-\theta}  \hat \varrho^\prime(C_\theta) (m-\ell^{\prime\prime}) \leq \Gamma_m^{\frac{1-2\theta}{1-\theta}} - \Gamma_{\ell^{\prime\prime}}^{\frac{1-2\theta}{1-\theta}} \leq \Gamma_m^{\frac{1-2\theta}{1-\theta}}  \quad \text{and} \quad \Gamma_m = O(m^{-\frac{1-\theta}{2\theta-1}}), \]
  %
  as $m \to \infty$. Mimicking the earlier discussion for the case $\theta = \frac12$, we can establish 
  \[\|x_{j_m}-\bar x\| =  O(m^{-\frac{(1-\theta)^2}{(1+p)\theta(2\theta-1)}}) \quad \text{and} \quad \|z_{j_m}-\bar z\| =  O(m^{-\frac{(1-\theta)^2}{(1+p)\theta(2\theta-1)}}), \]
  as desired. To show the last conclusion, assume $\liminf_{k\to\infty} \chi(z_k)n_\cS(k)^{1+q}>0$. Then, it follows $\chi(z_{j_i}) n_\cS(j_i)^{1+q} = \chi(z_{j_i}) i^{1+q} \geq \epsilon$ for some $\epsilon > 0$ and all $i$ sufficiently large. This yields $\Gamma_m = \sum_{i=m}^\infty \chi(z_{j_i}) \geq \sum_{i=m}^\infty \frac{\veps}{i^{1+q}} = \Omega(m^{-q})$ for $m \to \infty$. However, due to $q<\frac{1-\theta}{2\theta-1}$, this contradicts $\Gamma_m=O({m^{-\frac{1-\theta}{2\theta-1}}})$.
  \end{proof}
  

\begin{remark} Let us note that the standard KL analysis framework, \cite{AttBolSva13,BolSabTeb14}, is not directly applicable in our situation since the acceptance criterion $\rho_k \geq \eta_1$ only yields descent of the merit function $\mer$ in terms of $\|x_{k+1}-x_k\|$ and $\chi(z_k)$ and not in terms of $\|z_{k+1}-z_k\|$. Thus, the so-called sufficient decrease condition -- used in \cite{AttBolSva13,BolSabTeb14} -- is not necessarily satisfied. Here, we prove convergence of $\{x_k\}$ and $\{z_k\}$ under the KL inequality for potentially unbounded $\{B_k\}$ without requiring any additional stringent assumptions on $\|z_{k+1}-z_k\|$. We further note that applicability of the KL-theory for classical trust region methods is typically based on a subtle connection between the radius $\Delta_k$, the trust region step $s_k$, and the criticality measure and on a strict descent condition, see, e.g., \cite{AbsMahAnd05,NolRon13}. Such a condition is also not required in our analysis. \end{remark}



\section{Local Superlinear Convergence}
\label{sec:loc-super-conv}
In this section, we present our local convergence theory. Specifically, we will establish that a sequence generated by \cref{algo2} converges q-superlinearly to a solution $\bar z$ of the nonsmooth equation \eqref{normal_eq} under suitable local conditions. We first list our required assumptions for proving fast local convergence:
\begin{assumption} \label{assum3-1} Let $\{z_k\}$ be generated by \cref{algo2} and suppose that $\bar z \in \mathfrak A$ is an accumulation point of the sequence $\{z_k\}$. We then consider: 

\begin{enumerate}[label=\textup{\textrm{(D.\arabic*)}},topsep=0pt,itemsep=0ex,partopsep=0ex,leftmargin=8ex]
\item \label{D1} $\{z_k\}$ converges to $\bar{z}$ and there is $q > 0$ with $\liminf_{k\to\infty}\chi(z_k)n_{\cS}(k)^{1+q}=0$.
\item \label{D2} The function $f$ is twice continuously differentiable on $\dom{\vp}$.
\item \label{D3} The mapping $\proxs$ is semismooth at $\bar z$.
\item \label{D4} There exist $\kappa_M > 0$ and $K \in \N$ such that for all $k \geq K$ the matrix $D_k M_{k}$ is positive semidefinite and $M_k$ is invertible with $\|M_k^{-1}\| \leq \kappa_M$. 
\item \label{D7} The matrices $\{B_k\}$ satisfy the following Dennis-Mor\'{e}-type condition
\[  \lim_{k \to \infty} \frac{\|[B_k - \nabla^2 f(x_k)](x_k - \bar x)\|}{\|z_k - \bar z\|} = 0 \quad \text{where} \quad \bar x = \prox{\bar z}. \]
\item \label{D5}The error threshold $\epsilon_k$ used in the CG method satisfies $\epsilon_k \leq \mathcal E(\Fnor{z_k})$ where $\mathcal E : \Rn \to \R_+$ is continuous with $\mathcal E(0) = 0 $ and $\mathcal E(h) = o(\|h\|)$ as $h \to 0$.
%

\item \label{D6} The parameter $\nu_k$ satisfies $\nu_k\leq a_{n_{\cS}(k)}^2\|\prox{z_k+s_k}-x_k\|^{2p}$, where $\{a_k\} \subset \R_{++}$ is a positive sequence with $\limsup_{k\rightarrow\infty} {a_{k}^{1/p}}/{k^{1+\tilde q}}=0$ for every $\tilde q>0$ and $\limsup_{k\rightarrow\infty} {a_{k+1}}/{a_k}=\tilde{\kappa}_a<\infty$.
\end{enumerate}
\end{assumption}

We continue with several remarks. 
Convergence of the sequence $\{z_k\}$ and the condition on $\chi(z^k)$ -- as stated in assumption \ref{D1} -- can be ensured using the KL-framework, see \cref{thm:kl-convergence}. \ref{D2} and \ref{D3} are standard assumptions that guarantee sufficient smoothness and semismoothness of the normal map. Assumption \ref{D5} implies that the linear systems \eqref{eq3-2} are solved sufficiently accurate as $k$ increases and that the tolerance parameter $\epsilon_k$ is connected to the residual $\Fnor{z_k}$. In \ref{D4}, we formulate our main curvature and boundedness assumptions which are related to (but weaker than) the CD-regularity of the normal map $\oFnor$ at $\bar z$. Notice that positive semidefiniteness of $D_kM_k$ can be ensured if $B_k$ is positive semidefinite. Under certain structural properties on $\varphi$, it is possible to connect positive semidefiniteness of the matrices $D_kM_k$ to a second order optimality condition for $\min_x \, \psi(x)$. We refer to \cref{sec:sop} for further details. We will discuss assumption \ref{D4} in more detail in the next subsections and analyze its connection to second-order conditions for problem \eqref{eq1-1}. 
A similar variant of the Dennis-Mor\'{e} condition in \ref{D7} has been utilized recently in \cite{mannelhybrid}. \ref{D7} is obviously satisfied when we work with the full Hessian $B_k = \nabla^2 f(x_k)$.
In \cref{sec:qnm}, we verify that this form of the Dennis-Mor\'{e} condition does hold under suitable assumptions when the Hessian approximations $B_k$ are built via BFGS updates.

Finally, we propose the following choice for $\{\nu_k\}$ and $\{a_k\}$:  
\be \label{eq:nuk-ak} \nu_k :=\min\{\nu,a^2_{n_\cS(k)}\|\prox{z_k+s_k}-x_k\|^{2p}\} \quad \text{and} \quad a_k := k^{p}\log^{2p}(k). \ee
Then, the assumptions \ref{B1}, \ref{C3}, and \ref{D6} are all satisfied for every $p>0$ (with $\tilde \kappa_a = 1$). Next, we state the main convergence result of this section.

\begin{thm} \label{thm:main-local-conv}
Suppose that the conditions \ref{A1}--\ref{A2},  \ref{B1}, \ref{B3}, and \ref{D1}--\ref{D6} are satisfied. Furthermore, let us assume that \cref{algo2} does not terminate after finitely many steps and we choose $\Delta_{\min}>0$. Then, we have:
\begin{itemize}
\item Every trust region step is eventually successful, i.e., there exists $\bar K \in \N$ such that $k \in \cS$ for all $k \geq \bar K$ and the sequence $\{z_k\}$ converges q-superlinearly to $\bar z$. 
\item In addition, if $\proxs$ is $\beta$-order semismooth at $\bar z$ for $\beta \in (0,1]$, the function $\mathcal E$ in \ref{D5} satisfies $\mathcal E(h) = O(\|h\|^{1+\beta})$, $h \to 0$, and if we choose $B_k = \nabla^2 f(x_k)$ and $\nabla^2 f$ is Lipschitz continuous near $\bar x$, then the rate is of order $1+\beta$. 
\end{itemize}
\end{thm}

The proof of \cref{thm:main-local-conv} is split into several parts. To show acceptance of the trust region steps, we investigate the behavior of the reduction ratio $\rho_k$ and the descent properties of $\mer$ along the directions returned by the Steihaug-CG method. This analysis is carried out in \cref{sec:des-mer}. In particular, we will see that the matrix $D M$ essentially captures the curvature of the nonsmooth mapping $\psi \circ \proxs$. In \cref{sec:d4-alt}, we derive an equivalent formulation of  \ref{D4} that is used as a key tool in \cref{sec:des-mer}. Finally, in \cref{sec:proof-loc}, we combine our observations, assumptions, and results and provide the full proof. 

\subsection{An Alternative Formulation of Condition~\ref{D4}} \label{sec:d4-alt}

We first present an alternative characterization of the positive semidefiniteness and invertibility condition mentioned in assumption~\ref{D4}. 

\begin{lemma} \label{lemma:pos-pos} Let $B, D \in \mathbb{S}^n$ and $\lambda > 0$ be given. Let us set $M := B D + \frac{1}{\lambda} (I-D)$. If $D M$ is positive semidefinite and $M$ is invertible, then it holds that 
\be \label{eq:pos-pos} \iprod{h}{DM h} \geq \sigma \|Dh\|^2 \quad \forall~h \in \Rn, \ee
where $\sigma := \|DM^{-1}\|^{-1}$. Conversely, assume that \eqref{eq:pos-pos} holds for some $\sigma > 0$. Then, $DM$ is positive semidefinite and $M$ is invertible with $\|M^{-1}\| \leq 1+ \sigma^{-1}\|I-\lambda B\|$. 
\end{lemma}

\begin{proof} We start with the proof of the first part. Let $h \in \Rn$ be arbitrary and let us set $s := (I - \sigma M^{-1} D)h$. Using the symmetry of $D, D M$ and $M^{-\top}DM = D$, we have
\begin{align*}
\iprod{s}{DMs} & = \iprod{h - \sigma M^{-1} Dh}{D(M-\sigma D)h}  \\ &= \iprod{h}{D M h} - \sigma \|Dh\|^2- \sigma\iprod{M^{-1}Dh}{DM h} + \sigma^2 \iprod{M^{-1} Dh}{D^2 h} \\ & \leq \iprod{h}{DMh} - 2\sigma \|Dh\|^2 + \sigma^2\|DM^{-1}\| \|Dh\|^2.
\end{align*}
Hence, the positive semidefiniteness of $DM$ implies \eqref{eq:pos-pos}. Now, assume that condition \eqref{eq:pos-pos} holds for some $\sigma > 0$. Then, $D M$ is obviously positive semidefinite. Next, let $y$ be arbitrary with $My = 0$. Then, due to \eqref{eq:pos-pos}, we have $Dy = 0$ and we can infer
\[ M y = 0 \quad \iff \quad (B-\tfrac{1}{\lambda}I)Dy + \tfrac{1}{\lambda} y = 0 \quad \implies \quad y = 0. \]
Thus, $M$ is invertible. Let $y$ and $r$ now be given with $My = -r$. By \cref{lemma2-1}, we know that $\bar y = y - \lambda(My + r) = (I-\lambda B)Dy - \lambda r$ satisfies $M\bar y = - r$. Consequently, it follows $y = \bar y$ and applying \eqref{eq:pos-pos}, we obtain
\[ \|Dy\|^2 \leq {\sigma^{-1}} \iprod{y}{DM y} = - {\sigma^{-1}} \iprod{Dy}{r} \leq {\sigma^{-1}} \|Dy\| \|r\|. \]
This implies $\|Dy\| \leq \sigma^{-1}\|r\|$ and hence, it holds that $\|y\| = \|(I-\lambda B)Dy - \lambda r\| \leq \|I-\lambda B\|\|Dy\| + \lambda \|r\| \leq (1+ \sigma^{-1}\|I-\lambda B\|) \|r\|$
which concludes the proof. \end{proof}

Hence, under condition \ref{D4}, we see that there exists $\bar \sigma > 0$ such that
\be \label{eq:dm-lwb} \iprod{h}{D_k M_k h} \geq \bar \sigma \|D_k h\|^2 \quad \forall~h \in \Rn, \ee
for all $k \geq K$. Furthermore, if assumption \ref{B3} and \eqref{eq:dm-lwb} are satisfied, then \cref{lemma:pos-pos} implies that \ref{D4} has to hold.  

\subsection{Local Descent Properties of the Merit Function $H_\tau$} \label{sec:des-mer}

We start with an expansion result that allows to interpret the matrices $DM$ as curvature terms of $\psi \circ \proxs$. As mentioned in \cref{sec:tr-glob}, this provides a rigorous explanation of the specific choice of our trust region model. 

\begin{proposition}
\label{prop3-2}
Let $\bar{z}$ be a zero of $\oFnor$ and assume that $f$ is twice continuously differentiable in a neighborhood of $\prox{\bar z}$ and $\proxs$ is semismooth at $\bar z$. Then, we have
$$ \psi(\prox{z})-\psi(\prox{\bar{z}})=\frac{1}{2}\langle z-\bar{z}, DM(z-\bar{z})\rangle+o(\|z-\bar{z}\|^2) \quad \text{as} \quad z \to \bar z,$$
for all $D \in \partial \prox{z}$ and $M = \nabla^2 f(\prox{z})D +\frac{1}{\lambda} (I-D) \in \mathcal M^\lambda(z)$. 
\end{proposition}
\begin{proof} By assumption, $\envs$ is semismooth differentiable at $\bar{z}$, i.e, it is continuously differentiable around $\bar z$ with semismooth gradient at $\bar z$. By \citep[Proposition 7.4.10]{facchinei2007finite}, this implies
\begin{align}  \env{z}-\env{\bar{z}} & =\langle\nabla\env{\bar{z}},z-\bar{z}\rangle +\frac{1}{2\lambda}\langle z-\bar{z},(I-D)(z-\bar{z})\rangle+o(\|z-\bar{z}\|^2) \label{eq5-1} \end{align}
for every $D \in \partial \prox{z}$ and for $z \to \bar z$. Moreover, setting $x = \prox{z}$ and $\bar{x}=\prox{\bar{z}}$, we obtain
\begin{align}\label{eq5-2}
f(x)-f(\bar{x}) & 
=\langle \nabla f(\bar{x}),x-\bar{x} \rangle+\frac{1}{2}\langle x-\bar{x},\nabla^2 f(x)(x-\bar{x})\rangle + o(\|x-\bar{x}\|^2)
\end{align}
as $x \to \bar x$. Due to the Lipschitz continuity of $\proxs$, it follows $\|x-\bar x\| = O(\|z-\bar z\|)$ and $o(\|x-\bar{x}\|^2)=o(\|z-\bar{z}\|^2)$. Using \eqref{eq5-1}, $\nabla \env{\bar z} = \frac{1}{\lambda}(\bar z - \bar x)$, and $\frac12 \|\bar z -\bar x\|^2 + \frac12 \|z - \bar z\|^2 = \frac12 \|z-x\|^2 + \frac{1}{2}\|x-\bar x\|^2 + \iprod{\bar z - z}{\bar z - \bar x} + \iprod{z-x}{x-\bar x}$, we have
\begin{align}\label{eq5-3}
\varphi(x)-\varphi(\bar{x}) & =\env{z}-\env{\bar{z}}-\frac{1}{2\lambda}\|z-x\|^2+\frac{1}{2\lambda}\|\bar{z}-\bar{x}\|^2\notag\\
& \hspace{-0ex}=\frac{1}{\lambda}\langle \bar{z}-\bar{x},z-\bar{z}\rangle+\frac{1}{2\lambda}\|z-\bar{z}\|_{I-D}^2-\frac{1}{2\lambda}\|z-x\|^2 +\frac{1}{2\lambda}\|\bar{z}-\bar{x}\|^2+o(\|z-\bar{z}\|^2)\notag\\
& \hspace{-0ex}=\frac{1}{\lambda}\langle z-x,x-\bar{x}\rangle-\frac{1}{2\lambda}\|z-\bar{z}\|^2_{D}+\frac{1}{2\lambda}\|x-\bar{x}\|^2+o(\|z-\bar{z}\|^2)
\end{align}
for all $D \in \partial\prox{z}$ and $z \to \bar z$. Notice that the semismoothness of $\proxs$ implies $x - \bar x = D(z-\bar z) + o(\|z-\bar z\|)$ for all $D \in \partial \prox{z}$ and $z \to \bar z$. Combining this with \eqref{eq5-2}, \eqref{eq5-3}, and $\nabla f(\bar x) = \frac{1}{\lambda}(\bar x - \bar z)$, it follows
\begin{align}\label{eq5-4}
\psi(x)-\psi(\bar{x})\notag &=\frac{1}{\lambda}\langle z-\bar{z}-(x - \bar{x}),x-\bar{x} \rangle+\frac{1}{2}\langle x-\bar{x},\nabla^2 f(x)(x-\bar{x})\rangle\notag \\ & \hspace{4ex} -\frac{1}{2\lambda}\|z-\bar{z}\|^2_{D}+\frac{1}{2\lambda}\|x-\bar{x}\|^2+o(\|z-\bar{z}\|^2)\notag\\
&=\frac{1}{2\lambda}\|z-\bar{z}\|_{D}^2+\frac{1}{2}\langle D(z-\bar{z}),\nabla^2 f(\prox{z})D(z-\bar{z})\rangle -\frac{1}{2\lambda}\|x-\bar{x}\|^2+o(\|z-\bar{z}\|^2)\notag\\
&=\frac{1}{2}\langle z-\bar{z}, DM(z-\bar{z})\rangle+o(\|z-\bar{z}\|^2)
\end{align}
for all $D \in \partial\prox{z}$, $M = \nabla^2 f(\prox{z})D + \frac{1}{\lambda}(I-D)$ and $z \to \bar z$. \end{proof}
 
Next, we study the descent properties of the merit function $\mer$ along a sequence of directions $\{d_k\}$ that converges superlinearly with respect to $\{z_k\}$.  

\begin{lemma}\label{lemma3-3}
Assume that the conditions \ref{A1}, \ref{B3}, and \ref{D1}--\ref{D7} are satisfied and let $\{d_k\}$ be a superlinearly convergent sequence in the following sense
\be \label{eq:p-sup-dir} \|z_k + d_k - \bar z\| = o(\|z_k - \bar z\|) \quad k \to \infty. \ee
Then, for all $\eta \in (0,1)$ there exist $\bar \sigma > 0$ (which is independent of $\eta$) and $K_\eta \geq K$ such that for all $k \geq K_\eta$
\be \label{eq:p-des} \mer(z_k+d_k)-\mer(z_k) \leq - \tfrac{\bar \sigma}{2\lambda} \|\prox{z_k+d_k} - x_k\|^2 -\tfrac{\eta\tau\lambda}{2}\|\Fnor{z_k}\|^2. \ee
\end{lemma}
\begin{proof} We first set $\bar{x}=\prox{\bar{z}}$ and $\hat x_{k+1} := \prox{z_k+d_k}$. Then, by applying \cref{prop3-2} and \eqref{eq:p-sup-dir}, it follows
\[ \psi(\hat x_{k+1}) - \psi(\bar x) = \half \iprod{z_k + d_k - \bar z}{\hat D_k \hat H_k (z_k+d_k - \bar z)} + o(\|z_k - \bar z\|^2), \quad k \to \infty, \]
where $\hat H_k := \nabla^2 f(\hat x_{k+1})\hat D_k + \frac{1}{\lambda}(I-\hat D_k)$ and $\hat D_k \in \partial\prox{z_k+d_k}$. Since $\{z_k+d_k\}$ converges and $f$ is twice continuously differentiable on $\dom{\vp}$, the matrices $\{\hat H_k\}$ need to be bounded. Hence, we can infer $\psi(\hat x_{k+1}) - \psi(\bar x) = o(\|z_k - \bar z\|^2)$ as $k \to \infty$. We further note that the semismoothness of the proximity operator implies
\be \label{eq:little-semi} x_k - \bar x = D_k(z_k - \bar z) + o(\|z_k - \bar z \|) \quad \text{as} \quad k \to \infty, \ee
where $D_k \in \partial \prox{z_k}$ is the generalized derivative chosen in step 2 of \cref{algo2}. Due to condition \ref{D4} and as shown in \eqref{eq:dm-lwb}, there exists $\bar \sigma > 0$ with $\iprod{h}{D_k M_k h} \geq \bar \sigma\lambda^{-1} \|D_k h\|^2$ for all $h \in \Rn$ and $k \geq K$. Using \cref{prop3-2} a second time, \eqref{eq:little-semi}, \ref{B3}, the boundedness of $\{\nabla^2 f(x_k)\}$, and \ref{D7}, this yields 
\begin{align} \psi(\hat x_{k+1}) - \psi(x_k) & = - \frac12 \iprod{z_k-\bar z}{D_k H_k(z_k - \bar z)} + o(\|z_k - \bar z\|^2) \notag \\ 
& = - \frac12 \iprod{z_k-\bar z}{D_k M_k(z_k - \bar z)} + \frac12 \iprod{z_k-\bar z}{D_k[B_k-\nabla^2f(x_k)]D_k(z_k-\bar z)} + o(\|z_k-\bar z\|^2) \notag \\ & \leq - \frac{\bar \sigma}{2\lambda} \|D_k(z_k-\bar z)\|^2 + o(\|z_k - \bar z\|^2) \label{eq:esti_1000} \end{align}
for $k \geq K$ and $k \to \infty$, where $H_k := \nabla^2 f(x_k)D_k + \frac{1}{\lambda}(I-D_k)$. Furthermore, due to the nonexpansiveness of $\proxs$, we have $\|\hat x_{k+1}-\bar{x}\|= O(\|z_{k}+d_k-\bar{z}\|)=o(\|z_k-\bar{z}\|)$ and $\|x_k - \bar x\| = O(\|z_k - \bar z\|)$. Combining this with \eqref{eq:little-semi}, we then obtain
\[ \|\hat x_{k+1} - x_k\|^2 = \|x_k - \bar x\|^2 + o(\|z_k - \bar z\|^2)  = \|D_k(z_k - \bar z)\|^2 + o(\|z_k-\bar z\|^2). \]
As verified in \eqref{eq:fnor-lipschitz} in the proof of \cref{thm:global_conv}, the normal map is Lipschitz continuous with constant $L_F$. Thus, it holds that $\|\Fnor{z_k+d_k}\|^2 =O(\|z_k+d_k-\bar{z}\|^2) = o(\|z_k-\bar{z}\|^2) $ and using \eqref{eq:esti_1000}, we can infer 
%
%
%
\begin{align*} \mer(z_k+d_k)-\mer(z_k) \leq - \frac{\bar \sigma}{2\lambda} \|\hat x_{k+1} - x_k\|^2 -\frac{\tau\lambda}{2}\|\Fnor{z_k}\|^2+o(\|z_k-\bar{z}\|^2)  \end{align*}
for $k \to \infty$. Let us define $\kappa_H := \sup_k \|\nabla^2 f(x_k)\| < \infty$. We now choose $K_\eta \geq K$ sufficiently large, such that $\|\Fnor{z_k} - \Fnor{\bar z} - H_k(z_k-\bar z) \| \leq (6\kappa_M)^{-1} \|z_k - \bar z\|$, $\|[\nabla^2 f(x_k)-B_k](x_k-\bar x) \| \leq (6\kappa_M)^{-1} \|z_k - \bar z\|$, and 
\begin{align*} \|x_k-\bar x - D_k(z_k-\bar z)\| & \leq (6\kappa_M[\kappa_B+\kappa_H])^{-1} \|z_k - \bar z\|, \\  \mer(z_k + d_k) - \mer(z_k) & \leq - \frac{\bar\sigma}{2\lambda} \|\hat x_{k+1}-x_k\|^2 -\frac{\tau\lambda}{2}\|\Fnor{z_k}\|^2 + c \|z_k - \bar z\|^2 \end{align*}
%
%
with $c = (1-\eta)\tau\lambda / (8\kappa_M^2)$ for all $k \geq K_\eta$. This is possible due to the semismoothness of $\oFnor$ and $\proxs$ at $\bar z$ (see \cref{prop2-6}) and \ref{D7}. The bounded invertibility of the matrices $\{M_k\}$ then implies
\begin{align} \nonumber \|z_k - \bar z\| & \leq \kappa_M [\|\Fnor{z_k}-\Fnor{\bar z} - H_k(z_k - \bar z)\| + \|\Fnor{z_k}\| \\ & \hspace{2ex}+ \|\nabla^2 f(x_k)-B_k\| \|x_k - \bar x - D_k(z_k-\bar z)\| \nonumber   \nonumber  + \|[\nabla^2 f(x_k)-B_k](x_k-\bar x)\| ] \\ &\leq \tfrac12 \|z_k - \bar z\| + \kappa_M \cdot \chi(z_k).  \label{eq:use-later} \end{align}
and consequently, we have $\|z_k - \bar z\|^2 \leq 4 \kappa_M^2\cdot \chi(z_k)^2$. Finally, the choice of $c$ ensures that the bound stated in \cref{lemma3-3} holds for all $k \geq K_\eta$. \end{proof}

Notice that the result in \cref{lemma3-3} remains valid if condition \eqref{eq:p-sup-dir} is only satisfied on a subsequence. In that case, the descent property \eqref{eq:p-des} still holds if it is restricted to such subsequence. 


\subsection{Acceptance of TR Steps and Proof of \cref{thm:main-local-conv}} \label{sec:proof-loc}

In the following, we provide a technical lemma that bounds the norm of a trust region step $s_k$ in terms of $\|\Fnor{z_k}\|$.

\begin{lemma}
\label{lemma6-6} Suppose that  \ref{B3} and \ref{D4} are satisfied. Then, it holds that
\[ \|s_k\| \leq \|\bar{s}_k\| \leq \kappa_s \|\Fnor{z_k}\| \quad \forall~k \geq K, \]
where $\kappa_s :=\max\{\lambda,\kappa_M\} + \kappa_M (1+ \lambda\kappa_B)$. 
\end{lemma}

\begin{proof} We discuss three different cases.

\textbf{Case 1:} $\epsilon_k > \|\Fnor{z_k}\|$. Then, we have $\|D_k\Fnor{z_k}\| \leq \|D_k\| \|\Fnor{z_k}\| <\epsilon_k$,
where we used $\|D_k\|\leq 1$. Hence, the CG-method terminates in the first step with $\bar q_k = q_0 = 0$ and $\|s_k\| \leq \|\bar{s}_k\| =\lambda\|\Fnor{z_k}\|$.

\textbf{Case 2:} $\epsilon_k\leq \|\Fnor{z_k}\|$ and $\Delta_k \geq \|M_k^{-1} \Fnor{z_k}\|$. 
In this case, \cref{lemma3-8} is applicable and it follows $\|D_k (M_k \bar q_k+ \Fnor{z_k})\|\leq\epsilon_k$. Using \cref{lemma2-1}, we  have $\|M_k\bar{s}_k+\Fnor{z_k}\| \leq \|I-\lambda B_k  \|\epsilon_k$ and
\begin{align*}  \|s_k\|\leq\|\bar{s}_k\|\leq \kappa_{M}\|M_k\bar s_k\| & \leq \kappa_{M}(\|\Fnor{z_k}\|+\|I-\lambda B_k \|\epsilon_k)\\ & \leq\kappa_{M}(2+\lambda\kappa_B) \|\Fnor{z_k}\|. \end{align*}

\textbf{Case 3:} $\epsilon_k \leq \|\Fnor{z_k}\|$ and $\Delta_k < \|M_k^{-1} \Fnor{z_k} \|$. In this case, due to \eqref{eq3-3} and $\|D_k\| \leq 1$, we obtain
\begin{align*}
\|(I-\lambda  M_k)\bar q_k\| \leq \|(I-\lambda  B_k)D_k \| \cdot \Delta_k \leq \kappa_M(1+\lambda\kappa_B) \|\Fnor{z_k}\|    
\end{align*}
and $\|s_k\|  \leq\|\bar{s}_k\|  \leq \|\lambda  \Fnor{z_k}\|+\|(I-\lambda  M_k)\bar q_k \| \leq [\lambda+\kappa_M(1+ \lambda\kappa_B)] \|\Fnor{z_k}\|$.

Combining the different cases, it follows $\|s_k\|\leq\|\bar{s}_k\|\leq \kappa_s \|\Fnor{z_k}\|$. 
\end{proof}

We now present the proof of \cref{thm:main-local-conv}. Our overall strategy is to show that the directions $\{s_k\}$ are superlinearly convergent with respect to $\{z_k\}$. As in the proof of \cref{lemma3-3}, this is mainly a consequence of the semismoothness of $\proxs$ and $\oFnor$ and of the Dennis-Mor\'{e} condition \ref{D7}. The derivation utilizes the properties of the Steihaug-CG method, $\Delta_{\min} > 0$, and \ref{B3}, \ref{D4}, and \ref{D5}. In the second part of the proof, we then discuss the behavior of the sequences $\{\nu_k\}$ and $\{\mu_k\}$ to ensure $\rho_k \geq \eta_1$ for all $k$ sufficiently large. 

\begin{proof} \cref{thm:global_conv} ensures that $\bar z$ is a solution of \eqref{normal_eq} and that the sequences $\{\chi(z_k)\}$ and $\{\Fnor{z_k}\}$ converge to zero. Furthermore, as a consequence of the assumptions \ref{D1}, \ref{D4}, and \ref{D6}, there exists $K^\prime \geq K$ such that the conditions
\begin{itemize}
\item[(a)] $\Delta_{\min}\geq \max\{\lambda,\kappa_{M},\kappa_s\}\|\Fnor{z_k} \|$ and $\chi(z_k) \leq 1$;
\item[(b)] $D_k M_k\succeq 0$ and $\|M_k^{-1}\|\leq \kappa_{M}$;
\item[(c)] ${a_{n_\cS(k)+1}}/{a_{n_\cS(k)}}\leq 2\tilde{\kappa}_{a}$
\end{itemize}
hold for all $k \geq K^\prime$. By the algorithmic construction and \cref{lemma6-6}, we then have $\Delta_k \geq \Delta_{\min} \geq \kappa_M \|\Fnor{z_k} \| \geq \|M_k^{-1} \Fnor{z_k}\|$ and $\Delta_k \geq \Delta_{\min} \geq \|\bar s_k\|$ for all $k \geq K^\prime$ and $k-1 \in \cS$. This establishes $s_k = \bar s_k$ for all $k \geq K^\prime$, $k-1 \in \cS$. Moreover, \cref{lemma3-8} (ii) is applicable and thus the CG-method will return an $\veps_k$-accurate solution $\bar q_k$ of the linear system \eqref{eq3-2}. Similar to \eqref{eq:use-later} and utilizing \cref{lemma2-1}, \ref{D4}, and \ref{D5}, this implies 
\begin{align} \label{eq:th-sq} \|z_k+ s_k - \bar z \|  & = \|z_k+ \bar s_k - \bar z \| \\ \nonumber & \hspace{-10ex} \leq \kappa_M [\|\Fnor{z_k}-\Fnor{\bar z} - M_k(z_k - \bar z)\| + \|\Fnor{z_k} + M_k \bar s_k\|] \\ \nonumber & \hspace{-10ex} \leq \kappa_M [\|\Fnor{z_k}-\Fnor{\bar z} - H_k(z_k - \bar z)\|  + (1+\lambda\kappa_B)\mathcal E(\Fnor{z_k}) \\ \nonumber & \hspace{-6ex}+ \|\nabla^2 f(x_k)-B_k\| \|x_k - \bar x - D_k(z_k-\bar z)\| + \|[\nabla^2 f(x_k)-B_k](x_k-\bar x)\| ], \end{align}
for all $k \geq K^\prime$, $k-1\in\cS$, where $H_k = \nabla^2 f(x_k)D_k + \frac{1}{\lambda}(I-D_k)$. Next, the Lipschitz continuity of $\oFnor$ (cf. \eqref{eq:fnor-lipschitz}) and condition \ref{D5} readily yield $\mathcal E(\Fnor{z_k}) = o(\|z_k - \bar z\|)$ as $k \to \infty$. The assumptions \ref{D2}--\ref{D3} then imply that $\oFnor$ is semismooth at $\bar z$ and as before due to \ref{B3}, \ref{D1}, \ref{D3}, and \ref{D7}, we have
\be \label{eq:supi} \|z_k+ s_k - \bar z \| = o(\|z_k - \bar z\|) \quad \cS^\prime \ni k \to \infty, \ee
where $\cS^\prime := \{k: k-1 \in \cS\}$. Notice that \cref{lemma2-7} ensures $\vert\cS\vert, \vert\cS^\prime\vert = \infty$. Let $L_F$ denote the Lipschitz constant of $\oFnor$. Using \cref{lemma3-3} (on the subsequence defined by $\cS^\prime$) and \eqref{eq:supi}, there are $K^{\prime\prime} \geq K^\prime$ and $\bar \sigma > 0$ such that the conditions (a)--(c), and 
\begin{itemize}

\item[(d)] $\mer(z_k+s_k)-\mer(z_k) \leq - \frac{\bar \sigma}{2\lambda}\|\prox{z_k+s_k} - \prox{z_k}\|^2 -\frac{\eta_1\tau\lambda}{2}\|\Fnor{z_k}\|^2 $;
\item[(e)] $\|z_k + s_k - \bar z\| \leq \|z_k - \bar z\| / \kappa_b$ where $\kappa_b := 2 L_F \kappa_M (2\tilde \kappa_a)^{\frac{1}{p}}$; 
\item[(f)] and $\|z_k - \bar z\| \leq 2 \kappa_M \|\Fnor{z_k}\|$
\end{itemize}
are satisfied for all $k \in \cS^{\prime\prime} := \{k: k \in \cS^\prime, \, k \geq K^{\prime\prime}\}$. Notice that condition (f) can be shown as in \eqref{eq:use-later} and is a consequence of the semismoothness of $\oFnor$. In addition, the Lipschitz continuity of $\oFnor$ and (e)--(f) imply
\be \label{eq:p-lip-bnd} \|\Fnor{z_k+s_k}\| \leq L_F \|z_k + s_k - \bar z\|  \leq \frac{L_F}{\kappa_b} \|z_k-\bar z\| \leq (2\tilde \kappa_a)^{-\frac{1}{p}} \|\Fnor{z_k}\| \ee
for all $k \in \cS^{\prime\prime}$. We now assume that there exists an index $\ell \in \N$ with 
\be \label{eq:ind-conv} \ell \in \cS^{\prime\prime}, \quad a^2_{n_\cS(\ell)}\|\Fnor{z_{\ell}}\|^{2p} \leq {\bar{\sigma}}/{(2\eta_1\kappa_s^{2p})}. \ee
Let us first estimate the parameter ${\mu}_{\ell}$. Due to condition (a) and $\ell-1\in\cS$, we have $\min\{\Delta_{\ell},\lambda\chi(z_{\ell})\}=\lambda\chi(z_{\ell})$. Hence, using \ref{D6}, \cref{lemma6-6}, \eqref{eq:ind-conv}, and (d), it follows
\begin{align*} 
 {\mu}_{\ell} & \leq \frac{\chi(z_{\ell}) \cdot a_{n_\cS(\ell)}^2\|\prox{z_{\ell}+s_{\ell}} - \prox{z_{\ell}}\|^{2p}}{\min\{\Delta_{\ell},\lambda\chi(z_{\ell})\}}  = a_{n_\cS(\ell)}^2\lambda^{-1}\cdot\|\prox{z_{\ell}+s_{\ell}} - \prox{z_{\ell}}\|^{2p} \\ & \leq a_{n_\cS(\ell)}^2\lambda^{-1}\cdot\|s_{\ell}\|^{2p} \leq \lambda^{-1}\cdot \kappa_s^{2p} a_{n_\cS(\ell)}^2 \|\Fnor{z_\ell}\|^{2p} \leq {\bar{\sigma}}/{(2\lambda\eta_1)}
\end{align*}
and $\mer(z_\ell) - \mer(z_\ell+s_\ell) \geq \eta_1\mathrm{pred}(z_{\ell},s_{\ell},\Delta_{\ell},\nu_{\ell})$ which shows $\ell \in \cS$ and $\ell+1 \in \cS^{\prime\prime}$. Moreover, by \eqref{eq:p-lip-bnd} and (c), we have 
\begin{align*} a^2_{n_\cS(\ell+1)} \|\Fnor{z_{\ell+1}}\|^{2p} & =  a^2_{n_\cS(\ell)+1} \|\Fnor{z_{\ell+1}}\|^{2p}  \leq 4\tilde \kappa_a^2 a^2_{n_\cS(\ell)} \|\Fnor{z_{\ell+1}}\|^{2p} \leq a^2_{n_\cS(\ell)} \|\Fnor{z_{\ell}}\|^{2p}. \end{align*}
Thus, the conditions in \eqref{eq:ind-conv} are also satisfied for the iteration $\ell+1$ and inductively (since the conditions (a)--(f) hold for all $k \in \cS^{\prime\prime}$), we obtain $k \in \cS$ for all $k \geq \ell - 1$. The q-superlinear convergence of $\{z_k\}$ then follows from \eqref{eq:supi} (which now holds for all $k \geq \ell$). To complete the proof, we need to verify the existence of an index $\ell$ as in \eqref{eq:ind-conv}. By \ref{D6}, we have $\limsup_{\mathcal{S}^{\prime\prime}\ni k\rightarrow\infty} {a_{n_\cS(k)}^{2}}/{n_\cS(k)^{2p(1+q)}}=0$ and it follows
\begin{align*} \liminf\limits_{\mathcal{S}^{\prime\prime}\ni k\rightarrow\infty}a_{n_\cS(k)}^2 \|\Fnor{z_k}\|^{2p} & \leq \limsup\limits_{\mathcal{S}^{\prime\prime}\ni k\rightarrow\infty}\frac{a_{n_\cS(k)}^2}{n_\cS(k)^{2p(1+q)}} \cdot \liminf\limits_{\mathcal{S}^{\prime\prime}\ni k\rightarrow\infty} [n_\cS(k)^{1+q} \|\Fnor{z_k}\|]^{2p} = 0, \end{align*}
which proves the existence of such an index $\ell$ {(where we also applied \ref{D1})}. The additional conditions in the second statement of \cref{thm:main-local-conv} imply that $\oFnor$ is $\beta$-order semismooth at $\bar z$. The estimate in \eqref{eq:th-sq} can then be improved to $\|z_k+s_k - \bar z\| \leq O(\|z_k-\bar z\|^{1+\beta})$ which proves convergence of order $1+\beta$. \end{proof}


\begin{remark} As seen in \cref{lemma3-3}, the achievable descent of the merit function $\mer$ largely depends on the curvature constant $\bar \sigma$ which is typically unknown. Hence, in order to ensure $\eta_1\mu_k \leq 0.5 \bar \sigma$, the sequences $\{\nu_k\}$ and $\{\mu_k\}$ have to converge to zero. By contrast and as verified in \cref{sec:KL_conv}, the parameters $\{\nu_k\}$ should also not decrease too quickly to still allow applicability of the KL-theory. The implementable choices of $\{\nu_k\}$ and $\{\mu_k\}$ presented in \eqref{eq:nuk-ak} and \eqref{eq:muk-xk} balance these requirements and allow establishing unified global-local convergence results. Moreover, \cref{thm:main-local-conv} shows that \cref{algo2} does not suffer from ``Maratos-type effect'' that would prevent transition to fast local convergence and acceptance of the semismooth Newton steps.
\end{remark}

\section{Second-Order Properties} \label{sec:sop}

The goal of this section is to study second-order properties of the functions $\psi \circ \oprox$ and $\mer$ and to investigate different second-order optimality conditions and concepts for the original minimization problem \eqref{eq1-1} and for the auxiliary problem $\min_{z}\, (\psi \circ \oprox)({z})$. In particular, our results will allow us to discuss the conditions for superlinear convergence stated in \cref{assum3-1} in more detail and to connect them to second-order optimality conditions. 

\subsection{Preliminaries and Basic Differentiability Properties}

Due to the intrinsic nonsmoothness of the proximity operator, we can not expect that $\psi\circ\oprox$ or $\mer$ are differentiable everywhere like the forward-backward envelope introduced in \cite{PatBem13,PatSteBem14}. However, in the following proposition, we show that the mapping $\psi\circ\oprox$ enjoys stronger differentiability properties than $\psi$ at a stationary point.

\begin{proposition} \label{proposition:strict-diff}
Let $\bar{z}$ be a given solution of the nonsmooth equation \eqref{normal_eq}. Then, both $\psi\circ\oprox$ and $\mer$ are strictly differentiable at $\bar{z}$ with $\nabla (\psi \circ \oprox)(\bar z) = 0$. 
\end{proposition}
\begin{proof}
We first prove the conclusion for $\psi\circ\oprox$. By the definition of strict differentiability, see, e.g., \citep[Definition 9.13]{rockafellar2009variational}, we need to show:
\begin{align}\label{eq6-1}
  \lim\limits_{z,z'\rightarrow\bar{z}}\frac{\psi(\prox{z})-\psi(\prox{z'})}{\|z-z'\|}=0.            
\end{align}
Setting $x=\prox{z},x'=\prox{z'}$, we have:
\begin{align*}
\env{z}-\env{z'} &=\langle\nabla\env{z'},z-z'\rangle+o(\|z-z'\|), \; f(x)-f(x') &=\langle \nabla f(x'),x-x'\rangle+o(\|x-x'\|),
\end{align*}
as $z,z^\prime \to \bar z$. These two expansions are uniform near $\bar{z}$ because both $\nabla\oenv$ and $\nabla f$ are Lipschitz continuous near $\bar{z}$. As in the proof of \cref{prop3-2} and by the definition of the Moreau envelope $\oenv$, we can infer
\begin{align*}
   \psi(x)-\psi(x')&=f(x)+\env{z}-f(x')-\env{z'}-\frac{1}{2\lambda}\|x-z\|^2+\frac{1}{2\lambda}\|x'-z'\|^2\\
   & = \iprod{\nabla f(x^\prime)}{x-x^\prime} + \frac{1}{2\lambda} \|z-x^\prime\|^2 - \frac{1}{2\lambda} \|x-z\|^2 + o(\|z-z^\prime\|) 
   \\&= \langle \Fnor{z'},x-x' \rangle+o(\|z-z'\|),
  \end{align*}
where we have used the fact that $o(\|x-x'\|)=o(\|z-z'\|)$. Then, \eqref{eq6-1} follows from the continuity of $\oFnor$. To prove the claim for $\mer$, it suffices to show that $g:\mathbb{R}^n\rightarrow\mathbb{R}$, $g(z)=\frac{\lambda}{2}\|\Fnor{z}\|^2$ is strictly differentiable at $\bar{z}$ with zero gradient. Setting $h(z)=\frac{\lambda}{2}\|z\|^2$, we can write $g=h\circ\oFnor$. Moreover, we have $\partial h(\Fnor{\bar{z}})= \{\nabla h(\Fnor{\bar z})\} = \{\lambda\Fnor{\bar z}\} = \{0\}$. Hence, applying \citep[Proposition 2.2.4]{clarke1990optimization}, we get the conclusion.
\end{proof}

In the following, we briefly introduce different notions of generalized second-order differentiability that will be the basis of our analysis. We will mainly work with the second-order epi-derivative
\begin{align*}
  \mathrm{d}^2_-\theta(x\vert \alpha)(h) &:= \liminf_{t\downarrow 0, \, \tilde h \to h}~\Delta_t^2\theta(x\vert \alpha)(\tilde h), \quad \Delta_t^2  \;\! \theta(x\vert \alpha)(h) := \frac{\theta(x+th)-\theta(x)-t \cdot \iprod{\alpha}{h}}{\half t^2}.
\end{align*}
Here, $\mathrm{d}^2_-\theta(x\vert \alpha)(h)$ denotes the {lower second-order subderivative} of $\theta$ at $x$ relative to $\alpha$ in the direction $h \in \Rn$. 
We say that $\theta$ is {twice epi-differentiable} at $x$ for $\alpha \in \Rn$ if the second-order difference quotients $ \Delta_t^2  \;\! \theta(x\vert \alpha)(h)$ epi-converge in the sense of \cite[Definition 7.1]{rockafellar2009variational}. We will use $\mathrm{d}^2\theta(x\vert \alpha)$ to denote the corresponding epi-limit. The mapping $\theta$ is called {twice semidifferentiable} at $x$ if it is semidifferentiable and if the limit 
\[ \lim_{t\downarrow 0, \, \tilde h \to h} \frac{\theta(x+t\tilde h) - \theta(x) - t \cdot \theta^\prime(x;\tilde h)}{\frac12 t^2} \]
exists for all $h \in \Rn$. The limiting function will then be denoted by $\theta^{\prime\prime}(x; \cdot)$. The interested reader is referred to \cite[Chapter 13]{rockafellar2009variational} and \cite[Sections 2.2 and 3.3.5]{BonSha00} for a thorough discussion of these second-order concepts. Let us recall that a function $\varrho : \Rn \to \Rexx$ is called proper if $\varrho(x) > -\infty$ and $\dom{\varrho} \neq \emptyset$.  
We now collect several useful properties of the lower second-order subderivative which have been established in \cite[Proposition 13.5 and 13.20]{rockafellar2009variational}.
\begin{lemma} \label{lemma:prop-subderiv} Let $\theta : \Rn \to \Rex$ and $(x,\alpha) \in \dom{\theta} \times \Rn$ be given.
\begin{itemize}
\item[\rmn{(i)}] The subderivative $\mathrm{d}^2_-\theta(x\vert \alpha) : \Rn \to \Rexx$ is lower semicontinuous and positively homogeneous of degree 2, i.e., $\mathrm{d}^2_-\theta(x\vert \alpha)(th) = t^2  \mathrm{d}^2_-\theta(x\vert \alpha)(h)$ for all $t > 0$ and $h \in \Rn$.
\item[\rmn{(ii)}] If $\mathrm{d}^2_-\theta(x\vert \alpha)$ is a proper function, then it follows $\dom{\mathrm{d}^2_-\theta(x\vert \alpha)} \subset \mathcal C(x;\alpha) := \{h : \iprod{\alpha}{h} = \theta_-^\downarrow(x;h)\}$.
\end{itemize}
Suppose that $\theta$ is convex and let $\alpha \in \partial \theta(x)$ be given. Then, it additionally holds that:
\begin{itemize}
\item[\rmn{(iii)}] We have $\mathrm{d}^2_-\theta(x\vert \alpha)(h) \geq 0$ for all $h \in \Rn$. If $\theta$ is twice epi-differentiable at $x$ for $\alpha$, then $h \mapsto \mathrm{d}^2\theta(x\vert \alpha)(h)$ is a convex function. 
\end{itemize}
\end{lemma}
 
We call the pair $(\bar x,\bar z) \in \dom{\vp} \times \Rn$ a criticality pair of problem \eqref{eq1-1}, if $\Fnat{\bar x} = \Fnor{\bar z} = 0$ and if $\bar x = \prox{\bar z}$ or, equivalently, $\bar z = \bar x - \lambda\nabla f(\bar x)$. Next, we formulate our main second-order differentiability assumptions. 

\begin{assumption}
\label{assum:epi-diff} Let $(\bar x,\bar z) \in \dom{\vp} \times \Rn$ be a criticality pair of problem \eqref{eq1-1}. We assume: 
\begin{enumerate}[label=\textup{\textrm{(E.\arabic*)}},topsep=0pt,itemsep=0ex,partopsep=0ex]
\item \label{E1} The mapping $\vp$ is twice epi-differentiable at $\bar x$ for $-\nabla f(\bar x)$. 
\item \label{E2} The function $\vp$ is twice epi-differentiable at  $\bar x$ for $-\nabla f(\bar x)$ and it holds that
\[ \mathrm{d}^2\vp(\bar x\vert -\nabla f(\bar x))(h) = \iprod{h}{Qh} + \iota_{S}(h) \quad \forall~h,\] 
where $Q \in \R^{n \times n}$ is some symmetric, positive semidefinite matrix and $S \subset \Rn$ is a linear subspace. 
\end{enumerate}
\end{assumption}

Using the correspondence between $\bar x$ and $\bar z$, condition \ref{E1} coincides with assuming twice epi-differentiability of $\vp$ at $\prox{\bar z}$ for $\nabla \env{\bar z}$. Furthermore, due to $-\nabla f(\bar x) \in \partial \vp(\bar x)$ and \cref{lemma:prop-subderiv} (i) and (iii), the second-order subderivative $h \mapsto \Upsilon(h) := \mathrm{d}^2\vp(\bar x\vert -\nabla f(\bar x))(h)$ is a proper function. Assumption \ref{E2} additionally requires that $\Upsilon$ is a generalized quadratic. 
Hence, condition \ref{E2} essentially coincides with the second-order assumptions stated and utilized in \cite{SteThePat17,TheStePat18,mai2019anderson}. 

The class of functions and applications for which the second-order subderivative $\Upsilon$ is a generalized quadratic and satisfies the structural property stated in condition \ref{E2} is rather rich and encompasses (fully) amenable mappings, see, e.g., \cite{PolRoc92,PolRoc93} or \cite[Chapter 10 and 13]{rockafellar2009variational}, $C^{2}$-cone reducible constraints \cite[Section 3.4.4]{BonSha00}, and decomposable functions \cite{Sha03,milzarek2016numerical}. More specific examples and related references are discussed, e.g., in \cite[Section 5.3]{milzarek2016numerical}. 

Twice epi-differentiability is a powerful tool and allows to characterize differentiability properties of the Moreau envelope and proximity operator. In the following and based on the pioneering observations in \cite{PolRoc96,PolRoc96-2,rockafellar2009variational}, we briefly state some of these fundamental connections for our special situation.  

\begin{thm} \label{theorem:epi-semi} The following conditions are equivalent:
\begin{itemize}
\item[\rmn{(i)}] The mapping $\vp$ satisfies assumption \ref{E1}.
\item[\rmn{(ii)}] The proximity operator $\proxs$ is semidifferentiable at $\bar z$.
\item[\rmn{(iii)}] The Moreau envelope $\envs$ is twice semidifferentiable at $\bar z$.
\end{itemize}
Furthermore, in this case, it follows $(\envs)^{\prime\prime}(\bar z;h) = \min_{y \in \Rn} \mathrm{d}^2\vp(\prox{\bar z}\vert \nabla \env{\bar z})(y) + \frac{1}{\lambda}\|h-y\|^2$ and we have $(\proxs)^{\prime}(\bar z;h) = \argmin_{y \in \Rn} \mathrm{d}^2\vp(\prox{\bar z}\vert \nabla \env{\bar z})(y) + \frac{1}{\lambda}\|h-y\|^2$. In addition, assumption \ref{E2} is equivalent to the condition
\be \label{eq:prox-diff} \proxs \; \text{is differentiable at} \; \bar z. \ee 
\end{thm}
\begin{proof} The equivalence of the first three conditions and the formulae for $(\envs)^{\prime\prime}$ and $(\proxs)^\prime$ can be shown as in \cite[Theorem 3.5]{PolRoc96-2} or \cite[Exercise 13.45]{rockafellar2009variational}. The equivalence of \ref{E2} and \eqref{eq:prox-diff} essentially follows from \cite[Theorem 3.8]{PolRoc96-2}. 
\end{proof}
 
Let us notice that the results in \cref{theorem:epi-semi} do also hold in a much more general setting when $\vp$ is only assumed to be prox-bounded and prox-regular. However, since our algorithmic framework currently relies on the (uniform) $\Lambda$-firm nonexpansiveness of the proximity operator, we will concentrate on the convex case. 

\begin{remark} \label{rem:crit} It is possible to connect the linear subspace $S$ introduced in \ref{E2} to the associated critical cone 
\[ \mathcal C(\bar x) := \{h : \psi^\downarrow(\bar x; h) = 0 \} = \{h: \vp^\downarrow(\bar x;h) = -\iprod{\nabla f(\bar x)}{h}\} = N_{\partial\vp(\bar x)}(-\nabla f(\bar x)), \]
where $N_{\partial\vp(\bar x)}(-\nabla f(\bar x)) = \{v : \iprod{v}{y+\nabla f(\bar x)} \leq 0, \, \forall~y \in \partial\vp(\bar x)\}$ denotes the standard normal cone. Using this representation of the critical cone and applying \cite[Proposition 2.2]{LemSag96}, it can be shown that the strict complementarity condition
\[ -\nabla f(\bar x) \in \mathrm{ri}(\partial\vp(\bar x)) \] 
is equivalent to saying that $\mathcal C(\bar x)$ is a subspace. Moreover, differentiability of the proximity operator $\oprox$ will remain fully equivalent to assumption \ref{E2} with $S = \mathcal C(\bar x)$ under an additional parabolic derivability condition. We refer to \cite[Definition 13.11 and Example 13.62]{rockafellar2009variational} and \cite{MohMorSar19,MohSar20} for more details and novel results on parabolic derivability and parabolic epi-differentiability. 
\end{remark}

\subsection{Second-Order Optimality and Strong Metric Subregularity of $\oFnat$ and $\oFnor$ }
 
Based on \cref{theorem:epi-semi}, we first express and calculate the second-order derivatives of $\psi \circ \proxs$.

\begin{lemma} \label{lemma:psip-sec} Let $\bar z$ be a solution of \eqref{normal_eq} at which \ref{E1} is satisfied and let $f : \Rn \to \R$ be twice continuously differentiable in a neighborhood of $\prox{\bar z}$. Then, the mapping $\psi \circ \proxs$ is twice semidifferentiable at $\bar z$ and we have
\begin{align*} (\psi\circ \proxs)^{\prime\prime}(\bar z;h) & = \iprod{(\proxs)^\prime(\bar z;h)}{(F^\lambda_{\mathrm{nor}})^\prime(\bar z;h)} \\ &= \iprod{(\proxs)^\prime(\bar z;h)}{\nabla^2 f(\prox{\bar z})(\proxs)^\prime(\bar z;h) + \frac{1}{\lambda}(h - (\proxs)^\prime(\bar z;h))} \quad \forall~h. \end{align*}
\end{lemma}
\begin{proof} Let $\tilde h \in \Rn$ be given and let us define $p_t := \prox{\bar z + t \tilde h}$ and $\bar x = \prox{\bar z}$. Similar to the analysis in \cref{sec:loc-super-conv}, it follows
\begingroup
\allowdisplaybreaks
  \begin{align*}& \hspace{-4ex} \psi(\prox{\bar z + t \tilde h}) - \psi(\prox{\bar z})  \\ & = f(p_t) - f(\bar x) + \env{\bar z + t \tilde h} - \env{\bar z} + \frac{1}{2\lambda}\|\bar z - \bar x\|^2 -\frac{1}{2\lambda}\|\bar z + t \tilde h - p_t\|^2 \\ &   = \iprod{\nabla f(\bar x)}{p_t - \bar x} + \frac12 \iprod{p_t - \bar x}{\nabla^2 f(\bar x)(p_t - \bar x)} + \env{\bar z + t \tilde h} - \env{\bar z} + \frac{1}{2\lambda} \|\bar z - \bar x\|^2 \\ &  \hspace{4ex}  - \frac{1}{2\lambda} [ \|\bar z - \bar x\|^2 + 2 \iprod{(\bar z - \bar x)}{t\tilde h + \bar x - p_t} + \|t\tilde h + \bar x - p_t\|^2 ] + o(\|p_t - \bar x\|^2) \\ &  = \frac12 \iprod{p_t - \bar x}{\nabla^2 f(\bar x)(p_t - \bar x)} + [\env{\bar z + t \tilde h} - \env{\bar z} - t \iprod{\nabla\env{\bar z}}{\tilde h}] \\ & \hspace{4ex}  - \frac{1}{2\lambda} \|t\tilde h + \bar x - p_t\|^2 + o(\|p_t - \bar x\|^2), 
  \end{align*}
\endgroup
as $t \downarrow 0$. By \cref{theorem:epi-semi}, the proximity operator $\proxs$ is semidifferentiable at $\bar z$ and hence, due to the continuous differentiability of $\envs$ around $\bar z$, we obtain  
\begin{align*}  \lim_{t \downarrow 0} \frac{\env{\bar z + t h} - \env{\bar z} - t \iprod{\nabla\env{\bar z}}{h}}{\half t^2}  &= \lim_{t \downarrow 0} \int_{0}^1 \frac{2}{t} \iprod{\nabla \env{\bar z+\tau t h} - \nabla \env{\bar z}}{h} \, \mathrm{d}\tau \\ & \hspace{-8ex}= \int_{0}^1 2  \left\langle  \lim_{t \downarrow 0} \frac{\nabla \env{\bar z+\tau th} - \nabla \env{\bar z}}{t}, h \right\rangle \, \mathrm{d}\tau  \\ &  \hspace{-8ex} = \int_0^1 2 \iprod{(\nabla \envs)^\prime(\bar z;\tau h)}{h} \, \mathrm{d}\tau = \frac{1}{\lambda}\iprod{h - (\proxs)^\prime(\bar z;h)}{h}. \end{align*}
Here, we used $(\nabla \envs)^\prime(\bar z;h) = \frac{1}{\lambda}( h - (\proxs)^\prime(\bar z;h))$, the positive homogeneity of $(\nabla \envs)^\prime$, and the dominated convergence theorem to change the order of integration and directional differentiation. More specifically, due to the nonexpansiveness of $\proxs$, it holds
\begin{align*} \frac{2}{t}\vert  \iprod{\nabla \env{\bar z+\tau th} - \nabla \env{\bar z})}{h}\vert & \leq \frac{2 \|h\|}{\lambda t} \| (\tau th - \prox{\bar z+\tau th} + \prox{\bar z})\| \\ & \leq \frac{2 \|h\|}{\lambda t} (\tau t \|h\| + \|\tau th\|) = \frac{4\tau}{\lambda} \|h\|^2 \end{align*}
for all $t > 0$. Hence, we can infer $(\envs)^{\prime\prime}(\bar z; h) = \frac{1}{\lambda}\iprod{h - (\proxs)^\prime(\bar z;h)}{ h}$. (Notice that this limit has to coincide with the second-order semiderivative of $\envs$ at $\bar z$). Since $\psi \circ \proxs$ is strictly differentiable at $\bar z$ with $\nabla(\psi \circ \proxs)(\bar z) = 0$, it now follows
\begin{align*}
  (\psi \circ \proxs)^{\prime\prime}(\bar z;h) &= \iprod{(\proxs)^\prime(\bar z;h)}{\nabla^2 f(\bar x)(\proxs)^\prime(\bar z;h)} + (\envs)^{\prime\prime}(\bar z; h) - \frac{1}{\lambda}\|h - (\proxs)^{\prime}(\bar z; h)\|^2.
\end{align*}
This finishes the proof of \cref{lemma:psip-sec}.
\end{proof}

We now establish second-order optimality conditions and several second-order properties. Specifically, we will derive a sufficient condition for strong metric subregularity of the functions $\oFnor$ and $\oFnat$. Here, we say that a set-valued mapping $F : \Rn \rightrightarrows \R^m$ is strongly metrically subregular at $\bar \xi$ for $\bar y$ if $\bar y \in F(\bar \xi)$ and there exists a constant $\kappa > 0$ and a neighborhood $U$ of $\bar \xi$ such that
\[ \|x - \bar \xi\| \leq \kappa \dist(\bar y, F(x)) \quad \forall~x \in U. \]
If $F$ is single-valued, then strong metric subregularity implies that $\bar \xi$ is an isolated solution of the equation $F(x) = \bar y$. 
\begin{proposition} \label{prop:soc} Assume that condition {\ref{E1}} is satisfied at the criticality pair $(\bar x,\bar z) \in \dom{\vp} \times \Rn$ and let $f$ be twice continuously differentiable in a neighborhood of $\bar x$.
\begin{itemize} 
\item[\rmn{(i)}] Suppose that $\bar z \in \Rn$ is a local minimum of the mapping $\psi\circ \proxs$. Then, we have \vspace{-1ex}
\end{itemize}
\be \label{eq:soc-nec}\begin{aligned}
\max_{D \in \partial \prox{\bar z}} &\iprod{Dh}{\nabla^2f(\prox{\bar z})Dh + \frac{1}{\lambda}(I-D)h} \geq \iprod{(\proxs)^\prime(\bar z;h)}{(F_{\mathrm{nor}}^\lambda)^\prime(\bar z;h)} \geq 0 \quad \forall~h \in \Rn.
\end{aligned}
 \ee
\begin{itemize}
\item[\rmn{(ii)}] The quadratic growth condition: there exists $\sigma,\delta>0$ such that for all $z \in B_\delta(\bar z)$
\be \label{eq:soc-growth1}  \psi(\prox{z}) \geq \psi(\prox{\bar z}) + \frac{\sigma}{2\lambda} \|{\prox{z} - \prox{\bar z}}\|^2  \ee
implies the following second-order optimality condition:
\be \label{eq:soc-suff} \iprod{(\proxs)^\prime(\bar z;h)}{(F_{\mathrm{nor}}^\lambda)^\prime(\bar z;h)}  > 0 \quad \forall~h \; \text{with} \; (\proxs)^\prime(\bar z;h) \neq 0. \ee
Moreover, if condition \eqref{eq:soc-suff} is satisfied, then the mappings $\oFnor$ and $\oFnat$ are strongly metrically subregular at $\bar z$ and $\bar x$ for $0$, respectively. 
\end{itemize}
\end{proposition}

\begin{proof} We first verify part (i). By \cite[Theorem 13.24]{rockafellar2009variational}, the local minimum $\bar z$ satisfies the second-order necessary condition $\mathrm{d}^2(\psi\circ\proxs)(\bar z\vert 0)(h) \geq 0$ for all $h \in \Rn$. Due to $\Fnor{\bar z} = 0$, \cref{proposition:strict-diff}, and \cref{lemma:psip-sec}, the function $\psi \circ \proxs$ is twice semidifferentiable at $\bar z$ with $\nabla (\psi\circ\proxs)(\bar z) = 0$ and thus, it follows
\begin{align*}
  &\mathrm{d}^2(\psi\circ\proxs)(\bar z\vert 0)(h) = (\psi\circ\proxs)^{\prime\prime}(\bar z;h) =\iprod{(\proxs)^\prime(\bar z;h)}{(\oFnor)^\prime(\bar z;h)} \geq 0 \quad \forall~h \in \Rn.
\end{align*}
Applying \cite[Lemma 2.2]{QiSun93}, for every $h \in \Rn$ there exists $D \in \partial\prox{\bar z}$ such that $Dh = (\proxs)^\prime(\bar z;h)$ which establishes the maximum expression in \eqref{eq:soc-nec}. We now continue with the proof of the second part. Using \cref{proposition:strict-diff} and \cref{lemma:psip-sec}, it is easy to show that the second-order growth condition implies 
\[ \iprod{(\proxs)^\prime(\bar z;h)}{(F_{\mathrm{nor}}^\lambda)^\prime(\bar z;h)} \geq \frac{\sigma}{\lambda} \|(\proxs)^\prime(\bar z;h)\|^2 \quad \forall~h \in \Rn. \]
Next, let the second-order optimality condition \eqref{eq:soc-suff} be satisfied and suppose that the normal map is not strongly metrically subregular at $\bar z$ for $0$. Then there exist sequences $\{z_k\}$ and $\{\sigma_k\}$ with $z_k \to \bar z$ and $\sigma_k \to 0$ such that
\[ \|\Fnor{z_k}\| \leq \sigma_k\|{z_k}-{\bar z}\|. \]
Let us define $t_k = \|z_k - \bar z\|$ and $h_k = t_k^{-1}(z_k - \bar z)$. Without loss of generality we may assume that the sequence $\{h_k\}$ converges to some $h$ with $\|h\| = 1$. Using the semidifferentiability of $\oFnor$ and $\proxs$, this yields
\[ \|(\oFnor)^\prime(\bar z; h)\| = \lim_{k \to \infty} \frac{\|\Fnor{\bar z + t_k h_k} - \Fnor{\bar z}\|}{t_k} \leq \lim_{k \to \infty} {\sigma_k} = 0. \]
By the second-order condition \eqref{eq:soc-suff} this can only happen in the case $(\proxs)^\prime(\bar z;h) = 0$. However, we then obtain $0 = (\oFnor)^\prime(\bar z;h) = \frac{1}{\lambda} h$ which is a contradiction to $\|h\| = 1$. Similarly, if $\oFnat$ is not strongly metrically subregular at $\bar x$ for $0$, there exists $h \in \Rn$ with $\|h\| = 1$ and $(\oFnat)^\prime(\bar x;h) = 0$. Setting $V = I - \lambda\nabla^2 f(\bar x)$ and utilizing $\bar x - \lambda\nabla f(\bar x) = \bar z$, it follows
\begin{align} \nonumber 0 =  V (\oFnat)^\prime(\bar x;h) &=  V[h - (\proxs)^\prime(\bar z; Vh)] \\ &= [Vh - (\proxs)^\prime(\bar z;Vh)] + \lambda\nabla^2 f(\bar x) (\proxs)^\prime(\bar z;Vh) = \lambda(\oFnor)^\prime(\bar z;Vh). \label{eq:nor-nat} \end{align}

Again, by \eqref{eq:soc-suff}, this implies $(\proxs)^\prime(\bar z;Vh) = 0$ and $h = (\oFnat)^\prime(\bar x;h) + (\proxs)^\prime(\bar z;Vh) = 0$ which is a contradiction. This concludes the proof of \cref{prop:soc}. 
\end{proof}

In the following, we discuss connections between second-order optimality conditions and several second-order concepts for the problems $\min_x\,\psi(x)$ and $\min_z\,(\psi\circ\proxs)(z)$.

\begin{thm} \label{thm:conn-sec} Let $(\bar x,\bar z) \in \dom{\vp} \times \Rn$ be a given criticality pair and let $f$ be twice continuously differentiable around $\bar x$. Suppose that assumption {\ref{E1}} is satisfied.  
Then, the following conditions are equivalent:
\begin{itemize}
\item[\rmn{(i)}] The second-order sufficient condition holds at $\bar x$:
\be \label{eq:soc-org} \mathrm{d}^2\psi(\bar x\vert 0)(h) = \iprod{h}{\nabla^2 f(\bar x) h} + \mathrm{d}^2\vp(\bar x\vert -\nabla f(\bar x))(h) > 0 \quad \forall~h \in \Rn\backslash\{0\}. \ee
\item[\rmn{(ii)}] There exists $\sigma, \delta > 0$ such that $\psi(x) \geq \psi(\bar x) + \frac{\sigma}{2\lambda} \|x-\bar x\|^2$ for all $x \in B_\delta(\bar x)$.
\item[\rmn{(iii)}] The mapping $\oFnat$ is strongly metrically subregular at $\bar x$ for $0$ and the second-order necessary condition $\mathrm{d}^2\psi(\bar x\vert 0)(h) \geq 0$ holds for all $h \in \Rn$. 
\item[\rmn{(iv)}] The subdifferential $\partial \psi$ is strongly metrically subregular at $\bar x$ for $0$ and the necessary condition $\mathrm{d}^2\psi(\bar x\vert 0)(h) \geq 0$ is satisfied for all $h \in \Rn$. 
\item[\rmn{(v)}] The second-order sufficient optimality condition formulated in \eqref{eq:soc-suff} is fulfilled, i.e., we have
\[ \iprod{(\proxs)^\prime(\bar z;h)}{(\oFnor)^\prime(\bar z;h)} > 0 \quad \forall~h \; \text{with} \; (\proxs)^\prime(\bar z;h) \neq 0. \] 
\item[\rmn{(vi)}] The quadratic growth condition \eqref{eq:soc-growth1} holds at $\bar z$.
\item[\rmn{(vii)}] The normal map $\oFnor$ is strongly metrically subregular at $\bar z$ for $0$ and the second-order necessary optimality condition $\iprod{(\proxs)^\prime(\bar z;h)}{(\oFnor)^\prime(\bar z;h)} \geq 0$ is satisfied for all $h \in \Rn$.
\end{itemize}
\end{thm}
\begin{proof} 
By \cite[Theorem 13.24]{rockafellar2009variational} and \cref{lemma:prop-subderiv}, the assertions (i) and (ii) are equivalent. Notice that the representation of the subderivative ${\mathrm d}^2\psi(\bar x\vert 0)$ in (i) can be shown by applying a second-order Taylor expansion of $f$. The implications ``(ii)$\implies$(vi)'', ``(vi)$\implies$(v)'', and ``(v)$\implies$(vii)'' are an immediate consequence of the $\Lambda$-nonexpansiveness of the proximity operator and of \cref{prop:soc}. Next, we verify that the two second-order necessary conditions stated in (iii), (iv), and (vii) are actually equivalent. Due to \cref{theorem:epi-semi} and \cref{lemma:psip-sec}, it holds that
\begin{align*}
  \mathrm{d}^2\psi(\bar x\vert 0)((\proxs)^\prime(\bar z;h))&=\mathrm{d}^2(\psi \circ \proxs)(\bar z\vert 0)(h) = \iprod{(\proxs)^\prime(\bar z;h)}{(\oFnor)^\prime(\bar z;h)} \quad \forall~h \in \Rn.
\end{align*}
Thus, the standard necessary optimality conditions for the original problem $\min_x\,\psi(x)$ in (iii) are generally stronger and imply \eqref{eq:soc-nec}. To show full equivalence, we now verify $\mathcal R((\proxs)^\prime(\bar z;\cdot)) = \dom{\partial\Upsilon}$ where $\Upsilon(h) := \mathrm{d}^2\vp(\bar x\vert -\nabla f(\bar x))(h)$. By \cref{lemma:prop-subderiv}, the mapping $\Upsilon$ is convex, lower semicontinuous, nonnegative, positively homogeneous of degree $2$, and proper. Thus, utilizing \cref{theorem:epi-semi}, we have the following characterization
\[ p^\prime = (\proxs)^\prime(\bar z;h) \quad \iff \quad 0 \in \partial\Upsilon(p^\prime) + \frac{2}{\lambda}(p^\prime - h). \]
Let $h \in \dom{\partial\Upsilon}$ with $y \in \partial\Upsilon(h)$ be arbitrary. Then, we obtain 
\be \label{eq:prox-dir-sub} 0 \in \lambda\partial\Upsilon(h) + 2(h - [h + {\textstyle\frac12}\lambda y]) \quad \implies \quad h = (\proxs)^\prime(\bar z; h + {\textstyle\frac12}\lambda y) \ee
which yields $\mathcal R((\proxs)^\prime(\bar z;\cdot)) = \dom{\partial\Upsilon}$. Therefore, the second-order necessary condition \eqref{eq:soc-nec} implies
\[  \mathrm{d}^2\psi(\bar x\vert 0)(h) \geq 0 \quad \forall~h \in \dom{\partial\Upsilon}. \]
By \cite[Proposition 16.28 and Corollary 16.29]{BauCom11}, $\dom{\partial\Upsilon}$ is a dense subset of $\dom{\Upsilon}$ and for every $h \in \dom{\Upsilon}$ there exists a sequence $\{h_k\} \subset \dom{\partial\Upsilon}$ with $h_k \to h$ and $\Upsilon(h_k) \to \Upsilon(h)$. Due to $\dom{\mathrm{d}^2\psi(\bar x\vert 0)} = \dom{\Upsilon}$, this finally establishes $\mathrm{d}^2\psi(\bar x\vert 0)(h) \geq 0$ for all $h \in \Rn$ and proves the implication ``(v)$\implies$(iii)''. We now continue with the verification of ``(iii), (iv), (vii)$\implies$(i)''. We mimic the strategy in the proof of \cite[Theorem 9.2]{MohMorSar19} and assume that the second-order sufficient conditions are not satisfied, i.e., there exists $\bar h \neq 0$ with
\[ \mathrm{d}^2\psi(\bar x\vert 0)(\bar h) = 0. \]
By the second-order necessary optimality conditions, this implies that $\bar h$ is a solution of the minimization problem $\min_h \, \mathrm{d}^2\psi(\bar x\vert 0)(h)$. Using the representation \eqref{eq:soc-org} and the calculus mentioned in \cref{sec:foc}, $\bar h$ then has to satisfy the first-order optimality condition
\[ 0 \in \partial[\mathrm{d}^2\psi(\bar x\vert 0)](\bar h) = 2 \nabla^2 f(\bar x)\bar h + \partial \Upsilon(\bar h). \]
Moreover, due to \eqref{eq:prox-dir-sub}, we can infer $\bar h = (\proxs)^\prime(\bar z; \bar h - \lambda\nabla^2 f(\bar x)\bar h)$ or equivalently $(\oFnat)^\prime(\bar x;\bar h) = 0$. Since the strong metric subregularity of $\oFnat$ and $\oFnor$ imply 
\be \label{eq:sms-dir} \exists~\sigma_1, \sigma_2 > 0: \quad \|(\oFnat)^\prime(\bar x;h)\| \geq \sigma_1 \|h\| \quad \text{and} \quad \|(\oFnor)^\prime(\bar z;h)\| \geq \sigma_2 \|h\| \quad \forall~h \in \Rn, \ee
the implication ``(iii)$\implies$(i)'' follows immediately from \eqref{eq:sms-dir} and from the resulting contradiction $\bar h = 0$. Furthermore, if $\oFnor$ is strongly metrically subregular, we can use \eqref{eq:nor-nat} and \eqref{eq:sms-dir} to obtain $V\bar h = 0$ and $\bar h = (\proxs)^\prime(\bar z; V\bar h) = (\proxs)^\prime(\bar z; 0) = 0$ which is again a contradiction and yields ``(vii)$\implies$(i)''. Finally, the strong metric subregularity of the subdifferential $\partial\psi$, \eqref{eq:dist-nor}, and the nonexpansiveness of proximity operator guarantee the existence of $\sigma_3, \delta_3 > 0$ such that
\[ \|\Fnor{z}\| \geq \dist(0,\partial\psi(\prox{z})) \geq \sigma_3 \|\prox{z} - \prox{\bar z}\| \quad \forall~z \in B_{\delta_3}(\bar z). \]
As before this implies $\|(\oFnor)^\prime(\bar z;h)\| \geq \sigma_3 \|(\proxs)^\prime(\bar z;h)\|$ for all $h$ and by \eqref{eq:nor-nat} we conclude $(\oFnor)^\prime(\bar z;V\bar h) =  (\proxs)^\prime(\bar z;V\bar h) = 0$. This is a contradiction and establishes ``(iv)$\implies$(i)''. The remaining implication ``(iii)$\implies$(iv)'' is a simple consequence of \cref{lemma:conn-nat-nor}. 
 
\end{proof}
The results in \cref{thm:conn-sec} are quite satisfactory and provide a precise characterization of the gap between second-order necessary and sufficient optimality conditions as well as a strong connection between the different optimality concepts involving the natural residual $\oFnat$ and the normal map $\oFnor$. Let us note that a similar result for prox-regular and subdifferentially continuous problems was recently established in \cite{ChiHieNghTua19} using the subgradient graphical derivative. Further related second-order results based on parabolic epi-differentiability and parabolic regularity can be found in \cite{MohSar20}. We also refer to \cite{ArtGeo14,DruMorNhg14,DruIof15,DinSunZha17,MohMorSar19,MohSar20} for more discussions. A possible extension of \cref{thm:conn-sec} to the fully nonconvex, prox-regular setting as in \cite[Theorem 3.8]{ChiHieNghTua19} and \cite[Theorem 6.1 and 6.3]{MohSar20} is left for future work. The novel normal map-based second-order conditions in part (v) and (vii) of \cref{thm:conn-sec} complement the results in \cite{MohSar20} and are an appealing alternative to the classical conditions in (i) or (iii), since they only depend on (the existence of) the directional derivative of the proximity operator and can be formulated without requiring more involved geometrical or variational tools.
 
\begin{remark} \label{remark:soc-kl} The conditions in \cref{thm:conn-sec} imply that the merit function $\mer$ satisfies the KL-type inequality stated in assumption \ref{C1} with exponent $\frac12$. Specifically, using the Lipschitz continuity of the mappings $\proxs$ and $\oFnor$, the twice semidifferentiability of $\psi\circ\proxs$, and \cite[Exercise 13.7]{rockafellar2009variational}, we have
\[ \psi(\prox{z}) - \psi(\prox{\bar z}) = \frac12 (\psi \circ\oprox)^{\prime\prime}(\bar z; z-\bar z) + o(\|z-\bar z\|^2) \leq L_p \|z-\bar z\|^2 \]  
for $z \to \bar z$ and for some constant $L_p > 0$. Our claim then follows easily from the strong metric subregularity of the normal map $\oFnor$, see also \cref{lemma:calc-kl-mer} for comparison.
\end{remark}

\subsection{Second-Order Conditions and Bounded Invertibility under Assumption \ref{E2}}

Next, we present a special case of \cref{thm:conn-sec} under the stronger condition \ref{E2}. 
\begin{corollary} \label{cor:soc-diff} Let $f : \Rn \to \R$ be twice continuously differentiable and suppose that assumption \ref{E2} is satisfied at a criticality pair $(\bar x,\bar z) \in \dom{\vp}\times \Rn$. Then, the conditions \rmn{(i)}--\rmn{(vi)} in \cref{thm:conn-sec} are further equivalent to
\begin{itemize}
\item $D\Fnor{\bar z}$ is invertible and $D\prox{\bar z}^\top D\Fnor{\bar z}$ is positive semidefinite,
\end{itemize}
where $D\Fnor{\bar z}$ and $D\prox{\bar z}$ denote the (Fr\'echet) derivative of $\oFnor$ and $\proxs$ at $\bar z$, respectively. 
\end{corollary}
\begin{proof} The differentiability of $\proxs$ and $\oFnor$ is shown in \cref{theorem:epi-semi}. Defining $D := D\prox{\bar z}$  
and following the proof of \cref{prop:soc}, we also have
\[ \iprod{Dh}{D\Fnor{\bar z}h} \geq  \sigma\|Dh\|^2 \quad \forall~h \in \Rn, \]
for some $\sigma > 0$ which is a consequence of the growth condition \eqref{eq:soc-growth1}. Hence, the assertion in \cref{cor:soc-diff} now directly follows from \cref{lemma:pos-pos}.
\end{proof}

We conclude this section and show that assumption \ref{E2} implies CD-regularity of $\oFnor$. This further allows us to fully connect the second-order results in \cref{cor:soc-diff} and condition \ref{D4}.

\begin{proposition} \label{prop:diff-d4} In addition to the assumptions stated in \cref{cor:soc-diff}, let us assume that $\proxs$ is semismooth at $\bar z$. Then, the second-order conditions in \cref{cor:soc-diff} imply that $\oFnor$ is CD-regular at $\bar z$. Moreover, there exist $\bar \sigma, \bar \delta  > 0$ such that we have
\[ \iprod{Dh}{Mh} \geq \bar\sigma \|Dh\|^2 \quad \forall~h \in \Rn, \]
for all $z \in B_{\bar\delta}(\bar z)$ and $M \in \mathcal M^\lambda(z)$ with $M = \nabla^2 f(\prox{z})D + \frac{1}{\lambda}(I-D)$ and $D \in \partial\prox{z}$.
\end{proposition}

Thus, the second-order optimality condition \eqref{eq:soc-suff}, \ref{E2}, and the semismoothness of $\proxs$ are sufficient to guarantee the invertibility assumption \ref{D4}. 

\begin{proof} The differentiability and semismoothness of $\proxs$ imply that $\proxs$ and $\oFnor$ are strictly differentiable at $\bar z$ and Clarke's subdifferential $\partial \prox{\bar z}$ reduces to a singleton $\partial \prox{\bar z} = \{D\prox{\bar z}\} =: \{\bar D\}$, see, e.g., \cite[Theorem 2.6.7]{milzarek2016numerical} and \cite[Exercise 9.25 and Theorem 9.62]{rockafellar2009variational}. In addition, $\partial \Fnor{\bar z}$ coincides with $\mathcal M^\lambda(\bar z)$ and both sets reduce to the singleton $\{\nabla^2 f(\bar x)\bar D + \frac{1}{\lambda}(I-\bar D)\}$ where $\bar x = \prox{\bar z}$. We now show the following continuity property: 
\be \label{eq:dd-cont} \forall~\epsilon > 0 \quad \exists~\delta > 0 \quad \text{such that} \quad \|D(z) - \bar D\| < \epsilon \quad \forall~D(z) \in \partial \prox{z} \quad \text{and} \quad z \in B_\delta(\bar z). \ee
Suppose that this assertion is wrong, i.e., there exists $\epsilon > 0$ and sequences $\{z_k\}$ and $\{D(z_k)\}$, $D(z_k) \in \partial\prox{z_k}$, with $z_k \to \bar z$ and $\|D(z_k) - \bar D\| \geq \epsilon$ for all $k$. Due the local boundedness and upper semicontinuity of Clarke's subdifferential, there then exists a subsequence $\{k_\ell\}$ and $\mathcal D \in \partial \prox{\bar z}$ such that $D(z_{k_\ell}) \to \mathcal D$. Utilizing the strict differentiability of $\proxs$ at $\bar z$ this yields the contradiction $\mathcal D = \bar D$. 

Next, let $z \in B_{\delta}(\bar z)$ and $D(z) \in \partial\prox{z}$ be arbitrary. \cref{prop2-5} implies that the matrices $D(z)$ and $\bar D$ are positive semidefinite with eigenvalues in $[0,1]$. Let $ D(z) = P(z) Q(z) P(z)^\top$ be an eigenvalue decomposition of $D(z)$ with $Q(z) = \diag(q_1(z),...,q_n(z))$ and $q_1(z) \geq ... \geq q_n(z)$. Thanks to the continuity property \eqref{eq:dd-cont} and \cite[Lemma 4.3]{SunSun02} there then exists an orthogonal matrix $P$ such that
\[  \tilde D=PQP^\top, \quad Q = \diag(\lambda_1,...,\lambda_n), \quad \text{and} \quad \|P(z)-P\| = O(\|D(z)-\bar D\|) \]
(after possibly reducing $\delta > 0$).
Let $h\in\R^n$ be arbitrary and let us set $\tilde h=Ph$. Then, our assumption implies
\be \label{eq:lb-tildeG} \langle \tilde h, \tilde G\tilde h \rangle         \geq  \sigma  \|Q\tilde h \|^2, \quad \tilde G := P^\top\bar D^\top  D\Fnor{\bar z}  P,         \ee
for some $ \sigma > 0$. Without loss of generality, let us assume that the eigenvalues of $\tilde D$ satisfy $\lambda_1\geq ...\geq \lambda_\ell >0=\lambda_{\ell+1}=...=\lambda_n$ for some $\ell \geq 1$. (The following proof will also work in the case $\ell=0$). Then there is a constant $\tilde\sigma>0$ such that for any $y \in \mathrm{span}(\{e_1,...,e_{\ell}\})$ we have
\be \label{eq:lb-tildeG-2} \langle y, \tilde G y \rangle \geq \tilde\sigma \|y \|^2. \ee
Here, $\{e_i\}_i$ denotes the standard Euclidean basis. Notice that this property and the constant $\tilde \sigma$ do not depend on the choice of $P$, i.e., condition \eqref{eq:lb-tildeG-2} holds for all $P$ in
\[\mathcal P := \{P \in \R^{n \times n}: P \text{ is orthogonal}; Q = P^\top \tilde D P \text{ is diagonal with } Q_{11} \geq ... \geq Q_{nn}\}.\]  
We now set $M(z) := \nabla^2 f(\prox{z})D(z) + \frac{1}{\lambda}(I-D(z)) \in \mathcal M^\lambda(z)$ and $\tilde G(z) := P(z)^\top D(z)^\top M(z)  P(z)$. Reducing $\delta$ if necessary (this might change $P$), we obtain
\[     \langle  y,\tilde G(z) y \rangle\geq \frac{\tilde \sigma}{2}\|y\|^2 \quad \forall~y \in \mathrm{span}(\{e_1,...,e_{\ell}\}).      \]
Next, let $\tilde h=P(z)h$ with $h \in \Rn$ be arbitrary. As before, we then have $\iprod{h}{D(z)^\top M(z) h} = \langle \tilde h, \tilde G(z) \tilde h \rangle$. Setting $h_1=\sum_{i=1}^\ell \langle \tilde h, e_i \rangle e_i$ and $h_2=\tilde h-h_1$, we can deduce that:
\[     \langle h_1,\tilde G(z)  h_1\rangle \geq \frac{\tilde \sigma}{2}\|h_1\|^2.      \]
Furthermore, setting $B(z) = P(z)^\top\nabla^2f(\prox{z})P(z)$, we have
\begin{align*}
  \iprod{h_2}{\tilde G(z) h_2} &=  \iprod{Q(z)h_2}{B(z)Q(z)h_2} + \frac{1}{\lambda}\iprod{h_2}{Q(z)h_2} - \frac{1}{\lambda}\|Q(z)h_2\|^2
  \\&\geq \frac{1}{\lambda}\left[\iprod{h_2}{Q(z)h_2} - (1 + \lambda\|B(z)\|)\|Q(z)h_2\|^2\right].  
\end{align*}
Since the mappings $z \mapsto \partial\prox{z}$ and $z \mapsto \nabla^2 f(\prox{z})$ are uniformly bounded on $B_\delta(\bar z)$ and $\lambda$ is fixed, there exists $C_B > 0$ such that
\[     \max\{\lambda\|B(z)\|, \|Q(z)B(z)+\frac{1}{\lambda}(I-Q(z))\| \} \leq C_B \quad \forall~z \in {B}_\delta(\bar z).    \]
Let us define $\tilde \epsilon := 1/(1+C_B + \frac{4C_B^2}{\tilde \sigma}+\frac{\tilde\sigma}{4})$. Using the continuity of eigenvalues, $\lambda_i = 0$ for all $i = \ell+1,...,n$, and \eqref{eq:dd-cont}, we can decrease $\delta$ (if necessary) to guarantee 
\[   q_{i}(z) \leq \tilde \epsilon, \quad \forall~i \in \{\ell+1,...,n\}, \quad \forall~z \in B_\delta(\bar z).              \]
Hence, due to $h_2\in\mathrm{span}(\{e_{\ell+1}, ...,e_n\})$, this implies $\iprod{h_2}{Q(z)h_2} \geq \frac{1}{\tilde\epsilon} \|Q(z)h_2\|^2$ and Young's inequality yields
\begin{align*}
  \iprod{h_1}{\tilde G(z)h_2} &= \iprod{h_1}{[Q(z)B(z)+\frac{1}{\lambda}(I-Q(z))]Q(z)h_2} \geq - C_B\|h_1\|\|Q(z)h_2\| \geq -\frac{\tilde \sigma}{8}\|h_1\|^2-  \frac{2C_B^2}{\tilde\sigma}\|Q(z)h_2\|^2.    
\end{align*}
We now obtain
\begin{align*}
  \langle  \tilde h,\tilde G(z) \tilde h \rangle&=  \iprod{h_1}{\tilde G(z)h_1} +2 \iprod{h_1}{\tilde G(z)h_2}+\iprod{h_2}{\tilde G(z)h_2} \geq \frac{\tilde\sigma}{4}(\|h_1\|^2+\|Q(z) h_2\|^2) \geq \frac{\tilde\sigma}{4}\|Q(z)\tilde h\|^2, 
\end{align*}
where we have used the fact $0 \preceq Q(z)\preceq I$. The conclusion follows from the observation $\|D(z)h\|^2=\|Q(z)\tilde h\|^2$.           
\end{proof}
 
\section{A Quasi-Newton Variant of \cref{algo2}}
\label{sec:qnm}

In this section, we discuss a variant of our main algorithm that utilizes quasi-Newton updates to generate approximate and potentially cheaper second-order information while maintaining many of the convergence properties derived in the previous sections. Based on the structure of our approach, there are two different options on how such quasi-Newton updates can be built and used within the algorithm:
\begin{itemize}
\item[{\sf I}.] Approximation of the full (nonsmooth) curvature: $B_k \approx D_k M_k =  D_k [\nabla^2 f(\prox{z_k})D_k + \frac{1}{\lambda}(I-D_k)]$ where $D_k \in \partial \prox{z_k}$. 
\item[{\sf II}.] Approximation of the (smooth) Hessian information: $B_k \approx \nabla^2 f(\prox{z_k})$. 
\end{itemize}
Approximations of type {\sf I} basically lead to nonsmooth quasi-Newton methods which have been studied extensively in the literature, see, e.g., \cite{ip1992local,CheYam92,Qi97,LewOve13,SteThePat17,TheStePat18}. More specialized quasi-Newton techniques have also been investigated for the Moreau envelope $\envs$ and proximal point approaches in \cite{CheFuk99,RauFuk00,BurQia00} and for nonsmooth reformulations of KKT systems in, e.g., \cite{QiJia97,LiYamFuk01}. 

In this section, we propose a quasi-Newton variant of \cref{algo2} that follows the second strategy and only approximates the Hessian $\nabla^2f$. Motivated by its convincing practical performance and high relevance, we will focus on Broyden-Fletcher-Goldfarb-Shanno (BFGS) updates to build the quasi-Newton approximations of $\nabla^2 f$. The full modified algorithm is presented in \cref{algo-4}. Approximations of type {\sf II} exploit the structure of the underlying nonsmooth equation and problem. Related methods that apply similar strategies to smooth components of an underlying problem or smoothing techniques are developed and discussed in, e.g., \cite{Sac85,CheQi94,Che97,SunHan97,ArtBelDonLop14}. In addition, in \cite{HanSun97,ManRun20,mannelhybrid} local convergence properties of two related normal map-based approaches using Broyden-like updates are analyzed.

Our aim in this section is to prove q-superlinear convergence of \cref{algo-4}. Specifically, we want to show that the quasi-Newton approximations generated by \cref{algo-4} are uniformly bounded and satisfy the Dennis-Mor{\'e} condition formulated in \ref{D7}. This then allows us to apply \cref{thm:main-local-conv} to establish fast local convergence. We continue with several more detailed remarks on the Dennis-Mor{\'e} condition and on the boundedness of the matrices $\{B_k\}$:
\begin{itemize}
\item In the nonsmooth setting, Dennis-Mor{\'e} conditions of the form \ref{D7} typically rely on strict differentiability and Lipschitz properties of the underlying nonsmooth equation, see, e.g., \cite{ip1992local,QiJia97,LiYamFuk01,SteThePat17,TheStePat18}. However, as discussed in \cref{rem:crit}, differentiability of $\oFnor$ or $\oprox$ essentially requires the strict complementarity condition to hold. Here, we want to verify the Dennis-Mor{\'e} condition without utilizing strict complementarity and hence, we work with a quasi-Newton scheme that only approximates the smooth Hessian $\nabla^2 f$. 
\item Boundedness of the BFGS updates $\{B_k\}$ is a classical topic that has been investigated thoroughly in the last 50 years, see, e.g., \cite{DenMor74,DenMor77,ByrNocYua87,byrd1989tool}. In order to establish boundedness, the Hessian $\nabla^2 f(\bar x)$ is typically assumed to be positive definite at the limit point $\bar x = \lim_{k \to \infty} x_k$ and one of the following two conditions has to hold: 
\[ {\sum}_{k} \|x_k - \bar x\| < \infty \quad \text{or} \quad \|B_0 - \nabla^2 f(\bar x)\| \; \text{is sufficiently small}. \]
The latter condition requires the initial estimate $B_0$ to be sufficiently close to the true Hessian and is known as a bounded deterioration property. We note that the alternative summability condition ${\sum}_{k} \|x_k - \bar x\| < \infty$ is certainly satisfied if the sequence $\{x_k\}$ converges r-linearly to $\bar x$ which usually can be ensured in the strongly convex case, see, e.g., \cite{Wer79,Rit79,Rit81,ByrNocYua87,byrd1989tool}. Furthermore, in \cite{SteThePat17}, Stella et al., utilize KL-results (for the forward-backward envelope with KL exponent $\frac12$) to justify the assumption ${\sum}_{k} \|x_k - \bar x\| < \infty$. However, their framework requires the strict complementarity condition to hold and the generated quasi-Newton directions need to be gradient-related (in a uniform way). Unfortunately, it is not clear how the latter condition can be verified a priori if the quasi-Newton approximations $\{B_k\}$ are not known to be bounded.  
\end{itemize}

Overall, full global-to-local convergence results for BFGS-type approaches still seem to be fairly limited -- especially in the nonsmooth setting we are considering in this paper -- and require strong and additional assumptions. In the following, we will derive new KL-based results for the BFGS scheme used in \cref{algo-4} that, to some extent, can overcome the mentioned limitations. Let us further note that the framework and results by Li et al., \cite{LiYamFuk01}, are probably closest to our style of analysis. (However, the local convergence results in \cite{LiYamFuk01} are again based on the strict complementarity condition).

\begin{algorithm}[t]
        \caption{A Trust Region-type Normal Map-based Quasi-Newton Method}
         \label{algo-4}
        \begin{algorithmic}[1]  
            \Require Choose an initial point $z_0\in\mathbb{R}^n$ and positive definite matrices $B_0$, and positive number $\lambda$. Choose $\xi > 0$ and sequences $\{\xi_k\} \subset \R_{++}$, $\{\epsilon_k\} \subset \R_{+}$ and set $k=0$.
            \While{$F_{\text{nor}}^{\Lambda}(z_k)\neq0$}
            \State Choose $D_k\in\partial\text{prox}_{\varphi}^{\Lambda}(z_k)$. and set $M_k=B_kD_k+\frac{1}{\lambda}(I-D_k)$. ,                  			\State Run \cref{algo1} with $S = D_k M_k$, $g = D_k \Fnor{z_k}$, $\Delta = \Delta_k$, and $\epsilon = \epsilon_k \geq 0$ returning $\bar q_k = q$;
             \State Set $\bar{s}_k=\bar q_k-\lambda(F_{\text{nor}}^{\lambda}(z_k)+M_k\bar q_k)$ and $s_k=\min\{1,\frac{\Delta_k}{\|\bar{s}_k\|}\}\bar{s}_k$;
             \If{$\rho_k=\frac{\mer(z_k)-\mer(z_k+s_k)}{\mathrm{pred}(z_k,s_k,\Delta_k,\nu_k)} < \eta_1$}
             \State Set $z_{k+1}=z_{k}+s_k$ and update the BFGS approximation
             $$ B_{k+1} = \begin{cases} B_k & \text{if} \; \|d_k\| = 0 \; \text{or} \; d_k^\top y_k < \min\{\xi,\xi_k \|d_k\|^2\}, \\ B_k-\frac{B_kd_kd_k^\top B_k}{d_k^\top B_k d_k}+\frac{y_ky_k^\top}{d_k^\top y_k} & \text{otherwise,} \end{cases} $$ 
$~~~~~~~~~$where $d_k=\prox{z_{k+1}}-\prox{z_k}$ and $y_{k}=\nabla f(\prox{z_{k+1}})-\nabla f(\prox{z_k})$.
             \Else
             \State Set $z_{k+1}=z_k$ and $B_{k+1}=B_k$;
             \EndIf
             \State Update $\Delta_{k+1}$ based on $\rho_k$ by invoking \cref{algo3};
             \State $k\gets k+1$;
            \EndWhile
        \end{algorithmic}
\end{algorithm}

\subsection{Refined Properties of BFGS-Updates and the Dennis-Mor\'{e} Condition} \label{sec:bfgs-dm}

In this subsection, we investigate the BFGS scheme utilized in \cref{algo-4} and show that it can indeed satisfy the Dennis-Mor\'{e} condition \ref{D7} which was an essential component of our convergence analysis in \cref{sec:loc-super-conv}. We first introduce several additional notations in the next definition. 

\begin{defn}
\label{defn8-1}
We define the following terms:
\begin{itemize}
\item[\rmn{(i)}] Let $\mathcal X :=\{k\in\mathbb{N}:k\in\mathcal{S},~\|d_k\|\neq0\} = \mathcal S \cap \mathcal T$ denote the set of all successful iterates with $x_{k+1} \neq x_k$.
\item[\rmn{(ii)}] For any set $S \subset \Rn$, we define $\Delta S:=S-S$. Moreover, $\aff(S)$ denotes the affine hull of $S$ in $\Rn$.
\item[\rmn{(iii)}] For a given set $S \subset \Rn$ and a matrix $A \in \R^{n \times n}$, we define $\mathcal H(S,A):= \inf_{x \in S, \, \|x\| = 1}~\iprod{x}{Ax}$. (Notice that we set $\mathcal H(S,A) = \infty$ in the case $S \cap \{x: \|x\| = 1\} = \emptyset$). 
\end{itemize}
\end{defn}
Let $\{k_i : i \geq 0\}$ enumerate the indices in the set $\mathcal X$. By the definition of $\mathcal X$, it follows
\be \label{eq:sec08:exx} x_{j}=x_{k_{i}+1} \quad \text{and} \quad B_{j}=B_{k_{i}+1} \quad \forall~i, \quad \forall~j \; \text{with} \; k_i+1 \leq j \leq k_{i+1}. \ee 
This structural property will be used frequently throughout this section. In the following, we formulate our main conditions which allow a refined analysis of the BFGS updates.
 
\begin{assumption}
\label{assum8-2} Let $(\bar x,\bar z) \in \dom{\vp} \times \Rn$ be a criticality pair of \eqref{eq1-1}. We assume:
\begin{enumerate}[label=\textup{\textrm{(F.\arabic*)}},topsep=0pt,itemsep=0ex,partopsep=0ex]
\item \label{F1} The sequence $\{x_k\}$ converges to $\bar x$ and has finite length, i.e., $\sum_{k=0}^\infty \|x_{k+1}-x_k\| < \infty$. 
\item \label{F2} The mapping $f$ is twice continuously differentiable near $\bar x$ and $\nabla^2f$ is Lipschitz continuous near $\bar x$ with modulus $L_H$.
\item \label{F3} We have $\liminf_{\epsilon\to0}\mathcal H(\aff(\Delta S_\epsilon),\nabla^2 f(\bar x))>0$
where $S_\epsilon:=\{d\in\mathbb{R}^n: \exists~z\in{B}_\epsilon(\bar z) \, \text{with} \, d=\prox{z}-\prox{\bar z}\}$.
\item \label{F4} The parameters $\xi_k$ satisfy $\xi_k > 0$, $\lim_{k \to \infty} \xi_k = 0$, and $\liminf_{k \to \infty} \xi_k \ln(k) > 0$. 
\end{enumerate}
\end{assumption}

As we have seen in the proof of \cref{thm:kl-convergence}, convergence and finite length of $\{x_k\}$ can be guaranteed under the standard KL-framework.  
Assumption \ref{F3} can be interpreted as a curvature condition. If $\proxs$ is directionally differentiable at $\bar z$, then \ref{F3} clearly implies
\[ \iprod{(\proxs)^\prime(\bar z;h)}{\nabla^2 f(\bar x)(\proxs)^\prime(\bar z;h)} \geq \delta \|(\proxs)^\prime(\bar z;h)\|^2 \]
for all $h \in \Rn$ and for some $\delta > 0$. Thus, \ref{F3} is generally stronger than the second-order sufficient conditions studied in the last section. If $\varphi$ is polyhedral, a stronger connection between \ref{F3} and the second-order optimality conditions can be established and we can demonstrate that the curvature assumption in \ref{F3} can indeed be weaker than positive definiteness of $\nabla^2 f(\bar x)$. A detailed discussion of these observations can be found in \cref{sec:f3}. 

Next, we collect some basic properties of the BFGS updates and show that the skipping mechanism in step 6 of \cref{algo-4} and assumption \ref{F4} can ensure condition \ref{B2}. 

\begin{lemma} \label{lemma:sec08:bfgs} Let the sequence $\{B_k\}$ be generated by \cref{algo-4} and suppose that the initial matrix $B_0 \in \R^{n \times n}$ is symmetric and positive definite. Then, we have: 
\begin{itemize}
\item[\rmn{(i)}] The matrix $B_k$ is symmetric and positive definite for all $k \geq 0$. 
\item[\rmn{(ii)}] In addition, if the assumptions \ref{A1}, \ref{C2}, and \ref{F4} are satisfied, then it follows $\sum_{k=0}^\infty (1+\|B_k\|)^{-1} = \infty$. 
\end{itemize} 
\end{lemma}
\begin{proof} Notice that the matrix $B_k$ is only updated in the case $d_k^\top y_k > 0$. Due to $B_0 \succ 0$, the positive definiteness and symmetry of the matrices $\{B_k\}$ then follows from classical results, see, e.g., \cite{DenMor77}. We continue with a verification of part (ii). In the case $\|d_k\| = 0$ or $d_k^\top y_k < \min\{\xi,\xi_k\|d_k\|^2\}$, we obtain $\|B_{k+1}\| = \|B_k\|$. Otherwise, we have
\begin{align*}d^\top B_{k+1} d & = d^\top B_k d - \frac{(d^\top B_k d_k)^2}{d_k^\top B_k d_k} + \frac{(d^\top y_k)^2}{d_k^\top y_k} \\ & \leq \|B_k\| + \max\left\{\frac{1}{\xi},\frac{1}{\xi_k\|d_k\|^2} \right\} L^2 \|d_k\|^2 = \|B_k\| + L^2 \max\left\{\frac{\|d_k\|^2}{\xi},\frac{1}{\xi_k} \right\} \end{align*}
for all $d \in \Rn$ with $\|d\| = 1$. The continuity of $\proxs$ and assumption \ref{C2} imply that $\{\|d_k\|\}$ is bounded. Hence, utilizing $\liminf_{k \to \infty} \xi_k \ln(k) > 0$, there exist constants $c_D > 0$ and $k^\prime \in \N$ such that $\|B_{k+1}\| \leq \|B_k\| + c_D \ln(k)$ for all $k \geq k^\prime$. Inductively, this yields 
\[ \|B_{k}\| \leq \|B_{k^\prime}\| + c_D \sum_{j = k^\prime}^{k-1} \ln(j) \leq  \|B_{k^\prime}\| + c_D \int_{j = 1}^{k} \ln(j) = \|B_{k^\prime}\| + c_D [k (\ln(k)-1) +1] \]
for all $k \geq k^\prime$. Consequently, we can now select $k^{\prime\prime} \geq k^\prime$ such that $\|B_k\| \leq c_D \cdot k \ln(k) - 1$ for all $k \geq k^{\prime\prime}$ and it follows $\sum_{k=0}^\infty (1+\|B_k\|)^{-1} \geq \sum_{k\geq k^{\prime\prime}}^\infty (1+\|B_k\|)^{-1} \geq \frac{1}{c_D}\sum_{k\geq k^{\prime\prime}}^\infty \frac{1}{k\ln(k)} = \infty$. This finishes the proof of \cref{lemma:sec08:bfgs}. 
\end{proof}

In order to satisfy \ref{F4}, we can simply set $\xi_k = 1/\ln(k)$. Other choices and different skipping techniques are of course possible. \cref{lemma:sec08:bfgs} implies that the adaptive skipping strategy in \cref{algo-4} ensures the non-summability condition \ref{B2} which has been used in our global convergence analysis and in \cref{thm:kl-convergence} (i). This result can be seen as a first building block  allowing us to derive unified global and local results for \cref{algo-4} and transition to fast local convergence without requiring global convexity of the problem. Next, we present a technical proposition that will be used in the proof of \cref{thm8-3}. A proof of \cref{lemma:matrix_extension} can be found in \cref{sec:app:mat-ex}. 

\begin{proposition} \label{lemma:matrix_extension}
Let $\mathcal W \subseteq \Rn$ be a linear subspace and let $H : \Rn \to \mathbb S^n$ be a given continuous function. Suppose there exist $w \in \Rn$ and $\epsilon, \delta > 0$ such that 
\[ d^\top H(x) d \geq \delta \|d\|^2 \quad \forall~d \in \mathcal W, \quad \forall~x \in B_\epsilon(w). \]
Then there is a (not necessarily unique) continuous extension $G : \Rn \to \mathbb S^n$ of $H$ satisfying $G(x)d = H(x)d$ for all $d \in \mathcal W$ and $x \in B_\epsilon(w)$ and
\[ d^\top G(x) d \geq \frac{\delta}{2} \|d\|^2 \quad \forall~d \in \Rn, \quad \forall~x \in B_\epsilon(w). \]
Furthermore, if the function $H$ is Lipschitz continuous on $B_\epsilon(w)$, then $G$ can be chosen as a Lipschitz continuous mapping on $B_\epsilon(w)$ as well. 
\end{proposition}

We now show that many classical properties of the BFGS update can be transferred to the nonsmooth setting considered in this paper under the weaker conditions formulated in \cref{assum8-2}. Our proof is an extension of the seminal analysis presented in \cite{byrd1989tool}.

\begin{thm}
\label{thm8-3}
Suppose that the conditions \ref{F1}--\ref{F3} are satisfied and let the initial matrix $B_0 \in \R^{n \times n}$ be symmetric and positive definite. Then we have:
\begin{itemize}
\item[\rmn{(i)}] The sequences $\{\|B_k\|\}$ and $\{\|B_k^{-1}\|\}$ are both uniformly bounded.
\item[\rmn{(ii)}] Setting $E_k :=B_k-\nabla^2 f(x_k)$, it holds that $\sum_{k\in\mathcal X}{\|E_kd_k\|^2}/{\|d_k\|^2}<\infty$.  
\item[\rmn{(iii)}] It holds that $\sum\|B_{k+1}-B_{k}\|^2<\infty$. 
\end{itemize}
\end{thm}
\begin{proof}
Applying \ref{F2} and \ref{F3}, there are $\epsilon,\delta>0$ such that for every $d \in \aff(\Delta S_\epsilon)$ and $x\in B_\epsilon(\bar x)$, it holds that:
\[     d^\top \nabla^2 f(x) d \geq \delta\|d\|^2.                    \]
In addition, we can assume that $\nabla^2 f$ is Lipschitz continuous on $B_\epsilon(\bar x)$ with constant $L_H$. Hence, by \cref{lemma:matrix_extension} there exists a Lipschitz continuous extension $G$ of $\nabla^2 f$ satisfying $\lambda_{\min}(G(x)) \geq \frac{\delta}{2}$ for all $x \in B_\epsilon(\bar x)$ and 
\[ G(x) d = \nabla^2 f(x) d \quad \forall~d \in \aff(\Delta S_\epsilon), \quad \forall~x \in B_\epsilon(\bar x). \]
Since the sequence $\{x_k\}$ converges to $\bar x$ and we have $\xi_k \to 0$ and $d_k \to 0$, there exists $k_0\in\mathbb{N}$ with 
\[ x_k\in{B}_\epsilon(\bar x), \quad \xi_k \leq \delta, \quad \text{and} \quad \xi_k \|d_k\|^2 \leq \xi \quad \forall~k \geq k_0. \] 
In this case, we further have $d_k \in \Delta S_\epsilon$, and hence, setting $G_k := G(x_k)$, it follows 
\begin{align} \nonumber \|y_k-G_kd_k\| = \|y_k-\nabla^2 f(x_k)d_k\| & = \left\| \int_0^1 [\nabla^2 f(x_{k}+t(x_{k+1}-x_{k})) - \nabla^2 f(x^k)]d_k \, \mathrm{dt} \right\| \\ &  \leq \frac{L_H}{2} \|x_{k+1}-x_k\| \|d_k\| \label{eq:sec08:bd-gk}\end{align}
for all $k \geq k_0$ and $k \in \mathcal X$. This also shows $y_k^\top d_k \geq \delta \|d_k\|^2 \geq \min\{\xi,\xi_k \|d_k\|^2\}$ and thus, we can infer that the full BFGS update is performed for all $k \in \mathcal X \cap [k_0,\infty)$. We now follow the proof of \cite[Theorem 3.2]{byrd1989tool}. Let us define
\begin{align*}
    &\tilde B_k := G_k^{-\half} B_k G_k^{-\half}, \quad \tilde y_k := G_k^{-\half}y_k, \quad \tilde d_k := G_k^{\half}d_k, 
    \quad \cos(\vartheta_k) := \frac{\tilde d_k^\top \tilde B_k \tilde d_k}{\|\tilde B_k \tilde d_k\|\|\tilde d_k \|}, \quad q_k := \frac{\tilde d_k^\top \tilde B_k \tilde d_k}{\|\tilde d_k\|^2}.
\end{align*}
Recall that $\{k_i : i \geq 0\}$ enumerates the indices in the set $\mathcal X \cap [k_0,\infty)$. Then, we have
\begin{align*} \tilde B_{k_{i+1}} & = G_{k_{i+1}}^{-\half} \left[ B_{k_i}  - \frac{B_{k_i}d_{k_i}d_{k_i}^\top B_{k_i}}{d_{k_i}^\top B_{k_i}d_{k_i}} + \frac{y_{k_i}y_{k_i}^\top}{d_{k_i}^\top y_{k_i}} \right] G_{k_{i+1}}^{-\half} \\ & = \underbracket{\begin{minipage}[t][5.7ex][t]{10ex}\centering$\displaystyle G_{k_{i+1}}^{-\half} G_{k_i}^\half$\end{minipage}}_{=: \, P_{k_i}} \underbracket{\begin{minipage}[t][7ex][t]{35ex}\centering$\displaystyle\left[ \tilde B_{k_i}  - \frac{\tilde B_{k_i}\tilde d_{k_i}\tilde d_{k_i}^\top \tilde B_{k_i}}{\tilde d_{k_i}^\top \tilde B_{k_i}\tilde d_{k_i}} + \frac{\tilde y_{k_i} \tilde y_{k_i}^\top}{\tilde d_{k_i}^\top \tilde y_{k_i}} \right]$\end{minipage}}_{=: \, Q_{k_i}} G_{k_i}^\half G_{k_{i+1}}^{-\half}, \end{align*}
where we used $B_{k+1} = B_k$, $d_{k+1} = d_k$, and $y_{k+1} = y_k$ for all $k \in \mathcal X^{\mathsf C} \cap [k_0, \infty)$, see \eqref{eq:sec08:exx}. As in \cite{byrd1989tool}, we will now bound the eigenvalues of $B_k$ (and $\tilde B_k$) using the mapping $\Psi(B) := \tr(B) - \ln(\det(B))$. Here, we obtain
\[ \Psi(\tilde B_{k_{i+1}}) = \tr(P_{k_i}^\top P_{k_i} \cdot Q_{k_i}) - \ln(\det(P_{k_i}^\top P_{k_i})) - \ln(\det(Q_{k_i})). \]
for all $i \geq 0$. Setting $\Xi_{k_i} := G_{k_i}^\half [G_{k_{i+1}}^{-1} - G_{k_i}^{-1}]G_{k_i}^\half$ and using $P_{k_i}^\top P_{k_i} = I + \Xi_{k_i}$ and $Q_{k_i} \succ 0$ and Neumann's trace inequality, it follows  
\[ \Psi(\tilde B_{k_{i+1}}) \leq  (1+ \|\Xi_{k_i}\|) \tr(Q_{k_i}) - \ln(\det(P_{k_i}^\top P_{k_i})) - \ln(\det(Q_{k_i})) \quad \forall~i \geq 0. \]
As in \cite[Theorem 3.2]{byrd1989tool}, it holds that 
\begin{align*} \tr(Q_{k_i}) & = \tr(\tilde B_{k_i}) - \frac{\|\tilde B_{k_i}\tilde d_{k_i}\|^2}{\tilde d_{k_i}^\top \tilde B_{k_i}\tilde d_{k_i}} + \frac{\|\tilde y_{k_i}\|^2}{\tilde y_{k_i}^\top \tilde d_{k_i}} = \tr(\tilde B_{k_i}) + \frac{\|\tilde y_{k_i}\|^2}{\tilde y_{k_i}^\top \tilde d_{k_i}} - \frac{q_{k_i}}{\cos^2(\vartheta_{k_i})}, \\ \ln(\det(Q_{k_i})) & = \ln(\det(\tilde B_{k_i})) + \ln\left(\frac{\tilde y_{k_i}^\top \tilde d_{k_i}}{\tilde d_{k_i}^\top \tilde B_{k_i}\tilde d_{k_i}}\right) = \ln(\det(\tilde B_{k_i})) + \ln\left(\frac{\tilde y_{k_i}^\top \tilde d_{k_i}}{\|\tilde d_{k_i}\|^2}\right) - \ln(q_{k_i}). \end{align*}
Defining $\omega_1(t) := t - \ln(t)$, we have $\omega_1(t) \geq \ln(t)$ for all $t > 0$. Let $\lambda_1,...,\lambda_n > 0$ denote the eigenvalues of $\tilde B_{k_i}$, then it holds that $\Psi(\tilde B_{k_i}) = \sum_{j=1}^n \omega_1(\lambda_j)$ and we can infer
\begin{align*} \Psi(\tilde B_{k_{i+1}}) & \leq (1+\|\Xi_{k_i}\|) \cdot \Psi(\tilde B_{k_i}) + \|\Xi_{k_i}\|\sum_{j=1}^n \ln(\lambda_j) -  \ln\left(\frac{\tilde y_{k_i}^\top \tilde d_{k_i}}{\|\tilde d_{k_i}\|^2}\right) + (1+\|\Xi_{k_i}\|)\left[ \frac{\|\tilde y_{k_i}\|^2}{\tilde y_{k_i}^\top \tilde d_{k_i}} - \frac{q_{k_i}}{\cos^2(\vartheta_{k_i})} \right] \\ & \hspace{4ex} + \ln(q_{k_i}) - \ln(\det(P_{k_i}^\top P_{k_i}))  \\ & \leq (1+2\|\Xi_{k_i}\|) \cdot \Psi(\tilde B_{k_i}) + (1+\|\Xi_{k_i}\|) \cdot \frac{\|\tilde y_{k_i}\|^2}{\tilde y_{k_i}^\top \tilde d_{k_i}} -1 - \ln\left(\frac{\tilde y_{k_i}^\top \tilde d_{k_i}}{\|\tilde d_{k_i}\|^2}\right) + \ln(\cos^2(\vartheta_{k_i})) \\& \hspace{4ex}+ \left[ 1 - \frac{q_{k_i}}{\cos^2(\vartheta_{k_i})} + \ln\left( \frac{q_{k_i}}{\cos^2(\vartheta_{k_i})}\right) \right] - \ln(\det(P_{k_i}^\top P_{k_i})). \end{align*}
Next, the estimate in \eqref{eq:sec08:bd-gk} yields
\be \label{eq:yd-1}\frac{\|\tilde y_k - \tilde d_k\|}{\|\tilde d_k\|} = \frac{\|G_k^{-\frac12}[y_k - G_k d_k]\|}{\|\tilde d_k\|} \leq L_H \|G_k^{-\frac12}\| \|x_{k+1}-x_k\| \cdot \frac{\|d_k\|}{\|\tilde d_k\|} \leq \frac{L_H}{2\delta} \|x_{k+1}-x_k\| \ee
for all $k \in \mathcal X \cap [k_0, \infty)$. Hence, following the proof of \cite[Theorem 3.2]{byrd1989tool} and setting $\epsilon_k := \|x_{k+1}-x_k\|$, we can show
\[ \frac{\tilde y_{k}^\top \tilde d_{k}}{\|\tilde d_{k}\|^2} \geq 1 - \frac{L_H\epsilon_k}{2\delta} \quad \text{and} \quad \frac{\|\tilde y_k\|^2}{\tilde y_k^\top \tilde d_k} \leq \left[1+\frac{L_H\epsilon_k}{2\delta} \right]^2 \frac{\|\tilde d_k\|^2}{\tilde y_k^\top \tilde d_k}. \]
Due to $\epsilon_k \to 0$ there then exists a constant $\bar L > (2\delta)^{-1}L_H$ with ${\|\tilde y_k\|^2}/{(\tilde y_k^\top \tilde d_k)}  \leq 1+\bar L\epsilon_k$ for all $k \in \mathcal X \cap [k_0, \infty)$ sufficiently large. In addition, using the standard logarithm inequality $\ln(1-t) \geq -\frac{t}{1-t}$, $t < 1$, it follows $\ln(1-(2\delta)^{-1}L_H \epsilon_k) \geq -2\bar L\epsilon_k$ for all $k$ with $\bar L\epsilon_k < \half$. Next, we derive estimates for the terms $\|\Xi_{k_i}\|$ and $\ln(\det(P_{k_i}^\top P_{k_i}))$. Let $L_G$ denote the Lipschitz constant of the mapping $G$, then by Banach's perturbation lemma, we have
\begin{align*} \|\Xi_{k_i}\| \leq \|G_{k_i}\| \|G_{k_{i+1}}^{-1} - G_{k_{i}}^{-1}\| & \leq [\|G(\bar x)\| + L_G \|x_{k_i}-\bar x\|] \cdot \frac{\|G_{k_i}^{-1}\|^2\|G_{k_{i+1}}-G_{k_i}\|}{1-\|G_{k_i}^{-1}[G_{k_{i+1}}-G_{k_i}]\|} \\ & \leq \frac{4L_G[\|G(\bar x)\| + L_G \epsilon]}{\delta^2} \cdot \frac{\|x_{k_{i+1}}-x_{k_i}\|}{1-\|G_{k_i}^{-1}[G_{k_{i+1}}-G_{k_i}]\|}. \end{align*}
Using $\|x_{k_{i+1}}-x_{k_i}\| = \|x_{k_i+1}-x_{k_i}\| = \epsilon_{k_i} \to 0$ there exists another constant $\bar\xi$ such that $\|\Xi_{k_i}\| \leq \bar \xi \epsilon_{k_i}$ for all $i$ sufficiently large. Furthermore, let $\sigma_1,...,\sigma_n$ denote the eigenvalues of $P_{k_i}^\top P_{k_i}$. Then using $P_{k_i}^\top P_{k_i} = I + \Xi_{k_i}$, it holds that $\sigma_j \geq 1 - \|\Xi_{k_i}\|$ for all $j$. Hence, applying Bernoulli's inequality, we obtain
\[ \det(P_{k_i}^\top P_{k_i}) \geq (1-\|\Xi_{k_i}\|)^n \geq 1 - n\|\Xi_{k_i}\| \geq 1- n \bar \xi \epsilon_{k_i} > 0 \]
for all $i$ sufficiently large. As before this implies $\ln(\det(P_{k_i}^\top P_{k_i})) \geq -2n\bar \xi \epsilon_{k_i}$ for all $i$ with $n\bar \xi \epsilon_{k_i} \leq \frac12$. Combining the last estimates and results, there exists $i_0$ such that we have
\[ \Psi(\tilde B_{k_{i+1}}) \leq (1+2\bar\xi\epsilon_{k_i}) \cdot \Psi(\tilde B_{k_i}) + (3\bar L + (2n+2)\bar\xi) \epsilon_{k_i} + \ln(\cos^2(\vartheta_{k_i})) - \omega_{k_i} \]
for all $i \geq i_0$, where $\omega_{k_i} = \omega_2(q_{k_i}/\cos^2(\vartheta_{k_i}))$ and $\omega_2(t) := t - 1 - \ln(t)$. Using the non-negativity of the mapping $\omega_2$ and the summability of $\{\epsilon_{k_i}\}$, this shows that $\{\Psi(\tilde B_{k_i})\}$ is a deterministic supermartingale-type sequence. In particular, due to \cite[Proposition A.31]{Ber16} and $\Psi(B) \geq n$ (see, e.g., \cite{byrd1989tool}), it follows that $\{\Psi(\tilde B_{k_i})\}$ is bounded and we have
\be \label{eq:sec08:sum-qn} \sum_{i = i_0}^\infty \omega_{k_i} - \ln(\cos^2(\vartheta_{k_i})) < \infty. \ee
As a consequence, the eigenvalues of the matrices $\{\tilde B_{k_i}\}$ are contained in a compact interval $J \subset \R_{++}$. Since the matrices $\{G_k\}$ satisfy $G_k \succeq \frac{\delta}{2}I$ and $\|G_k\| \leq \|G(\bar x)\| + L_G \epsilon =: c_{\bar G}$ for all $k \geq k_0$, this establishes uniform boundedness of $\{\|B_k\|\}$ and $\{\|B_k^{-1}\|\}$. We further note that the condition \eqref{eq:sec08:sum-qn}, implies $\cos(\vartheta_{k_i}) \to 1$ and $q_{k_i} \to 1$ as $i \to \infty$. Moreover, following equation (3.27) in \cite{byrd1989tool}, it holds that
\begin{align*} e_k^2 := \frac{\|E_kd_k\|^2}{\|d_k\|^2} = \frac{\|[B_k - G_k]d_k\|^2}{\|G_k^{-\half}\tilde d_k\|^2} \leq c_{\bar G}^2 \cdot \frac{\|[\tilde B_{k}-I]\tilde d_{k}\|^2}{\|\tilde d_k\|^2} = c_{\bar G}^2 \cdot \left[ \frac{q_k^2}{\cos^2(\vartheta_k)} - 2q_k + 1 \right]. \end{align*}
Since we have $q_{k_i}^2 / \cos^2(\vartheta_{k_i}) - 2q_{k_i}+1 \to 0$ and $q_{k_i} / \cos^2(\vartheta_{k_i}) \geq q_{k_i}$, we can apply Lemma 2.3 and the estimate (2.16) in \cite{RodNes21} to obtain 
\[ \omega_{k_i} - \ln(\cos^2(\vartheta_{k_i})) = \frac{q_{k_i}}{\cos^2(\vartheta_{k_i})} - \ln(q_{k_i}) - 1 \geq \frac14 \left[ \frac{q_{k_i}^2}{\cos^2(\vartheta_{k_i})} - 2q_{k_i} + 1 \right] \geq \frac{e_{k_i}^2}{4c_{\bar G}^2} \]
for all $i$ sufficiently large. Hence, part (ii) follows from \eqref{eq:sec08:sum-qn}. Next, the uniform boundedness of $\{\|\tilde B_k^{-1}\|\}$ ensures the existence of a constant $\sigma_{\tilde B}$ such that $\tilde d_k^\top \tilde B_k \tilde d_k \geq \sigma_{\tilde B}\|\tilde d_k\|^2$ (for all $k$ sufficiently large. Let us consider an arbitrary index $k \in \mathcal X \cap [k_0, \infty)$. Then, reusing \eqref{eq:yd-1} and our earlier estimates, we have
\begin{align*}
\mathcal E_{1,k}^2 & := \left  \| \frac{\tilde y_k\tilde y_k^\top}{\tilde d_k^\top \tilde y_k} - \frac{\tilde y_k\tilde y_k^\top}{\tilde d_k^\top \tilde B_k \tilde d_k} \right\|_F^2 = \frac{\|\tilde y_k\|^4}{(\tilde d_k^\top \tilde y_k)^2} \frac{((\tilde B_k \tilde d_k-\tilde y_k)^\top \tilde d_k)^2}{(\tilde d_k^\top \tilde B_k \tilde d_k)^2} \\ & \leq \frac{2(1+\bar L\epsilon_k)^2}{\sigma_{\tilde B}^2} \left [ \frac{((\tilde d_k - \tilde y_k)^\top \tilde d_k)^2}{\|\tilde d_k\|^4} +  \frac{(\tilde d_k^\top[\tilde B_k - I] \tilde d_k)^2}{\|\tilde d_k\|^4} \right] \leq \frac{4}{\sigma^2_{\tilde B}} \left[ \frac{L_H^2}{4\delta^2} \epsilon_k^2 + \frac{4}{\delta^2} e_k^2 \right] \leq \frac{\max\{L_H^2,16\}}{\sigma_{\tilde B}^2 \delta^2} [\epsilon_k^2 + e_k^2] \end{align*} 
for all $k \in \mathcal X$ sufficiently large. Furthermore, due to $\|\tilde y_k\tilde y_k^\top-\tilde d_k\tilde d_k^\top\|_F^2 \leq (\|(\tilde y_k-\tilde d_k)(\tilde y_k-\tilde d_k)^\top\|_F + 2 \|(\tilde y_k-\tilde d_k)\tilde d_k^\top\|_F)^2 \leq 2 \|\tilde y_k-\tilde d_k\|^4 + 4 \|\tilde y_k-\tilde d_k\|^2 \|\tilde d_k\|^2$, it holds that
\[ \mathcal E_{2,k}^2 := \left \|\frac{\tilde y_k\tilde y_k^\top}{\tilde d_k^\top \tilde B_k \tilde d_k} -  \frac{\tilde d_k\tilde d_k^\top}{\tilde d_k^\top \tilde B_k \tilde d_k}  \right\|_F^2 \leq \frac{2}{\sigma_{\tilde B}^2} \left[ \frac{\|\tilde y_k - \tilde d_k\|^4}{\|\tilde d_k\|^4} + 2 \frac{\|\tilde y_k - \tilde d_k\| ^2}{\|\tilde d_k\|^2} \right] = O(\epsilon_k^2) \]
and similarly, we obtain
\begin{align*}
    \mathcal E_{3,k}^2 &:= \left \|\frac{\tilde d_k\tilde d_k^\top}{\tilde d_k^\top \tilde B_k \tilde d_k} -  \frac{\tilde B_k\tilde d_k\tilde d_k^\top \tilde B_k}{\tilde d_k^\top \tilde B_k \tilde d_k}  \right\|_F^2 \leq \frac{2}{\sigma_{\tilde B}^2} \left[ \frac{\|[\tilde B_k - I] \tilde d_k\|^4}{\|\tilde d_k\|^4} + 2 \frac{\|[\tilde B_k - I] \tilde d_k\| ^2}{\|\tilde d_k\|^2} \right]  \leq \frac{2}{\sigma^2_{\tilde B}} \left[ \frac{16}{\delta^4} e_k^4 + \frac{8}{\delta^2} e_k^2\right].
\end{align*}
Finally, combining the last results, we can establish  
\begin{align*} \|B_{k+1}-B_k\|_F^2 = \left \| G_k^\half \left[ \frac{\tilde y_k \tilde y_k^\top}{\tilde d_k^\top \tilde y_k} - \frac{\tilde B_k\tilde d_k \tilde d_k^\top \tilde B_k}{\tilde d_k^\top \tilde B_k \tilde d_k} \right] G_k^\half \right \|_F^2 & \leq \tr(G_k)^2 [\mathcal E_{1,k} + \mathcal E_{2,k} + \mathcal E_{3,k}]^2 \\ &  \leq 3nc_{\bar G}^2 \cdot [\mathcal E_{1,k}^2 + \mathcal E_{2,k}^2 + \mathcal E_{3,k}^2]. \end{align*}
Utilizing $B_{k+1} = B_k$ for all $k \notin \mathcal X$, the statement in part (iii) now follows from (ii) and \ref{F1}. 
\end{proof}
\begin{remark} Our analysis extends the existing classical results for BFGS updates (for smooth problems) provided in \cite{DenMor77,byrd1989tool} that are based on the stronger convergence condition 
\[ \sum \|x_k - \bar x\| < \infty. \] 
In contrast, in \cref{thm8-3} we have shown that many fundamental properties of the BFGS scheme still hold under the significantly weaker finite length assumption $\sum \|x_{k+1}-x_k\| < \infty$. This generalization is mainly achieved by considering an adaptive rescaling of $B_k$ based on $G_k = G(x_k)$ rather than on the fixed matrix $G_* := G(\bar x)$. 
\end{remark} 

Next, we establish the key result of this section. Specifically, based on the structural properties derived in \cref{thm8-3}, we show that that the Dennis-Mor{\'e}-type condition formulated in assumption \ref{D7} is satisfied when using BFGS approximations of the Hessian $\nabla^2 f$ as in \cref{algo-4}. Our result will allow us to link KL-theory to superlinear convergence of the quasi-Newton method. 

Recall that the indices $\{j_i\}$ enumerate the elements of the set $\cS$. Furthermore, the term $\Gamma_\ell$, which appeared in the proof of \cref{thm:kl-convergence}, was defined as follows $\Gamma_\ell := \sum_{i = \ell}^\infty \chi(z_{j_i})$. 

\begin{thm}
\label{thm8-5} Suppose that the conditions \ref{E1} and \ref{F1}--\ref{F3} are satisfied and assume that the sequence $\{\Gamma_\ell\}$ converges q-linearly to zero. Then, we have
\[ \sum_{k \in \cS} \frac{\|E_k(x_k-\bar x)\|^2}{\|z_k - \bar z\|^2} < \infty \quad \text{and} \quad \lim_{k \to \infty} \frac{\|E_k(x_k-\bar x)\|}{\|z_k - \bar z\|} = 0. \]
\end{thm}
\begin{proof} First, the q-linear convergence of $\{\Gamma_\ell\}$ yields $\chi(z_k) \to 0$ as $\mathcal S \ni k \to \infty$. As in \cref{remark:csq-stat}, condition \ref{F1} then allows us to infer $z_k \to \bar z$ as $k \to \infty$. Furthermore and as discussed earlier, the assumptions \ref{E1} and \ref{F3} imply that the second-order sufficient conditions hold at $\bar z$ and by \cref{thm:conn-sec}, the normal map $\oFnor$ is strongly metrically subregular at $\bar z$ for $0$. Since $\nabla f$ is locally Lipschitz continuous near $\bar x$ by \ref{F2}, $\chi$ is also locally Lipschitz continuous near $\bar z$. Thus, there are constants $\kappa_\chi, L_{\chi} > 0$ such that we have
\be \label{eq:sec08:chi-met} \kappa_\chi \|z_k - \bar z\| \leq \chi(z_k) \leq L_\chi \|z_k-\bar z\| \ee
for all sufficiently large $k$. In addition, there exists $\bar \gamma \in (0,1)$ with $\Gamma_{\ell+1} \leq \bar \gamma \Gamma_\ell$ for all $\ell$ sufficiently large. Let us now define $r_k := \|E_k(x_k-\bar x)\| / \Gamma_{n_\cS(k)}$. Let us first consider an index $j_\ell \notin \mathcal X$. Then, by \eqref{eq:sec08:exx}, we have $x_{j_{\ell+1}}= x_{j_\ell+1}=x_{j_\ell}$ and $B_{j_{\ell+1}}=B_{j_{\ell}+1} = B_{j_\ell}$ and it follows
\[r_{j_\ell}^2 = \frac{\|E_{j_\ell}(x_{j_\ell}-\bar x)\|^2}{\Gamma_{n_{\mathcal S}(j_\ell)}^2} = \frac{\|E_{j_{\ell+1}}(x_{j_{\ell+1}}-\bar x)\|^2}{\Gamma_\ell^2} = \frac{\Gamma_{\ell+1}^2}{\Gamma_\ell^2} \cdot r_{j_{\ell+1}}^2 \leq \bar \gamma^2 r_{j_{\ell+1}}^2. \]
For the case $j_\ell \in \mathcal X$, we first establish
\be \label{eq:sec08:chi-02} \Gamma_\ell \geq \kappa_\chi [\|z_{j_{\ell+1}}-\bar z\| + \|z_{j_\ell}-\bar z\|] \geq \kappa_\chi \|z_{j_{\ell+1}}-z_{j_\ell}\| =  \kappa_\chi \|z_{j_{\ell}+1}-z_{j_\ell}\| \geq  \kappa_\chi \|x_{j_\ell+1}-x_{j_\ell}\|, \ee
which follows from \eqref{eq:sec08:chi-met} (and also holds for $j_\ell \notin \mathcal X$) provided $\ell$ is sufficiently large. Using $x_{j_\ell+1} = x_{j_{\ell+1}}$, $E_{j_\ell+1} = E_{j_{\ell+1}}$, and Young's inequality, we then obtain 
\begin{align} \nonumber r_{j_\ell}^2 & \leq \left[    \frac{\|E_{j_\ell}(x_{j_\ell}-x_{j_\ell+1})\|}{\kappa_\chi\|x_{j_\ell+1}-x_{j_\ell}\|}+\frac{\|(E_{j_\ell}-E_{j_\ell+1})(x_{j_\ell+1}-\bar x)\|}{\kappa_\chi \|z_{j_\ell+1}-\bar z\|} +\frac{\|E_{j_\ell+1}(x_{j_\ell+1}-\bar x)\|}{\Gamma_\ell} \right]^2 \\ \nonumber & \leq \left[ \frac{1}{\kappa_\chi} (e_{j_\ell} + \|B_{j_\ell+1}-B_{j_\ell}\| + L_H\| x_{j_\ell+1}-x_{j_\ell} \|) + \bar \gamma r_{j_{\ell+1}} \right]^2 \\ \label{eq:sec08:chi-03} & \leq \frac{3}{\kappa_\chi^2} \frac{1}{1-\bar\gamma} [e_{j_\ell}^2 + \|B_{j_\ell+1}-B_{j_\ell}\|^2 + L_H^2\| x_{j_\ell+1}-x_{j_\ell} \|^2] + \bar\gamma r_{j_{\ell+1}}^2, 
\end{align}
for all $\ell$ sufficiently large, where $L_H$ denotes the local Lipschitz constant of the Hessian $\nabla^2 f$. Moreover, applying \eqref{eq:sec08:chi-02}, we have
\[ r_{j_\ell} \leq \frac{\|E_{j_\ell}(x_{j_\ell}-\bar x)\|}{\kappa_\chi \|z_{j_\ell}-\bar z\|} \leq \frac{1}{\kappa_\chi} \|E_{j_{\ell}}\| \leq \frac{1}{\kappa_\chi} [\|B_{j_{\ell}}\| + \|\nabla^2 f(x_{j_\ell})\|]. \]
Hence, using \cref{thm8-3} (i), $x_k \to \bar x$, and the continuity of $\nabla^2 f$, the sequence $\{r_{j_\ell}\}$ needs to be bounded from above, i.e., there exists $\bar r$ such that $r_{j_\ell} \leq \bar r$ for all $\ell$. Next, let us choose a sufficiently large index $\ell^\prime \in \N$ such that the latter estimates hold for all $\ell \geq \ell^\prime$. Summing the expression \eqref{eq:sec08:chi-03} for $\ell^\prime \leq \ell \leq m-1$, it follows
\begin{align*} (1-\bar\gamma) \sum_{\ell=\ell^\prime}^{m-1} r_{j_{\ell}}^2 + \bar \gamma [ r_{j_{\ell^\prime}}^2 - \bar r^2]  & \leq \sum_{\ell=\ell^\prime}^{m-1} [r_{j_{\ell}}^2 -\bar \gamma  r_{j_{\ell+1}}^2] \\ & \leq \frac{3}{\kappa_\chi^2(1-\bar\gamma)} \sum_{\ell \geq \ell^\prime, \, j_\ell \in \mathcal X} [e_{j_\ell}^2 + \|B_{j_\ell+1}-B_{j_\ell}\|^2 + L_H^2\| x_{j_\ell+1}-x_{j_\ell} \|^2].   \end{align*}
Consequently, taking the limit $m \to \infty$ and appyling \cref{thm8-3} (ii) and (iii) and \ref{F1}, we can infer $\sum_{\ell} r_{j_{\ell}}^2 < \infty$. Finally, the q-linear convergence of $\{\Gamma_\ell\}$ and \eqref{eq:sec08:chi-met} yield
\[ \frac{\|E_{j_\ell}(x_{j_\ell}-\bar x)\|}{\|z_{j_{\ell}}-\bar z\|} \leq L_\chi \frac{\|E_{j_\ell}(x_{j_\ell}-\bar x)\|}{\chi(z_{j_\ell}) }= L_\chi \frac{\|E_{j_\ell}(x_{j_\ell}-\bar x)\|}{\Gamma_\ell - \Gamma_{\ell+1}}\leq \frac{L_\chi}{1-\bar\gamma} r_{j_\ell}. \]
Since $z_k$, $x_k$, and $E_k$ remain unchanged for $k\notin\mathcal{S}$, this finishes the proof of \cref{thm8-5}. 
\end{proof}

Let us note that the condition derived in \cref{thm8-5} is slightly different from the alternative and more standard Dennis-Mor{\'e} condition
\[ \lim_{k \to \infty} \frac{\|E_k(\prox{z_k+s_k}-\prox{z_k})\|}{\|s_k\|} = 0, \] 
which appears frequently in the local convergence analysis of classical trust region-type methods \cite{Pow74,Yua85,ByrKhaSch96,CarGouToi11} (for the case $\vp\equiv 0$). In our situation, this condition can only be guaranteed  for successful iterations $k \in \mathcal X$ and does not necessarily hold for \textit{all} iterations. We resolve this technical restriction and directly work with the Dennis-Mor\'{e}-type condition stated in \ref{D7} and \cref{thm8-5}.

\subsection{Summary and Superlinear Convergence}

We are now in the position to fully connect our results established in \cref{sec:KL_conv}, \cref{sec:loc-super-conv}, and \cref{sec:bfgs-dm}. In the following, we give a schematic overview of our different global and local results for \cref{algo-4} that illustrates how these results interact with each other leading to superlinear convergence.

\begin{itemize}[leftmargin=6ex]
\item[\textsf{A.}] \textit{Global Convergence and Standard KL.} First, suppose that the conditions \ref{A1}--\ref{A2}, \ref{B1}, \ref{C1}--\ref{C3}, \ref{F4}, and $\Delta_{\min} > 0$ are satisfied. 
\begin{itemize}
\item[$\bullet$] \cref{lemma:sec08:bfgs} (ii) then implies that assumption \ref{B2} has to hold. Consequently, \cref{thm:global_conv} and \cref{thm:kl-convergence} (i) are applicable and we can infer that $\{(z_k,x_k)\}$ converges to the criticality pair $(\bar z,\bar x)$ and the sequence $\{x_k\}$ has finite length, i.e., condition \ref{F1} is satisfied.
\end{itemize}
\item[\textsf{B}.] \textit{Strong KL and Dennis-Mor\'{e}}. In addition to the conditions in \textsf{A}, let the assumptions \ref{D3} and \ref{F2}--\ref{F3} hold. 
\begin{itemize}
\item[$\bullet$] \cref{thm8-3} is then applicable and the BFGS approximations $\{B_k\}$ need to stay in a compact subset of $\Spp$. Thus, assumption \ref{B3} is satisfied. Furthermore, since semismoothness implies semidifferentiability of $\proxs$ at $\bar z$, see, e.g., \cite{QiSun93}, assumption \ref{D3} and \ref{F3} imply that the second-order sufficient conditions formulated in \cref{thm:conn-sec} hold and thus, by \cref{remark:soc-kl}, the merit function $\mer$ satisfies the KL-type inequality stated in \cref{sec:KL_conv} with exponent $\half$. Since the matrices $\{B_k\}$ are now bounded, the stronger convergence results in \cref{thm:kl-convergence} (ii) are applicable guaranteeing q-linear convergence of $\{\Gamma_\ell\}$. \cref{thm8-5} hen implies that the Dennis-Mor\'{e} condition \ref{D7} is satisfied. (Notice that assumption \ref{E1} is not required here thanks to \cref{theorem:epi-semi}). Finally, the uniform positive definiteness of the BFGS matrices $\{B_k\}$ and \cref{lemma:pos-pos} ensure assumption \ref{D4}. 
\end{itemize}
\item[\textsf{C}.] \textit{Superlinear Convergence}. In addition to the assumptions in \textsf{A}--\textsf{B}, suppose that \ref{D5} and \ref{D6} hold. 
\begin{itemize}
\item[$\bullet$] All prerequisites of \cref{thm:main-local-conv} are fulfilled and hence, we can infer q-superlinear convergence of $\{z_k\}$. 
\end{itemize}
\end{itemize}

This discussion demonstrates that the quasi-Newton method is highly compatible with the KL-theory which allows to establish full and satisfactory global-to-local convergence results. We summarize our observations in the next theorem. To the best of our knowledge, this is the first result which fully links the KL-framework, the Dennis-Mor{\'e} condition, and superlinear convergence of BFGS-type schemes. 

\begin{thm} \label{thm:sec08:qn}
Let the conditions \ref{A1}--\ref{A2}, \ref{B1}, \ref{C1}--\ref{C3}, \ref{D3}, \ref{D5}--\ref{D6}, \ref{F2}--\ref{F4}, and $\Delta_{\min} > 0$ be satisfied and assume that \cref{algo-4} does not terminate after finitely many steps. Then, we have:
\begin{itemize}
\item Every trust region step is eventually successful and the sequence $\{z_k\}$ converges q-superlinearly to $\bar z$. 
\item If the proximal mapping $\proxs$ is $\beta$-order semismooth at $\bar z$ and if the error function $\mathcal E$ in \ref{D5} satisfies $\mathcal E(h) = O(\|h\|^{1+\beta})$ as $h \to 0$, then it further follows
\be \label{eq:rate-sum} \sum_k \frac{\|z_{k+1} - \bar z\|^2}{\|z_k - \bar z\|^2}  < \infty. \ee
\end{itemize}
\end{thm}
\begin{proof} The first part of \cref{thm:sec08:qn} follows from our discussion in \textsf{A}--\textsf{C}. In order to establish the summability result in the second statement, we can mimic \eqref{eq:th-sq} to obtain
\[ {\|z_{k+1}-\bar z\|}\leq O(\|z_k-\bar z\|^{1+\beta}) + \kappa_M {\|E_k(x_k - \bar x)\|}. \]
Dividing both sides by $\|z_k-\bar z\|$, taking squares, and using the q-superlinear convergence of $\{z_k\}$ this yields \eqref{eq:rate-sum}. 
 \end{proof}

\begin{remark} Given the q-superlinear convergence of $\{z_k\}$ as established in \cref{thm:sec08:qn}, it is possible to derive additional properties of the BFGS updates $\{B_k\}$. In particular, following \cite{ren1983convergence,Sto84}, we expect the matrices $B_k$ to converge to some symmetric, positive definite matrix $B_*$ (which can be different from $\nabla^2 f(\bar x)$). A detailed verification of this observation is left for future work. 
\end{remark}

\begin{remark} The summability condition \eqref{eq:rate-sum} can also be used to further specify the rate of convergence. Specifically, due to  \eqref{eq:rate-sum}, for all $\eta > 0$ there exists $k_0 \in \N$ such that $\sum_{j \geq k_0}{\|z_{j+1}-\bar z\|^2}/{\|z_j-\bar z\|^2} \leq \eta$. 
 
Applying the arithmetic-geometric mean inequality, it then follows
\be \label{eq:rate-rate} \frac{\|z_k-\bar z\|}{\|z_{k_0}-\bar z\|} = \left[\prod_{j = k_0}^{k-1} \frac{\|z_{j+1}-\bar z\|^2}{\|z_{j}-\bar z\|^2} \right]^\frac12 \leq \left [ \frac{1}{k-k_0} \cdot {\sum}_{j = k_0}^{k-1} \frac{\|z_{j+1}-\bar z\|^2}{\|z_j - \bar z\|^2} \right]^{\frac{k-k_0}{2}} = \left(\frac{\eta}{k-k_0}\right)^{\frac{k-k_0}{2}} \ee
for all $k > k_0$. The structure of this rate is similar to the ones recently derived in \cite{RodNes21,RodNes21-1,JinMok21}. Let us also note that the results in \cite{RodNes21,RodNes21-1,JinMok21} are non-asymptotic and provide a more explicit dependence on the problem parameters (Lipschitz constants, strong convexity parameter, dimension, etc.). However, this non-asymptotic analysis requires stronger assumptions such as a bounded deterioration condition or strong self-concordance. In contrast, the rate in \eqref{eq:rate-rate} is a simple consequence of our more classical convergence analysis of the BFGS method.
\end{remark}

Finally, we note that the summability condition in \cref{thm:sec08:qn} is well-known in the smooth case, see, e.g., \cite{Rit79,Rit81}. A similar result has also been established recently for Broyden-like methods in \cite[Theorem 1]{Man21}.

\subsection{Further Comments on Condition \ref{F3}} \label{sec:f3}
Finally, let us briefly discuss possible further connections between assumption \ref{F3} and the second-order optimality conditions derived in \cref{sec:sop}. As already mentioned, \ref{F3} is generally stronger than the second-order optimality condition \eqref{eq:soc-org}, since the curvature information of $\varphi$ is neglected in \ref{F3}. However, if $\varphi$ is a polyhedral function, it can be shown that this curvature information will vanish. 

Here, the function $\vp : \Rn \to \Rex$ is called polyhedral, if its epigraph $\mathrm{epi}(\vp)$ is a polyhedral set. In this case, the second-order subderivative of $\psi$ and $\vp$ reduce to
\begin{align*}  \mathrm{d}^2\psi(\bar x\vert 0)(h) & = \iprod{h}{\nabla^2 f(\bar x) h} + \mathrm{d}^2\vp(\bar x\vert -\nabla f(\bar x))(h) \\ & =  \iprod{h}{\nabla^2 f(\bar x) h} + \iota_{\mathcal C(\bar x)}(h) = \begin{cases} \iprod{h}{\nabla^2 f(\bar x) h} & \text{if } h \in \mathcal C(\bar x), \\ +\infty & \text{otherwise}, \end{cases} \end{align*}
where $\mathcal C(\bar x) := \{h : \psi^\downarrow(\bar x;h) = 0\}$ denotes the critical cone introduced in \cref{rem:crit}, see, e.g., \cite[Theorem 3.1]{Roc88}. Consequently, since $\mathcal C(\bar x)$ is a cone, we have the following equivalence:
\[ \mathrm{d}^2\psi(\bar x\vert 0)(h) > 0 \quad \forall~h \in \Rn \backslash \{0\} \quad \iff \quad \mathcal H(\mathcal C(\bar x),\nabla^2 f(\bar x)) > 0. \]
Following \cite[Proposition 12.30]{rockafellar2009variational} and \cite[section 4]{facchinei2007finite}, we can infer that the proximity operator $\oprox$ is a piecewise affine-linear, semidifferentiable mapping and there exists $\delta > 0$ such that
\[ \prox{\bar z+h}-\prox{\bar z}=(\proxs)^\prime(\bar z;h) \quad \forall~h \in B_\delta(0). \]
Hence, for all $\varepsilon \in (0,\delta)$, it follows 
\be \label{eq:con-se} S_\varepsilon \subset \mathcal R((\proxs)^\prime(\bar z;\cdot)) = \dom{\partial\Upsilon} \subset \mathcal C(\bar x), \ee
where $\Upsilon(h) := \mathrm{d}^2\vp(\bar x\vert -\nabla f(\bar x))(h)$. Notice that the result $\mathcal R((\proxs)^\prime(\bar z;\cdot)) = \dom{\partial\Upsilon}$ has been shown in the proof of \cref{thm:conn-sec} in a more general context. The condition \eqref{eq:con-se} and \cite[Proposition 6.4]{BauCom11} then imply $\mathrm{aff}(\Delta S_\epsilon) \subset \mathrm{aff}(\Delta\mathcal C(\bar x)) = \mathrm{aff}(\mathcal C(\bar x))$ and 
\[ \liminf_{\epsilon\to0}\mathcal H(\mathrm{aff}(\Delta S_\epsilon),\nabla^2 f(\bar x)) \geq \mathcal H(\mathrm{aff}(\mathcal C(\bar x)),\nabla^2f(\bar x)).  \]
Consequently, in the polyhedral case, \ref{F3} is satisfied if the following strong second-order sufficient condition holds: 
\[\iprod{h}{\nabla^2 f(\bar x)h} > 0 \quad \forall~h \in \mathrm{aff}(\mathcal C(\bar x)) \backslash \{0\}. \]
 
Thus, positive definiteness of $\nabla^2 f(\bar x)$ is only required on the affine hull of the critical cone $\mathcal C(\bar x)$.


\section{Numerical Experiments}
\label{sec:ne}
In this section, we demonstrate the efficiency of the proposed algorithm on a sparse logistic regression, a nonconvex image compression, and a constrained log-determinant problem. All experiments are performed using MATLAB R2020a on a laptop with Intel Core i7 9750h (6 cores and 12 threads) 3.5GHz and 16GB memory.

\subsection{Implementational Details}
\label{sec:parameter}

We first describe some general implementational details of \cref{algo2}. In the following, we will refer to \cref{algo2} as TRSSN. 

We start with a brief overview of the utilized parameters. Most of the parameters are fixed throughout the conducted numerical experiments. The trust region parameters used in \cref{algo3} are listed in \cref{table1}. When updating the trust region radius, we set $\Delta_{k+1}=\gamma_0\Delta_k$ if $\rho_k<\eta_1$ and $\Delta_{k+1}=\Delta_k$ if $\rho_k \in [\eta_1,\eta_2)$. In the case $\rho_k \geq \eta_2$, we choose $\Delta_{k+1}=\gamma_1\Delta_k$.

The parameter $\nu_k$ in the predicted reduction term $\mathrm{pred}_k$ is set as in \cref{eq:nuk-ak}. In particular, we select $a_k := c_a k^{p}\log^{2p}(k)$ with $p \in (0,1)$. The specific choice of $c_a$ and $p$ is given in \cref{table1}. The parameter $\nu_k$ is then chosen as:
\[   \nu_k=\min\{\nu,10^{-3}(n_{\cS}(k)\log^2(n_{\cS}(k)))^{0.2}\|\prox{z_k+s_k}-x_k\|^{0.2}\}. \]
Furthermore, $\tau$ and $\nu$ are set to satisfy \ref{B1}; we choose:
\[   \tau =\frac{2c_\tau}{L^2\lambda^2+2} ,\quad \nu= \frac12\min\left\{\tau, c_{\nu}\left[1-\frac\tau2\left(\frac{L^2\lambda^2}{2}+1\right)\right]\right\}, \quad c_{\tau}=c_{\nu}=0.05.                  \]
 The tolerances $\{\epsilon_k\}$ in the Steihaug-CG method are chosen adaptively via $\epsilon_k=\min\{ \|\Fnor{z_k}\|^{2.5},0.01\}$. The maximal number of CG-iterations is limited to 10. Since we need to compare our normal map-based algorithm with other approaches, we generate a comparable pair of initial points $x_0$ and $z_0$. For given $x_0 \in \Rn$, we determine a corresponding initial point via $z_0 = \argmin_{\prox{z}=x_0}\|\Fnor{z}\|$. In our tests, $z_0$ can be computed explicitly. 

The evaluation of the merit function $\mer$ requires an additional evaluation of the gradient $\nabla f$, which can cause higher computational costs. To avoid this computation, we first check the condition 
\be \label{eq:num-pmer} \psi(\prox{z_k+s_k})\geq \mer(z_k) \ee
and we set $\rho_k=-1$ if \eqref{eq:num-pmer} holds. Notice that \eqref{eq:num-pmer} implies $\mer(z_k+s_k)\geq \psi(\prox{z_k+s_k}) \geq \mer(z_k)$, i.e., we have $\rho_k\leq 0$ and this step would be rejected. Thus, we do not need to compute $\Fnor{z_k+s_k}$ in this case. As we mainly consider large-scale problems, we utilize L-BFGS updates to approximate $\nabla^2 f$. Based on \cite{byrd1994representations}, we implement the following compact form of the (L-)BFGS scheme
\[  B_{k}=\gamma_k I-\begin{bmatrix} S_k & Y_k  \end{bmatrix}  \begin{bmatrix} \frac{1}{\gamma_k}S_k^\top S_k  & \frac{1}{\gamma_k}\cL_k \\  \frac{1}{\gamma_k}\cL_k^\top  & -\mathcal D_k  \end{bmatrix}^{-1} \begin{bmatrix} S_k^\top  \\ Y_k^\top   \end{bmatrix} \]
where $S_k=[d_{k-m},...,d_{k-1}]$, $Y_k=[y_{k-m},...,y_{k-1}]$, $\cL_k$ is the strictly lower part of $S_k^\top Y_k$, $\mathcal D_k$ is the diagonal part of $S_k^\top Y_k$, and $m \in \N$ is a memory parameter. We choose $\gamma_k$ as $\gamma_k= \iprod{y_k}{y_k}/\iprod{s_k}{y_k}$. If an algorithm utilizes a quasi-Newton technique, then we apply L-BFGS approximations with memory $m=10$.
 
\begin{table}[t]
\centering
\begin{tabular}{ccccC{1.7cm}C{1.7cm}cccc} 
\specialrule{1pt}{0pt}{0pt} \\[-1.5ex]
\multirow{2}*{Parameter} & \multicolumn{3}{c}{TR Radius} & \multicolumn{2}{c}{Acceptance Threshold} & \multicolumn{3}{c}{$\nu_k$ and ${\mathrm{pred}}_k$} \\ \cmidrule(lr){2-4} \cmidrule(lr){5-6} \cmidrule(lr){7-9}
& $\Delta_{\min}$ & $\gamma_0$ & $\gamma_1$  & $\eta_1$ & $\eta_2$  & $c_a$  & $p$ \\[0.5ex]
 \hline \\[-1.5ex]
  Value & $10^{-2}$ & $0.25$ & $2$ & $10^{-6}$ & $0.75$ &$10^{-3}$ & $0.1$  \\[0.5ex]
\specialrule{1pt}{0pt}{0pt} \\[-1.5ex]
\end{tabular}
\caption{Parameters of the trust region method.}
\label{table1}
\end{table}

\begin{table}[t]
  \centering
   \begin{tabular}{ccccccccc}  
 \cmidrule[1pt](){1-4} \cmidrule[1pt](){6-9} 
  Algorithm & \texttt{func} & \texttt{grad} & \texttt{prox} &  & Algorithm & \texttt{func} & \texttt{grad} & \texttt{prox} \\[0.5ex]  
  \cmidrule(){1-4} \cmidrule(){6-9} \\[-1.5ex]
   \text{TRSSN} & $\ell$ &  0-$\ell$ & $\ell$ & &  \text{TRSSN-\texttt{O}} & $\ell$ & 0 & $\ell$ \\[0.5ex]
  \text{PNOPT} &  $\ell$ & 1 & multiple & & \text{ForBES} & $\ell$ & $\ell$ & $\ell$ \\[0.5ex]
  \text{FISTA} &  0 &  1 & 1  & &  \text{ASSN} & 0 & $\ell$ & $\ell$ \\[0.5ex]
  \text{iPiano} & 0 & 1 &  1 & & \text{SpaRSA} & $\ell$ & 1 & $\ell$ \\[0.5ex]
 \cmidrule[1pt](){1-4} \cmidrule[1pt](){6-9} \\[-1.5ex]
  \end{tabular}
  \caption{Theoretical iteration costs for the tested algorithms. \texttt{func}, \texttt{grad}, \texttt{prox} denote the no. of  function value, gradient, and  proximity operator evaluations during each iteration of the respective algorithm. The no. of line search steps for algorithms using line search is given by $\ell$. Similarly, the no. of trust region steps between two successful steps is assumed to be $\ell$.}
  \label{table-theoretical}
  \end{table}

\subsection{Sparse Logistic Regression}
\label{sec:logistic}
We first consider a sparse logistic regression problem of the form:
\be\label{eq:logreg-prob} \min_{x}~\psi(x)=f(x)+\varphi(x),\quad f(x):=\frac{1}{N}{\sum}_{i=1}^{N}f_i(x),\quad\varphi(x):=\mu\|x\|_1, \ee
where $f_i(x):=\log(1+\exp(-b_i\cdot\langle a_i,x \rangle))$ denotes the logistic loss function and the data pairs $(a_i,b_i)\in\mathbb{R}^n\times\{-1,1\}$ are given.  The Lipschitz constant of $\nabla f$ can be computed explicitly via $L={\|A\|_2^2}/(4N)$, where $A=(a_1,...,a_N)^\top\in\mathbb{R}^{N\times n}$. Here, the proximity operator is  the well-known shrinkage operator
\begin{align*}
  \mathrm{prox}_{\mu\lambda\|\cdot\|_1}(z) = \mathrm{sgn}(z) \odot \max\left\{0,\vert z\vert-{\mu}{\lambda}\right\},
\end{align*}
where all the operations are understood componentwisely. 
The generalized derivatives of $\mathrm{prox}_{\mu\lambda\|\cdot\|_1}$ at $z$ can be represented as diagonal matrices with
\begin{align}
\label{eq:generalizedjac}
D(z) = \mathrm{diag}(d(z)) \quad \text{and} \quad d_{i}(z) = 
\begin{cases}
0  &\text{if~~}\vert z_i\vert\leq {\mu}{\lambda}, \\
1   & \text{otherwise}, 
\end{cases}
\end{align}
see \cite{milzarek2014semismooth}. We set $\lambda=10$, $\mu =0.002$, and $x^0 = 0$ in all experiments. We compare TRSSN and its full Hessian version TRSSN-H with the following methods: 
 
\textbf{TRSSN-\texttt{O}}~\cite{mannelhybrid}. \text{TRSSN-\texttt{O}} (and its full Hessian version TRSSN-\texttt{OH}) is one of the mentioned existing normal map-based trust region method using a heuristic globalization. Compared to TRSSN, it differs in the acceptance mechanism and the generation of the direction $s_k$. We implement TRSSN-\texttt{O} based on \cite[Algorithm 2]{mannelhybrid}. As suggested in \cite[Section 3.5.2, Algorithm 2]{pieper2015finite}, we adjust the inner product from the standard Euclidean one to $\langle \cdot,\cdot \rangle_D$ in the CG solver.  
Throughout \cref{sec:ne}, we use the same parameters for the trust region algorithms \text{TRSSN} and \text{TRSSN-\texttt{O}}. As mentioned in \cite{pieper2015finite}, $\psi$ can remain unchanged at an iteration with an arbitrarily small trust region radius. In this case, we perform line search along the direction of $-\Fnors(z_k)$ to avoid stagnation.

\textbf{PNOPT}~\cite{lee2014proximal}. \text{PNOPT} is a proximal Newton method which uses a quasi-Newton approximation of $\nabla^2f$ in the proximal step. We use the source code released by the authors\footnote{\url{https://web.stanford.edu/group/SOL/software/pnopt/}}. All parameters are set to the default values.

\textbf{ASSN}~\cite{xiao2018regularized}. \text{ASSN} is a semismooth Newton method for solving monotone equations. Following \cite[section 4.1]{xiao2018regularized} and  based on the source code provided by the authors, we implement ASSN with full Hessian information to solve $F^{\lambda}_{\mathrm{nat}}(x) = 0$. We also tried an L-BFGS version of ASSN (called ASLB). However, its performance is not comparable with the base algorithm, which agrees with an observation made in \cite[section 4.2]{xiao2018regularized}. Thus, we only report the performance of ASSN. As noted in \cite{xiao2018regularized}, $\lambda$ should be no larger than ${1}/{L}$ to guarantee monotonicity of $F^{\lambda}_{\mathrm{nat}}$. Hence, we choose $\lambda = {1}/{L}$ in the \text{ASSN} code.  

\textbf{FISTA}~\cite{beck2009fast}. \text{FISTA} is a first-order method with Nesterov-type acceleration. We use the known Lipschitz constant as step size.  

\textbf{ForBES}~\cite{SteThePat17}. \text{ForBES} applies the semismooth Newton method to the natural residual $F^{\lambda}_{\mathrm{nat}}$ using a forward-backward envelope as merit function. We use the code provided by the authors\footnote{\url{https://github.com/kul-forbes/ForBES}}. We choose the Lipschitz constant $L$ as initial value for $1/\lambda$. The released MATLAB function then corrects $\lambda$ to $\lambda={0.95}/L$. All other parameters are set to be default values.

\begin{table}[t]
  \centering
   \begin{tabular}{lcclcclcc}  
 \cmidrule[1pt](){1-9}  
  Dataset & $N$ & $n$ & Dataset & $N$ & $n$ & Dataset & $N$ & $n$\\[0.5ex]  
  \cmidrule[.5pt](){1-9}\\[-1.5ex]
   \texttt{BIO} & 145\,751 & 75  & \texttt{news20} & 19\,996  & 1\,355\,191 & \texttt{epsilon} & 400\,000 & 2\,000\\[0.5ex] 
   \texttt{CINA} & 16\,033  & 132 & \texttt{rcv1} & 20\,242 & 47\,236 &  \texttt{gisette} & 6\,000 & 5\,000    \\[0.5ex] 
   \texttt{covtype} & 581\,012 & 54 & \texttt{real-sim} & 72\,309 & 20\,958  &    \\[0.5ex]    
   \cmidrule[1pt](){1-9}\\[-1.5ex]
  \end{tabular}
  \caption{Information of the different datasets.}
  \label{table2}
  \end{table}

We tested TRSSN and the mentioned algorithms on eight different datasets (\texttt{BIO}\footnote{\url{https://osmot.cs.cornell.edu/kddcup/datasets.html}}, \texttt{CINA}\footnote{\label{data}\url{http://www.causality.inf.ethz.ch/data}}, \texttt{covtype}\footnote{\label{libsvmnote}\url{https://www.csie.ntu.edu.tw/~cjlin/libsvmtools/datasets/}}, \texttt{epsilon}\footref{libsvmnote}, \texttt{gisette}\footref{libsvmnote}, \texttt{new20}\footref{libsvmnote}, \texttt{rcv1}\footref{libsvmnote}, \texttt{real-sim}\footref{libsvmnote}). More information about these datasets can be found in \cref{table2}. The results of our comparison are shown in \cref{fig2b}. Specifically, in \cref{fig2b}, we plot the relative error $\texttt{rel$\_$err} = (\psi(x)-\psi^*) /\max\{1,\psi^*\}$ with respect to cpu-time. Here, $\psi^*$ is the lowest objective function value encountered by all the algorithms during the experiment. In \cref{table-theoretical}, we list the approximate numbers of function value, gradient, and proximal evaluations each of the tested algorithms requires per iteration. Almost all of the second order algorithms need one proximity operator evaluation per line search or trust region trial step. 

The methods ForBES, PNOPT, and TRSSN generally achieve the best results and outperform the first-order approach FISTA on all datasets. PNOPT performs well in its early stage and it can quickly recover solutions with medium accuracy. However, the costs for solving the proximal Newton subproblems seem to dominate when the dimension of the problem increases. This is most apparent on \texttt{news20} and \texttt{rcv1}. 
%
On the dataset \texttt{news20}, ASSN performs well and transition to fast local convergence can be observed (we believe that a more problem dependent and tuned choice of the involved parameters can also lead to earlier fast local convergence on the other datasets). TRSSN-\texttt{O} performs similar to TRSNN for \texttt{CINA}, but is generally outperformed by TRSSN. We have observed that many of the trust region trial steps in TRSSN-\texttt{O} are rejected causing smaller trust region radii and slower convergence. 
%
TRSSN-H converges faster than TRSSN on the datasets \texttt{gisette}, \texttt{news20}, and \texttt{rcv1}. Due to the higher dimension of these problems, the L-BFGS approximations used in TRSSN provide less advantages and TRSSN-H can benefit from an early transition to fast superlinear convergence.
We notice that ForBES and PNOPT use line search to damp the second-order step and to ensure convergence. ASSN utilizes alternative projection steps if the semismooth Newton step violates a certain acceptance criterion. By contrast, the trust region-type framework (paired with the inexact CG-method) requires less computational steps and gradient and function values can be fully reused in case a trial step is unsuccessful. 

\begin{figure}[htbp]
\centering
\includegraphics[width=12.5cm]{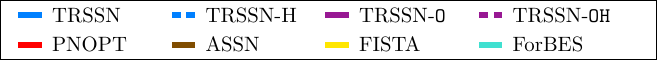}
\subfigure[\texttt{BIO}.]{
\includegraphics[width=6cm]{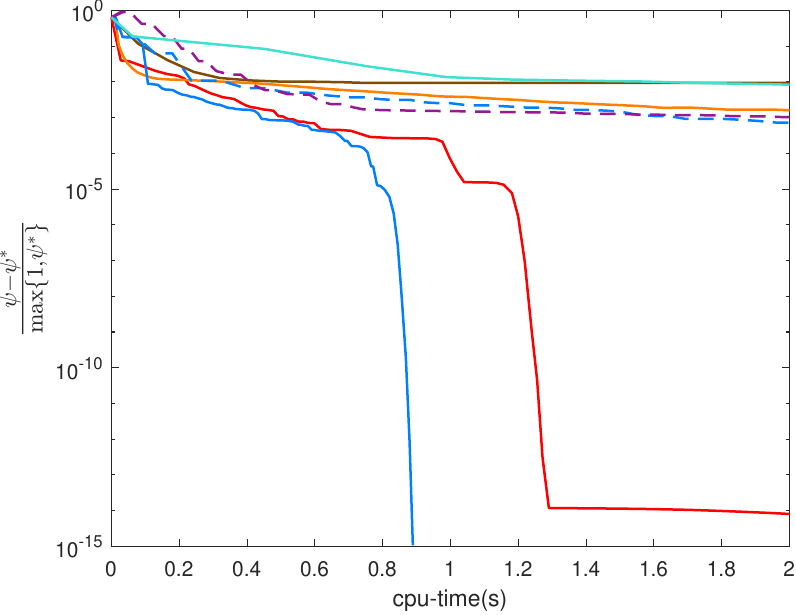}
}
\;
\subfigure[\texttt{CINA}.]{
\includegraphics[width=6cm]{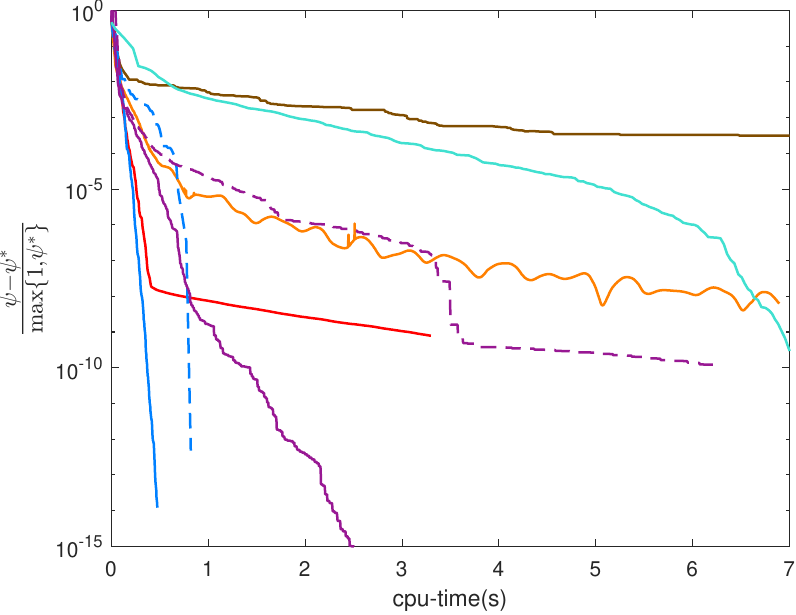} 
}
\\[-1ex]
\subfigure[\texttt{covtype}.]{
\includegraphics[width=6cm]{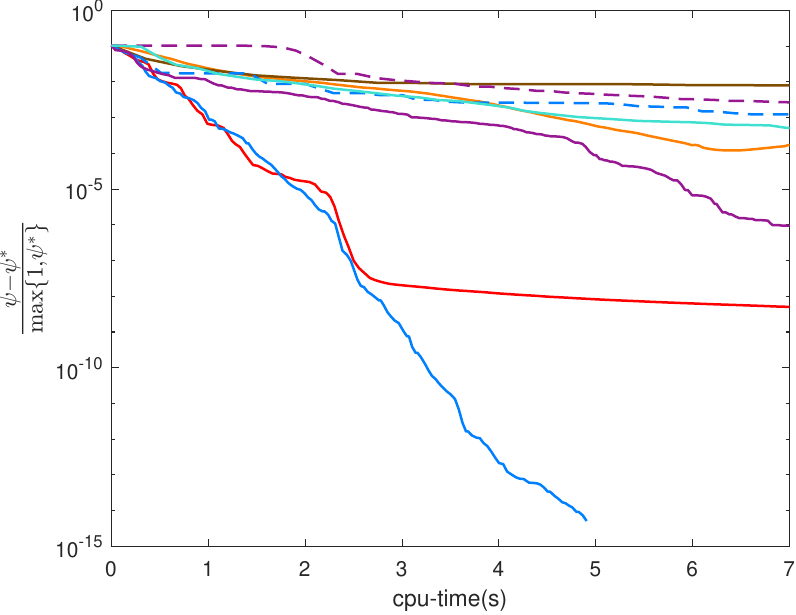}
}
\;
\subfigure[\texttt{epsilon}.]{
\includegraphics[width=6cm]{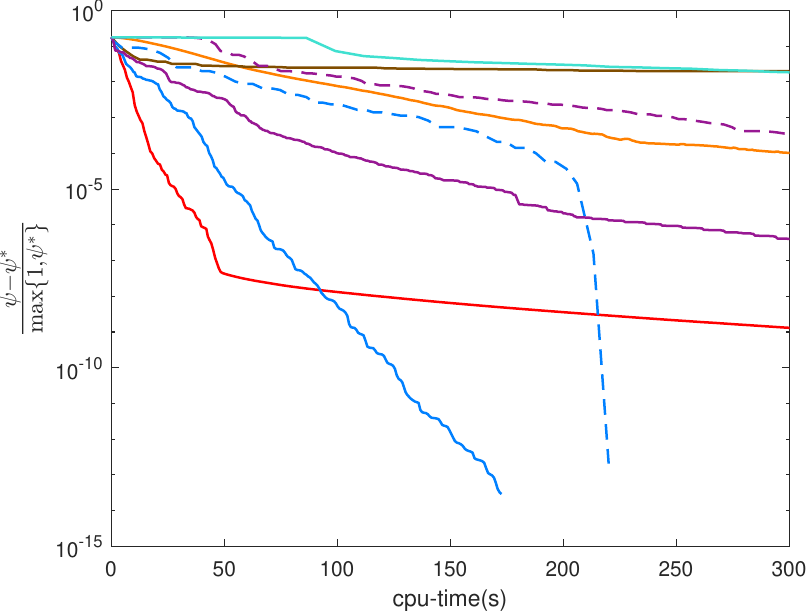}
}
\\[-1ex]
\subfigure[\texttt{gisette}.]{
\includegraphics[width=6cm]{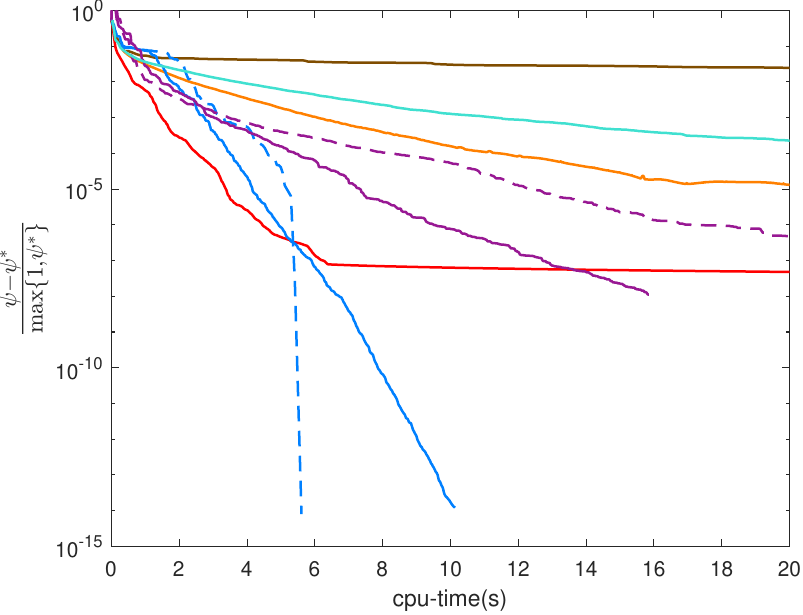}
}
\;
\subfigure[\texttt{news20}.]{
\includegraphics[width=6cm]{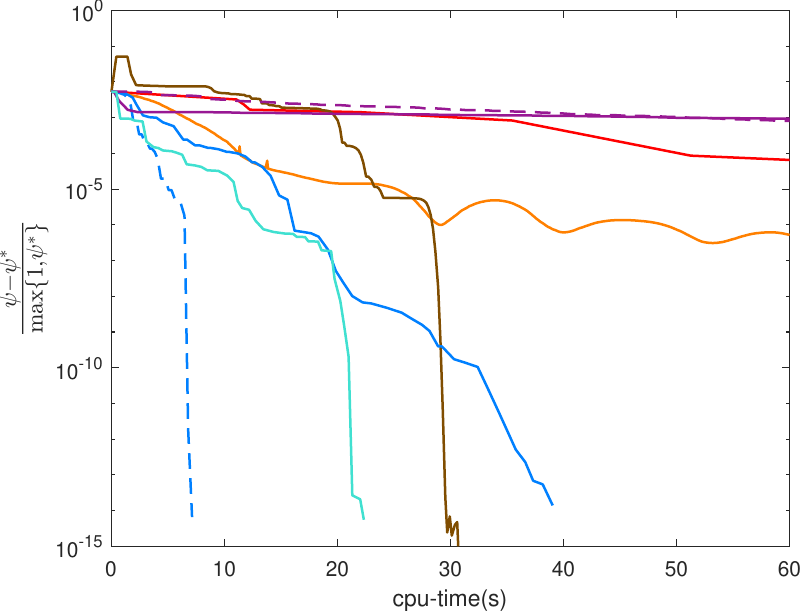} 
}
\\[-1ex]
\subfigure[\texttt{rcv1}.]{
\includegraphics[width=6cm]{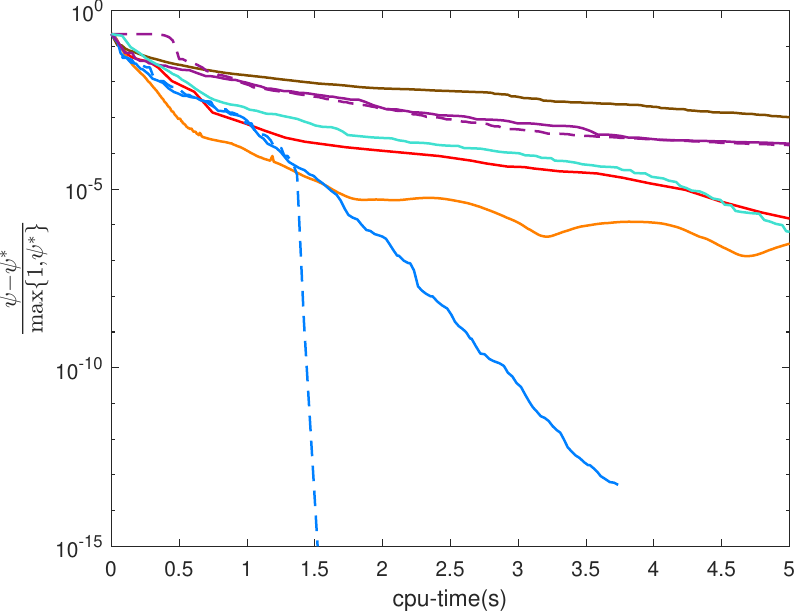}
}
\;
\subfigure[\texttt{real-sim}.]{
\includegraphics[width=6cm]{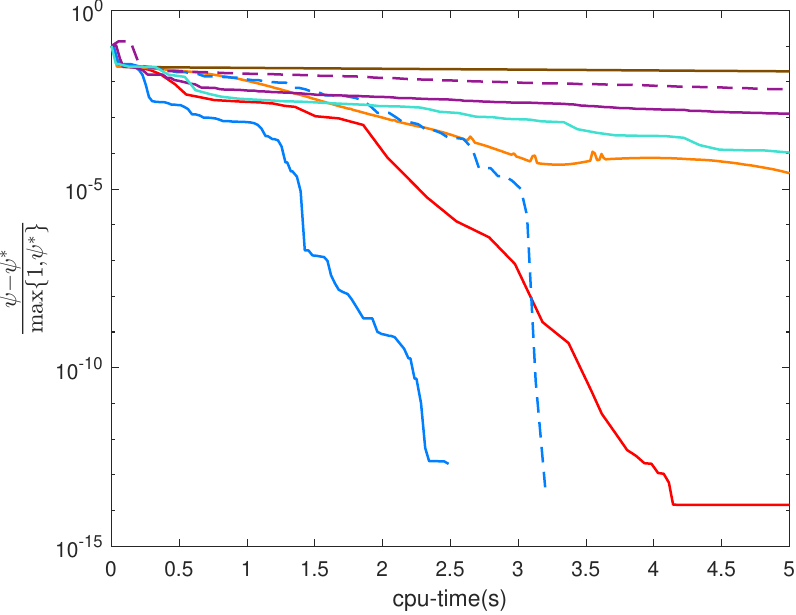}
}
\caption{Change of the relative error $\texttt{rel$\_$err}$ with respect to the cpu-time for solving the $\ell_1$-logistic
regression problem \cref{eq:logreg-prob}.}
\label{fig2b}
\end{figure} 

\begin{figure}[htbp]
\centering
\includegraphics[width=12.5cm]{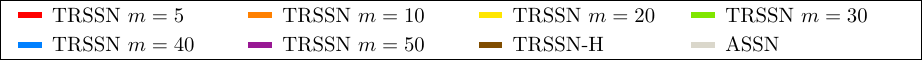}
\subfigure[\texttt{CINA} -- iterations.]{
\includegraphics[width=6cm,trim = 0 0 25 15,clip]{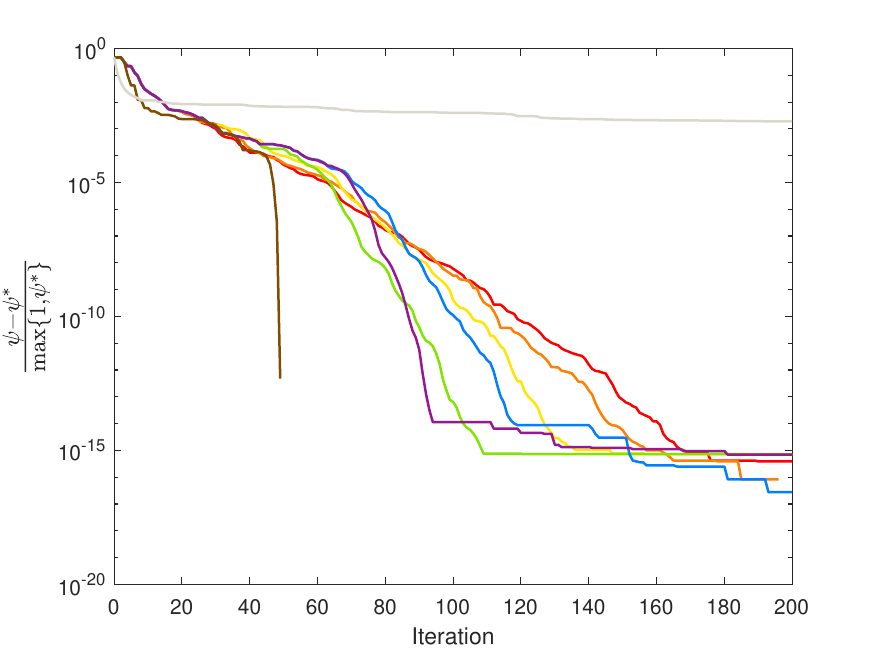}
}
\;
\subfigure[\texttt{CINA} -- cpu-time.]{
\includegraphics[width=6cm,trim = 0 0 25 15,clip]{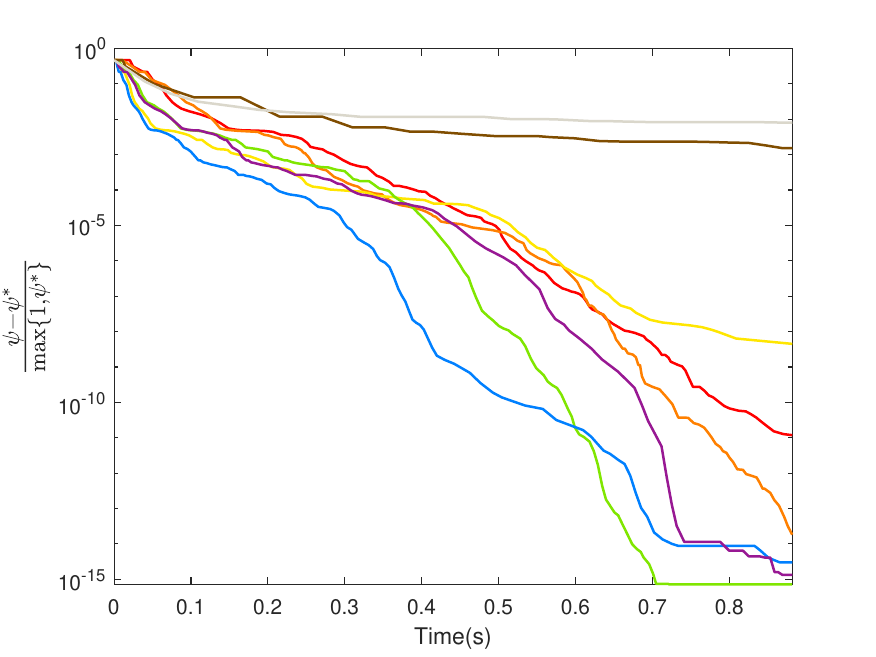}
}
\\[-1ex]
\subfigure[\texttt{rcv1} -- iterations.]{
\includegraphics[width=6cm,trim = 0 0 25 15,clip]{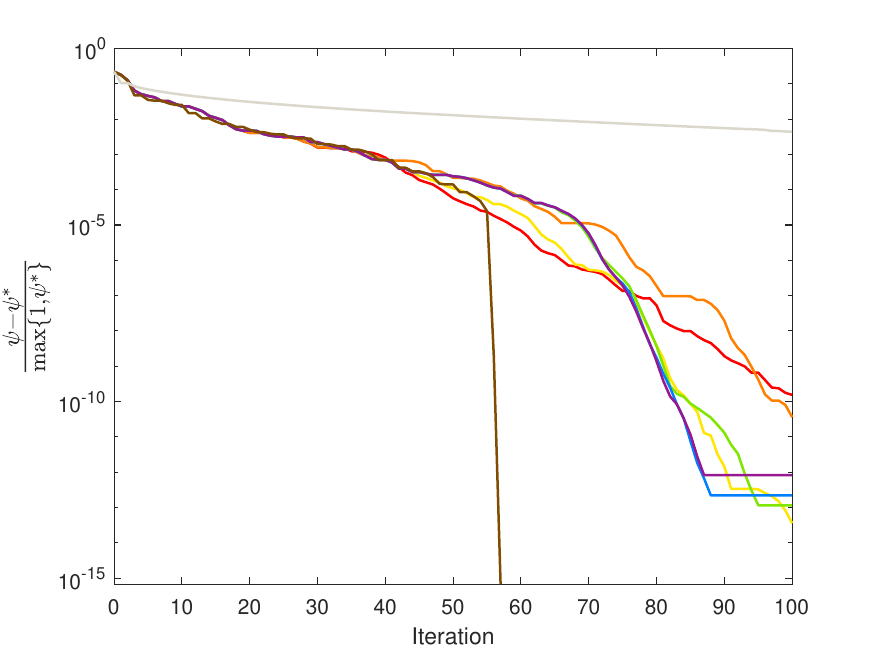}
}
\;
\subfigure[\texttt{rcv1} -- cpu-time.]{
\includegraphics[width=6cm,trim = 0 0 25 15,clip]{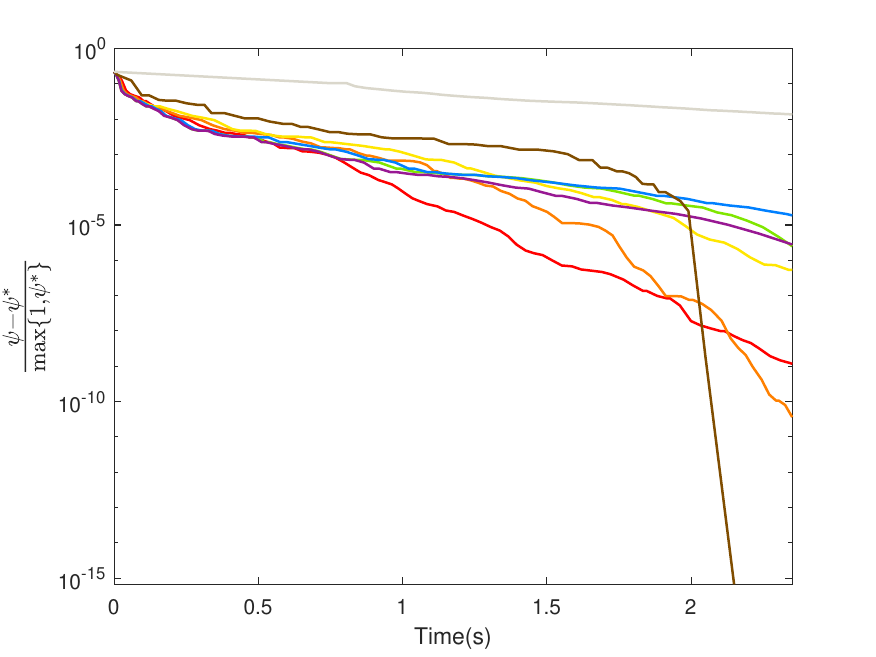}
}
 
\caption{Results for the logistic regression problem. Comparison of TRSSN on \texttt{CINA} and \texttt{rcv1} for different choices of the L-BFGS memory parameter $m$.}
\label{fig1}
\end{figure} 

In \cref{fig1}, we discuss the performance of TRSSN and TRSSN-H on the datasets \texttt{CINA} and \texttt{rcv1} for different choices of the L-BFGS memory parameter $m$. \cref{fig1} illustrates that full Hessian information can be beneficial in certain situations. However, the computational costs can also be much higher as shown in \cref{fig1} (b) and \cref{fig2b}. It turns out that the numerical performance of TRSSN with L-BFGS updates is not very sensitive to the choice of $m$.

\subsection{Linear Diffusion Based Image Compression}
\label{subsec:lineardiffusion}
Next, we test TRSSN on a linear diffusion based image compression problem. The compression model we consider has been studied in  \cite{galic2008image,schmaltz2009beating}.  
The model utilizes a homogeneous diffusion based interpolation to find the optimal data. The associated optimization formulation is given by 
\begin{align}
\min\limits_{x,c}~\frac{1}{2}\|x-u\|^2+\lambda\|c\|_1 \quad
\mathrm{s.t.}\quad \diag(c)(x-u)-(I-\diag(c))\mathcal L u=0, \label{eq9-1}
\end{align}
where $u\in\mathbb{R}^n$ denotes the (stacked) ground truth image, $x\in\mathbb{R}^n$ is the reconstructed image, $c\in\mathbb{R}^n$ denotes the inpainting or compression mask, and $\mathcal L \in \R^{n \times n}$ is the discretized Laplacian. If $c\neq 0$ and $c\in[0,1]^n$, then the matrix $A(c):=\diag(c)+(\diag(c)-I)\mathcal L$ can be shown to be invertible \cite{MaiBruWeiFor11}. Setting $x=A(c)^{-1}\diag(c)u$, problem $\cref{eq9-1}$ can be rewritten as:
\begin{align*}
\min_c~\frac{1}{2}\|A(c)^{-1}\diag(c)u-u\|^2+\mu\|c\|_1 \quad \mathrm{s.t.} \quad c\in [0,1]^n.
\end{align*}
Thus, \cref{eq9-1} reduces \cref{eq1-1}
with $f(c)=\frac{1}{2}\|A(c)^{-1}\diag(c)u-u\|^2$ and $\varphi(c)=\lambda\|c\|_1+\iota_{[0,1]^n}(c)$. By \cite[Lemma 5.2]{OchCheBroPoc14}, the gradient of $f$ can be calculated as follows:
\begin{align}
\label{eq9-2}
 \nabla f(c)=\diag(-\mathcal L x+u-x)[A(c)^{\top}]^{-1}(x-u), \quad x=A(c)^{-1}\diag(c)u. 
\end{align}
Since the Lipschitz constant of $\nabla f$ can not be computed exactly, we use an adaptive strategy to estimate $L$. 
In each trial step, we calculate $c_{k+1}=\prox{z_k+s_k}$ and choose $$L_k=\max\left\{2\cdot\frac{f(c_{k+1})-f(c_k)-\langle \nabla f(c_k),c_{k+1}-c_k\rangle}{\|c_{k+1}-c_k\|^2},2L_{k-1}\right\}.$$
We reset $\lambda=1/L_k $ and adjust $\nu$ and $\tau$ as specified in \cref{sec:parameter}. Moreover, we set $\Delta_{\min}=10^{-6}$ for all images. In all of our examples, we select the initial estimate $L_0 = 0.1$. 
According to~\cite[Theorem 1]{yu2013decomposing}, it holds that
\[  \prox{z}=\mathcal P_{[0,1]^n}\circ\mathrm{prox}_{\mu\lambda \|\cdot\|_1}(z)=\min\{0,\max\{1,\mathrm{prox}_{\mu\lambda \|\cdot\|_1}(z)\}\}.       \]
%
The corresponding generalized derivative of this proximity operator can be constructed similarly to \cref{eq:generalizedjac}. Here, we choose
\begin{align*}
D(z) = \mathrm{diag}(d(z)) \quad \text{and} \quad d_{i}(z)= 
\begin{cases}
0 &  \text{if } z_i\leq {\mu}{\lambda} \; \text{ or } \; z_i\geq {\mu}{\lambda}+1 , \\
1 & \text{otherwise},
\end{cases}
\end{align*}
see, e.g., \cite[Example 4.2.17]{milzarek2016numerical}. Since this problem is nonconvex, most of the algorithms tested in \cref{sec:logistic} are no longer directly applicable. We compare our method with the following algorithms:

\begin{figure}[thbp]
\centering
\includegraphics[width=12.5cm]{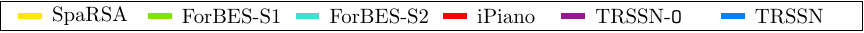}
\subfigure[\texttt{books} -- cpu-time.]{
\includegraphics[width=6cm]{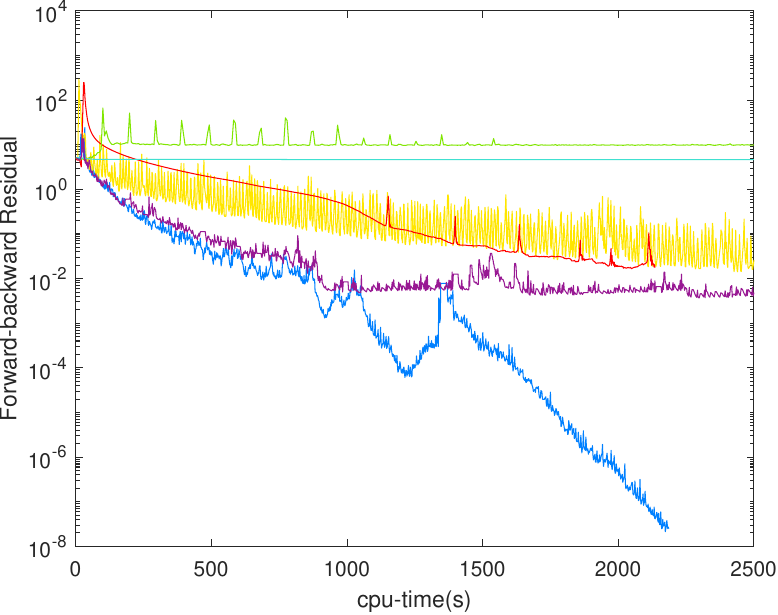}
}
\;
\subfigure[\texttt{coffee} -- cpu-time.]{
\includegraphics[width=6cm]{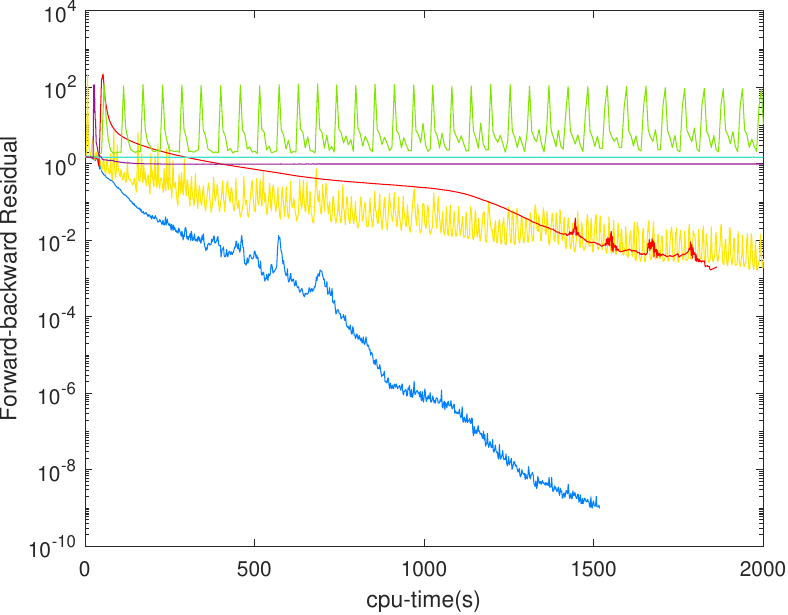}
}
\\[-1ex]
\subfigure[\texttt{mountain} -- cpu-time.]{
\includegraphics[width=6cm]{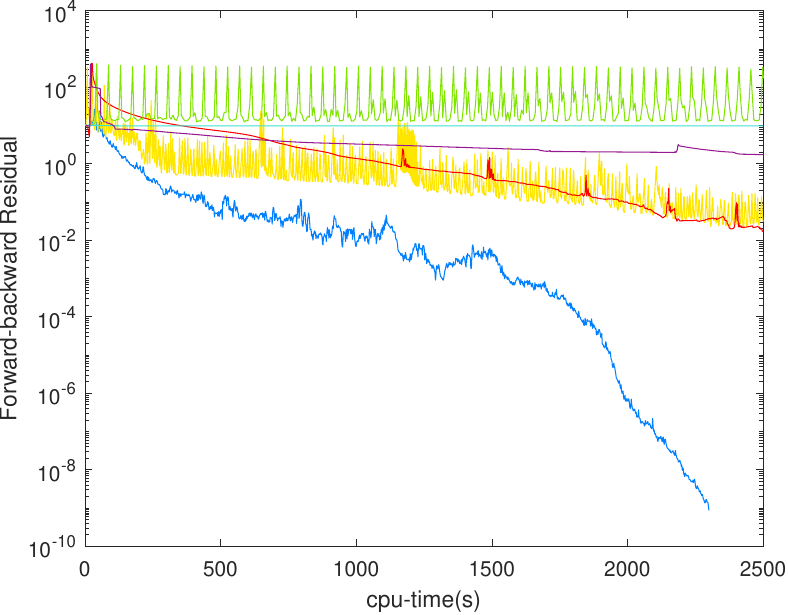}
}
\;
\subfigure[\texttt{stones} -- cpu-time.]{
\includegraphics[width=6cm]{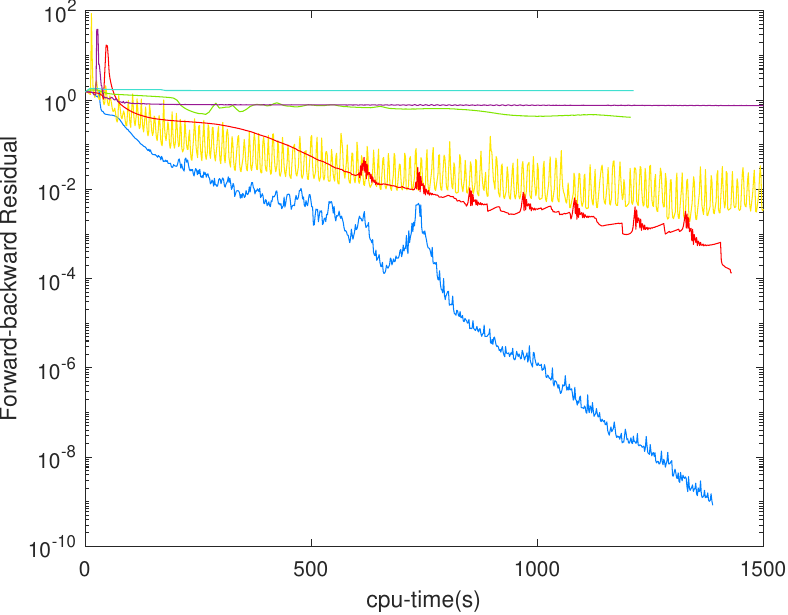}
}
\caption{Numerical comparison of iPiano, SpaRSA, ForBES, TRSSN-\texttt{O}, and TRSSN on a diffusion based image compression problem. Plot of the norm of the natural residual with respect to required cpu-time for different images.}
\label{fig3-2}
\end{figure} 

\textbf{ForBES}~\cite{SteThePat17}. We use the same code as in \cref{sec:logistic}. In ForBES, $f$ is assumed to be defined on $\Rn$ which is not the case for \cref{eq9-1}. To apply ForBES, we tested two strategies. The first variant, referred to as ForBES-S1, sets $f$ to $+\infty$ outside of $\dom{\varphi}$. In our second strategy, ForBES-S2, $\nabla f$ is computed via \cref{eq9-2} regardless of the constraint $c\in[0,1]^n$.  

\textbf{iPiano}~\cite{OchCheBroPoc14}. \text{iPiano} is a forward-backward splitting method with momentum. We implement iPiano with backtracking following the recommendations in \cite{OchCheBroPoc14}. As in  \cite{OchCheBroPoc14}, we set $\alpha={1.99(1-\beta)}/{L}$, $\eta=2$, and $\beta=0.8$. The initial Lipschitz constant $L_0$ is set to $0.1$. As suggested in \cite{OchCheBroPoc14},  $L$ is increased adaptively by $5\%$ every five steps. 

\textbf{SpaRSA}~\cite{WriNowFig09}. \text{SpaRSA} is a proximal gradient method with Barzilai-Borwein (BB) step sizes.  We implement SpaRSA based on the code provided by the authors\footnote{\url{https://www.lx.it.pt/~mtf/SpaRSA/}}. In SpaRSA, we set $\alpha_k = \iprod{r_k}{r_k}/\iprod{r_k}{s_k}$, $r_k = \nabla f(c_k)-\nabla f(c_{k-1})$, and $s_k = c_k - c_{k-1}$,  
as initial step size. As in \cite{WriNowFig09}, the lower and upper bounds on $\alpha_k$ are given by $\alpha_{\min}=10^{-30}$ and $\alpha_{\max}=10^{30}$. Following a strategy proposed in \cite{WenYinGolZha10}, we use a nonmonotone line search procedure, \cite{ZhaHag04}, to ensure convergence of the approach and stability of the BB step sizes. 



\begin{figure}[t!]
\centering
\subfigure[\texttt{books}, original figure]{
\includegraphics[width=5cm]{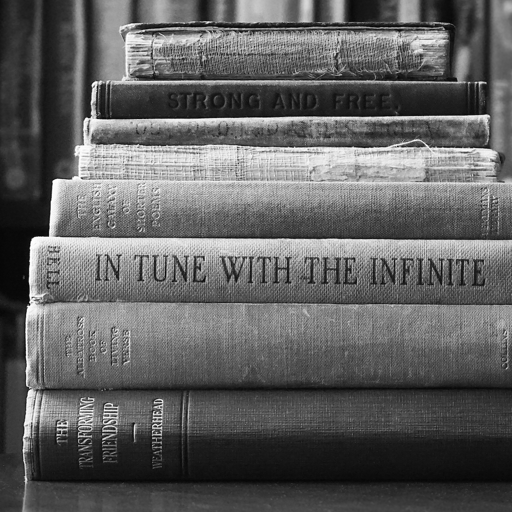}
}
\subfigure[\texttt{books}, mask (${ds}\!=\!8.43\%$)]{
\includegraphics[width=5cm]{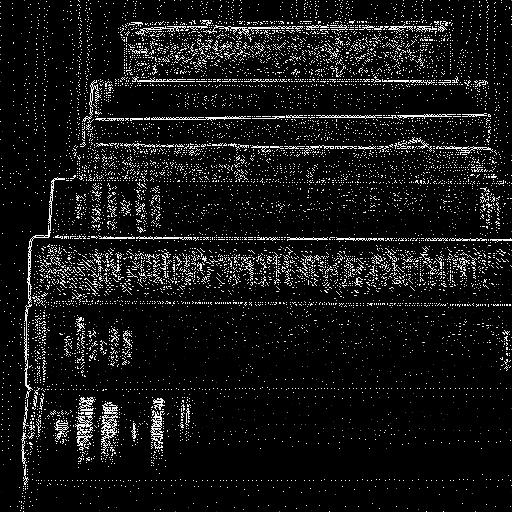}
}
\subfigure[\texttt{books}, reconstruction]{
\includegraphics[width=5cm]{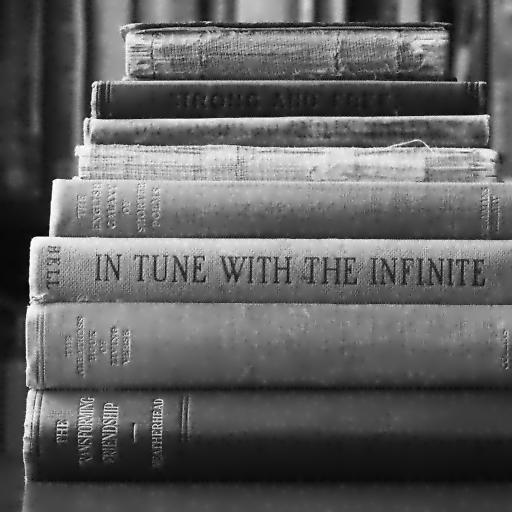}
}
\subfigure[\texttt{coffee}, original figure]{
\includegraphics[width=5cm]{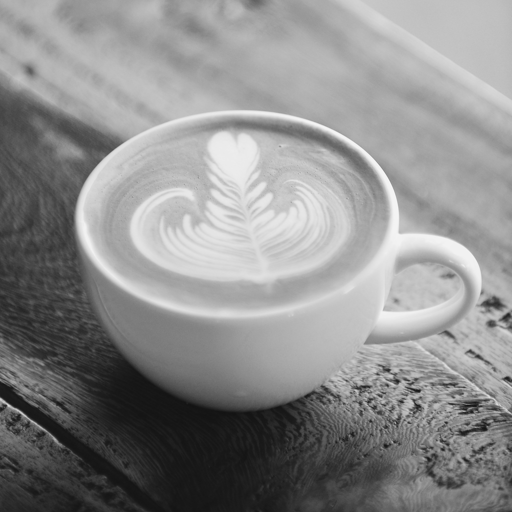}
}
\subfigure[\texttt{coffee}, mask (${ds}\!=\!4.45\%$)]{
\includegraphics[width=5cm]{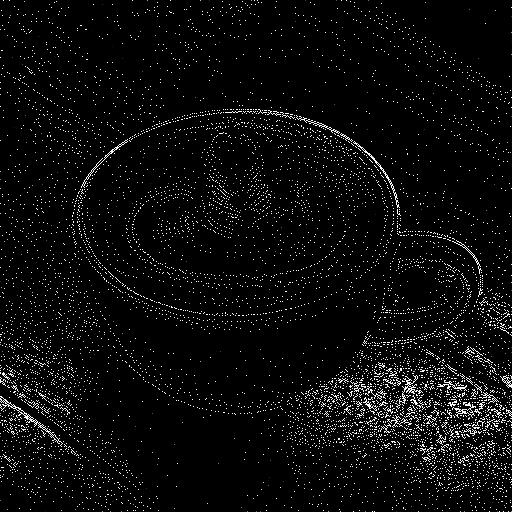}
}
\subfigure[\texttt{coffee}, reconstruction]{
\includegraphics[width=5cm]{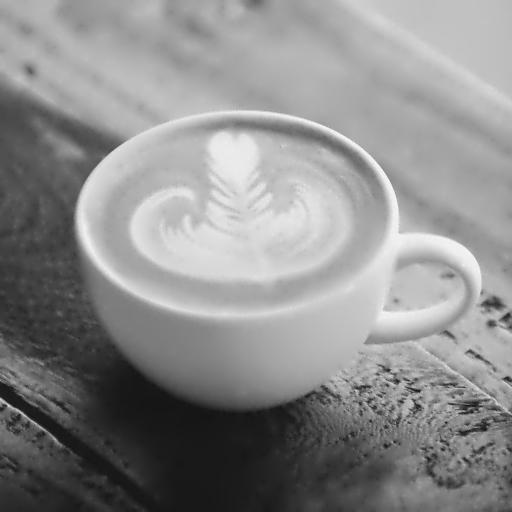}
}
\subfigure[\texttt{stones}, original figure]{
\includegraphics[width=5cm]{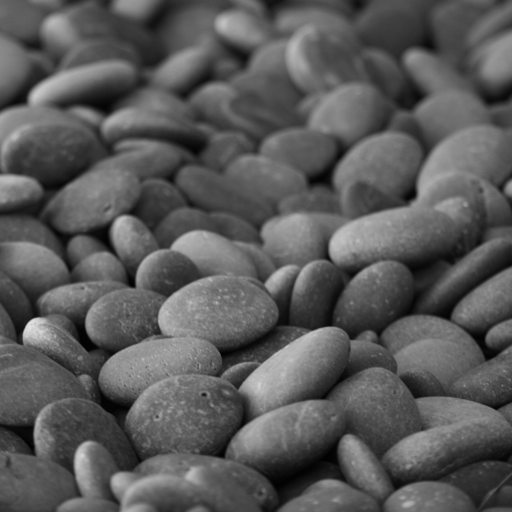}
}
\subfigure[\texttt{stones}, mask (${ds}\!=\!4.37\%$)]{
\includegraphics[width=5cm]{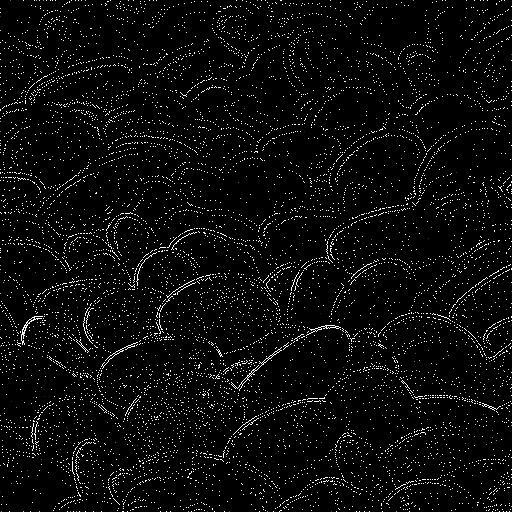}
}
\subfigure[\texttt{stones}, reconstruction]{
\includegraphics[width=5cm]{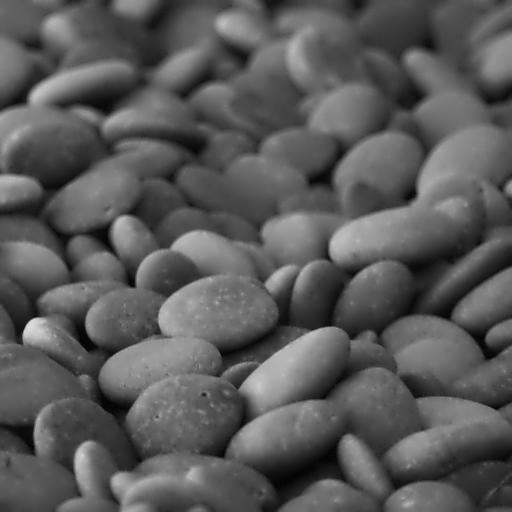}
}
\caption{Illustration of the different ground truth images, the inpainting or compression masks, and the corresponding reconstructions for \texttt{books}, \texttt{coffee}, and \texttt{stones}. The density ($ds$) of the masks $c$ is calculated via ${ds} := 100\% \cdot [ \vert\{i: c_i > 0\}\vert / (512 \times 512)]$.}
\label{fig4}
\end{figure} 

Our results are shown in \cref{fig3-2} and \cref{fig4}. We compare the performance of the algorithms on the four images \texttt{books}, \texttt{coffee}, \texttt{mountain}, and \texttt{stones}\footnote{Image credentials: \texttt{books} by Suzy Hazelwood; \texttt{coffee} by Atichart Wongubon; \texttt{mountain} by Denis Linine; \texttt{stones} by Travel Photographer; all images can be found on StockSnap.}. All images are rescaled to size $512\times 512$ and we use $\mu = 0.01$ for \texttt{books}, $\mu = 0.003$ for \texttt{coffee} and \texttt{stones}, and $\mu = 0.02$ for \texttt{mountain}. In all tested images, the performance of SpaRSA and iPiano is similar, while iPiano is generally more stable. 
ForBES-S1 always reaches the maximum number of line search steps. In such a case, it performs a first-order step as safe-guard. Hence, the performance of ForBES-S1 is similar to the other first-order methods. ForBES-S2 suffers from the sudden change of $f$ near the boundary. This causes the Lipschitz constant of $\nabla f$ to be large and $\lambda$ to be small which finally yields marginal updates in each iteration. The performance of TRSSN-\texttt{O} is  not competitive as too many trial steps are rejected causing the approach to behave like a first-order method. For the image \texttt{books},  TRSSN-\texttt{O} is comparable to TRSSN at the beginning, but is eventually outperformed by TRSSN. Overall, as demonstrated in \cref{fig3-2}, TRSSN outperforms the other approaches in terms of cpu-time and the performance of TRSSN appears to be more stable. 
 
\subsection{Constrained Log-determinant Optimization}
Finally, we consider the constrained log-determinant minimization problem
\begin{equation}
  \label{log_det}
    \min_{X\in\mathbb{S}^n_+, 0\preceq X\preceq I}~\log(\det(X+S_1))-\mu\log(\det
    (X+S_2)),
\end{equation}
where $S_1,S_2\in\mathbb{S}^{n}_{++}$ and $\mu\in(0,1)$. Problem \eqref{log_det} is studied in \cite{lau2022uniqueness,geng2014capacity} for the computation of the capacity region of a two-receiver Gaussian broadcast channel with private and common messages. 
%
Setting $f(X) :=\log(\det(X+S_1))-\mu\log(\det
(X+S_2))$, the gradient and Hessian of $f$ are given by:
\begin{align*}
  \nabla f(X) & = (X+S_1)^{-1}-\mu(X+S_2)^{-1}, \\
  \nabla^2f(X)[H] &=-(X+S_1)^{-1}H(X+S_1)^{-1}+\mu(X+S_2)^{-1}H(X+S_2)^{-1}.
\end{align*}
The Lipschitz modulus of $\nabla f$ is further given by $\|S_1^{-1}\|^2+\mu\|S_2^{-1}\|^2$. As pointed out in \cite{han2015large}, when the dimension $n$ is large, direct calculation of the determinant causes significant numerical error. Thus, we use Cholesky decompositions (of $X+S_1$ and $X+S_2$) to compute the function, gradient, and Hessian values of $f$. 
Moreover, defining $\mathcal C :=\{X\in\mathbb{S}^n: 0\preceq X\preceq I\}$, it holds that $\mathcal P_{\mathcal C}(X)=\mathcal P_{\mathbb{S}^n_+}(I-\mathcal P_{\mathbb{S}^n_+}(I-X))$. Therefore, we can apply the calculus for spectral operators \cite[Theorem 6.2]{ding2020spectral} and existing formulas for $(\mathcal P_{\mathbb{S}^n_+})^\prime(X,H)$ and $\partial \mathcal P_{\mathbb{S}^n_+}$, \cite{sun2006strong}, to characterize the generalized derivatives of $\mathcal P_{\mathcal C}$. Notice that the evaluation of $\mathcal P_{\mathcal C}(X)$ involves a (full) eigendecomposition of the matrix $X$.

Let $X=U DU^\top$ be the eigenvalue decomposition of $X$, where $D=\diag(\lambda_1,\dots,\lambda_n)$ are the ordered eigenvalues of $X$ with $\lambda_1\leq \dots \leq\lambda_{d_1}<0=\lambda_{d_1+1}=\dots=\lambda_{d_1+d_2}<\lambda_{d_1+d_2+1}\leq \dots\lambda_{d_1+d_2+d_3}<1=\lambda_{d_1+d_2+d_3+1}=\dots=\lambda_{d_1+d_2+d_3+d_4}<\lambda_{d_1+d_2+d_3+d_4+1}\leq\dots\leq \lambda_n$. Then, for $W\in\mathbb{S}^n$, we have
 \begin{equation*}
   \partial{\mathcal P}_{\mathcal C}(X)[W]=U^\top\begin{pmatrix}
      0 & 0 & \bar W_{13}\odot S_{13} & \bar W_{14}\odot S_{14} & \bar W_{15}\odot R_{15} \\
      0 & \partial{\mathcal P}_{\mathbb{S}^n_+}(0)[\bar W_{22}] & \bar W_{23} & \bar W_{24} & \bar W_{25}\odot T_{25} \\
     \bar W_{31}\odot S^\top_{13} & \bar W_{32} & \bar W_{33} & \bar W_{34} & \bar W_{35}\odot T_{35} \\
     \bar W_{41}\odot S_{14}^\top & \bar W_{42} & \bar W_{43} & \partial{\mathcal P}_{\mathbb{S}^n_+}(0)[\bar W_{44}] &  0 \\
     \bar W_{51}\odot R_{15}^{\top} & \bar W_{25} & \bar W_{35} & 0 & \bar W_{55} 
   \end{pmatrix}U,
 \end{equation*}
where $\bar W=U^\top WU$ and is decomposed into the block structure $[\bar W_{ij}]_{i,j\in [5]}$ and it holds that $[S_{13}]_{ij}=\frac{\lambda_{d_1+d_2+j}}{\lambda_{d_1+d_2+j}-\lambda_i}$, $[S_{13}]_{ij}=\frac{\lambda_{d_1+d_2+d_3+j}}{\lambda_{d_1+d_2+d_3+j}-\lambda_i}$, $[T_{25}]_{ij}=\frac{1-\lambda_{d_1+i}}{\lambda_{d_1+d_2+d_3+d_4+j}-\lambda_{d_1+i}}$, $[T_{35}]_{ij}=\frac{1-\lambda_{d_1+d_2+i}}{\lambda_{d_1+d_2+d_3+d_4+j}-\lambda_{d_1+d_2+i}}$, and $[R_{15}]_{ij}=\frac{1}{\lambda_{d_1+d_2+d_3+d_4+j}-\lambda_i}$. 

Since this problem is nonconvex, we test the same algorithms as in \cref{subsec:lineardiffusion} with the same parameter settings. (TRSSN and TRSSN-\texttt{O} use L-BFGS approximations with $m = 10$). As in \cref{sec:logistic}, TRSSN-H denotes the variant of TRSSN using the full Hessian of $f$. We test all methods on 5 randomly generated problems ($S_i = U_i^\top U_i + 10^{-4}I$, $U_i \sim \mathcal U_n[0,1]$, $i \in \{1,2\}$) with dimension $n \in \{200,500,800\}$ and report the performance in \cref{fig4-1} and \cref{table3}. The TRSSN variants use $\lambda = 0.005$ and we choose $\mu=0.5$. 

\begin{figure}[t!]
  \centering
   \includegraphics[width=16.1cm]{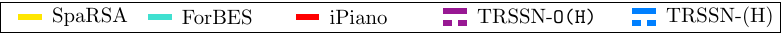}
   \subfigure[$n=200$ -- cpu-time.]{
   \includegraphics[width=5.2cm]{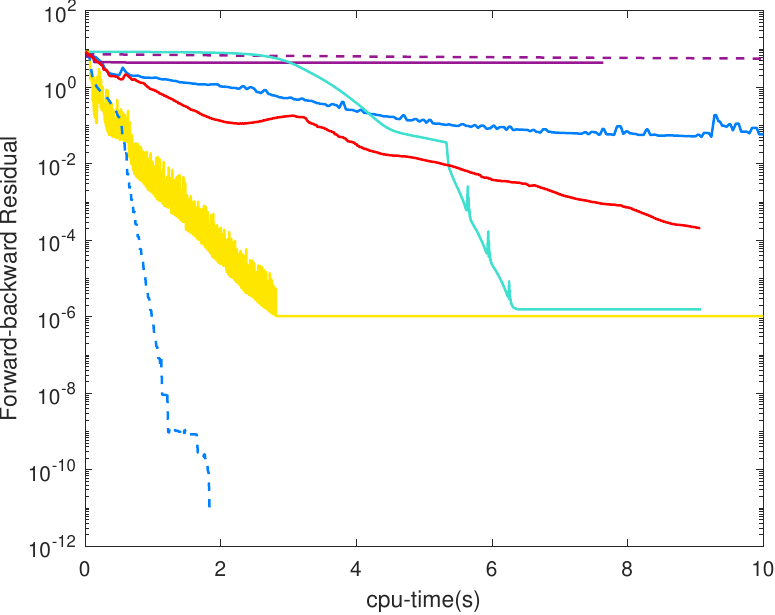}
   }
   \subfigure[$n=500$ -- cpu-time.]{
   \includegraphics[width=5.2cm]{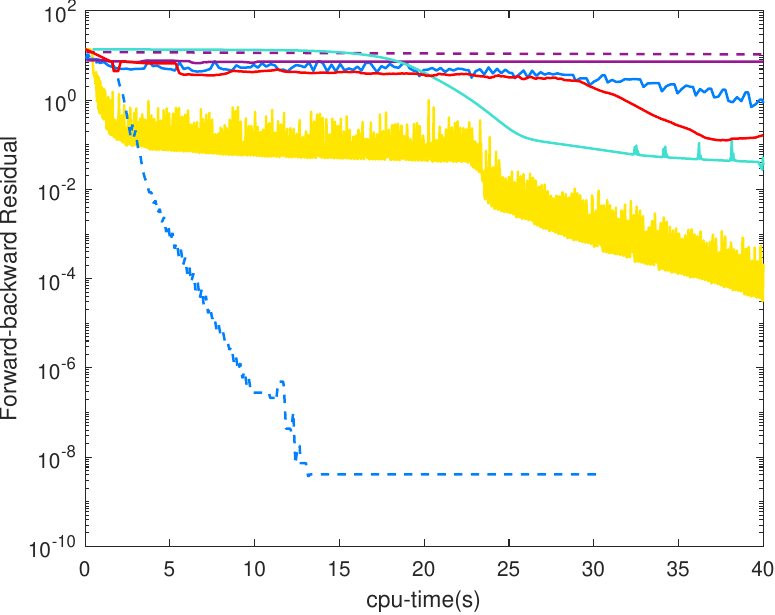}
   }
   \subfigure[$n=800$ -- cpu-time.]{
   \includegraphics[width=5.2cm]{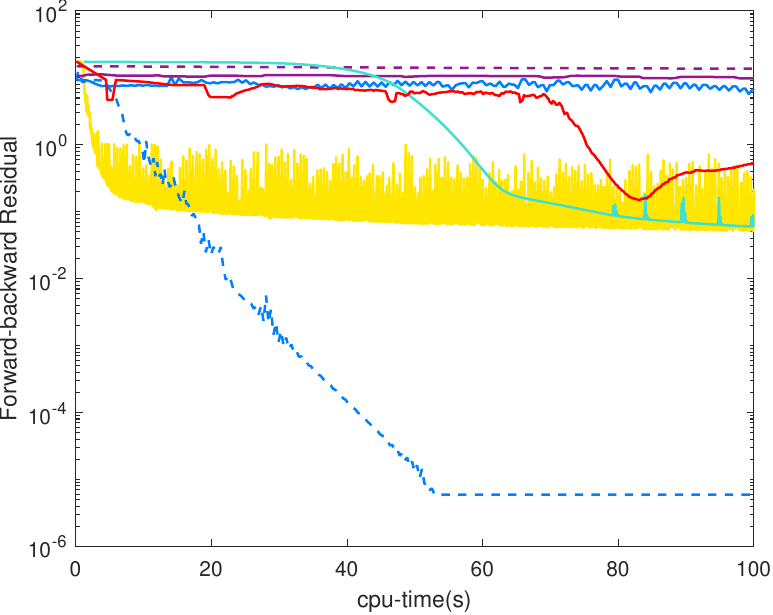}
  }
   \caption{Comparison of SpaRSA, ForBES, iPiano, TRSSN-\texttt{O(H)}, and TRSSN-(H) on problem \eqref{log_det}. Plot of the norm of the natural residual with respect to required cpu-time for different $n$.}
   \label{fig4-1}
   \end{figure} 

In most tests, TRSSN and TRSSN-\texttt{O(H)} tend to stagnate at a point where the auxiliary function $\psi\circ\proxs = \psi \circ \mathcal P_{\mathcal C}$ can no longer be improved. In fact, the projection $\mathcal P_{\mathcal C}(X)$ remains unchanged for small perturbations of $X$ if the perturbation is in the normal direction of $\mathcal C$ at $\mathcal P_{\mathcal C}(X)$. Hence, the L-BFGS approximations in TRSSN and TRSSN-\texttt{O(H)} only receive marginal updates and both algorithms eventually reduce to first-order schemes. The performance of the first-order methods iPiano and SpaRSA depends on $n$. SpaRSA outperforms iPiano when $n\in\{200,500\}$. However, iPiano seems to be more stable and performs better than SpaRSA for $n=800$. 
ForBES outperforms iPiano, but -- as indicated in \cref{fig4-1} and \cref{table3} -- is generally slower and less robust than SpaRSA and TRSSN-H. Notice that the eigendecomposition required to compute $\mathcal P_{\mathcal C}$ and to build the generalized derivative is the dominant computational cost of each iteration of TRSSN(-H) and TRSSN-\texttt{O(H)}. The evaluation of the Hessian mainly requires two (reusable) Cholesky decompositions and the inversion of triangular matrices which is less expensive than a full eigendecomposition. 
%
Thus, the cheaper L-BFGS approximations can not compensate the increased iteration numbers and prolonged convergence. Finally, as the numerical errors in the computed eigendecompositions increase with the dimension $n$, the overall achievable accuracy tends to decrease for larger choices of $n$.

\begin{table}[t!]
    \centering
    \setlength{\tabcolsep}{2pt}
    \begin{tabular}{cp{1pt}cccp{1pt}cccp{1pt}cccp{1pt}ccc}  
      \cmidrule[1pt](){1-17}  
      $n = 200$  & & \multicolumn{3}{c}{$\texttt{tol} = 10^{-2}$} & & \multicolumn{3}{c}{$\texttt{tol} = 10^{-4}$} & & \multicolumn{3}{c}{$\texttt{tol} = 10^{-6}$} & & \multicolumn{3}{c}{$\texttt{tol} = 10^{-8}$} \\ 
      \cmidrule[.5pt](){3-5} \cmidrule[.5pt](){7-9} \cmidrule[.5pt](){11-13} \cmidrule[.5pt](){15-17} 
      Algorithm  & & time & iter & $p_{\texttt{s}}$ & & time & iter & $p_{\texttt{s}}$ & & time & iter & $p_{\texttt{s}}$ & & time & iter & $p_{\texttt{s}}$  \\ 
       \cmidrule[.5pt](){1-17}
       \multicolumn{1}{l}{SpaRSA} && 0.72  & 155.8 & 100\% && 1.68 & 363.6 & 100\% && 2.73 & 586.8 & 80\% && - & - & 0\%  \\[0.5ex]   
       \multicolumn{1}{l}{ForBES} && 3.36 & 156.2  & 80\% && 4.41 & 205 & 60\% && 5.93 & 300 & 60\% && 7.17 & 395.3 & 60\%  \\[0.5ex]    
       \multicolumn{1}{l}{iPiano} && 5.72  & 1215.6  & 100\% && 6.19 & 1315.2 & 100\% && 6.33 & 1377 & 20\% && - & - & 0\%  \\[0.5ex] 
       \multicolumn{1}{l}{TRSSN} && 15.45 & 447  & 60\% && - & & 0\% && -  & - &  0\% && -  & - & 0\%  \\[0.5ex] 
       \multicolumn{1}{l}{TRSSN-H} && 0.52  & 36.4 & 100\% && 0.69 & 46.6 & 100\% && 0.88  & 57.4 & 100\%  && 1.15  & 80 & 100\% \\[2ex] 
%
        \cmidrule[1pt](){1-17}
        $n = 500$ \\
        \cmidrule[.5pt](){1-17}
        \multicolumn{1}{l}{SpaRSA} && 23.11 & 775.8 & 100\% && 37.24 & 1237.6 & 100\% && - & - & 0\% && - & - & 0\%  \\[0.5ex]   
        \multicolumn{1}{l}{ForBES} && 38.72 & 310& 40\% && 37.46 & 332 & 20\% && 46.82 & 455 & 20\% && - & - & 0\% \\[0.5ex]   
        \multicolumn{1}{l}{iPiano} && 41.25 &  1406.2 & 100\% && 49.54 & 1691.4 & 100\% && - & - & 0 \% && - & - & 0\%  \\[0.5ex] 
        \multicolumn{1}{l}{TRSSN} && 85.26 & 388.4 & 100\%  && - & -  & 0\% && - & -  & 0\% &&  - & -  & 0\%  \\[0.5ex] 
        \multicolumn{1}{l}{TRSSN-H} && 11.26 & 175 & 60\% && 13.83 &195 & 60\% && 16.93  & 219 & 60\%  && 12.40 & 120 & 20\% \\[2ex] 
%
%
        \cmidrule[1pt](){1-17}
        $n = 800$ \\
        \cmidrule[.5pt](){1-17}
        \multicolumn{1}{l}{SpaRSA} && - & - & 0\% && - & - & 0\% && - & - & 0\% && - & - & 0\% \\[0.5ex]
        \multicolumn{1}{l}{ForBES} && - & - & 0\% && - & - & 0\% && - & - & 0\% && - & - & 0\% \\[0.5ex]
        \multicolumn{1}{l}{iPiano} && 152.69 & 1976 & 80\% && - & - & 0\% && - & - & 0\% && - & - & 0\%   \\[0.5ex] 
        \multicolumn{1}{l}{TRSSN} && - & - & 0\% && - & - & 0\% && - & - & 0\% && - & - & 0\% \\[0.5ex]
        \multicolumn{1}{l}{TRSSN-H} && 21.61 & 78.2 & 100\% && 39.91 & 127.4 & 100\% && 60.53 & 195.2 & 80\% && 94.61 & 312 & 20\% \\[0.5ex] 
          \cmidrule[1pt](){1-17}\\[-1.5ex]
       \end{tabular}
    \caption{Results for \eqref{log_det} for 5 independent runs. We report the average cpu-time and it- erations required to achieve $\|F_{\mathrm{nat}}^1(X_k)\| \leq \texttt{tol}$ for different choices of \texttt{tol}. 
    The percentage $p_{\texttt{s}}$ is the success rate and shows how many runs satisfy the stopping criterion after a maxi- mum no. of iterations (2000 for first-order and 500 for other methods). The averages are taken with respect to successful runs.}
  \label{table3}
\end{table}

\section{Conclusion}
In this work, we propose a trust region-type semismooth Newton method (TRSSN) for solving a class of nonsmooth nonconvex optimization problems. Our approach is based on a reformulation of the associated first-order optimality conditions using Robinson's normal map. We construct a novel merit function and reduction ratio that allow to embed the computation of inexact semismooth Newton-type steps for the normal map in a trust region framework and that can guarantee strong convergence properties. In particular, the Kurdyka-{\L}ojasiewicz (KL) theory is applicable and transition to fast q-superlinear convergence can be established under appropriate local assumptions. We then show that the classical second-order necessary and sufficient optimality conditions for problem \cref{eq1-1} have an alternative normal map-based and explicit representation that only involves directional derivatives. This allows to link different optimality concepts, such as, strong metric subregularity for $\oFnat$ and $\oFnor$, the quadratic growth condition, and second-order sufficient conditions, and can justify some of the assumptions required for local convergence. Furthermore, motivated by its high practical relevance, we investigate the convergence properties of a quasi-Newton version of TRSSN. In particular, we show that BFGS updates and the KL framework are highly compatible and new boundedness results are derived under a significantly weaker finite length condition. We then combine our global and local results, the KL and second-order theory, and a Dennis-Mor{\'e} condition to establish superlinear convergence of the BFGS variant of TRSSN without requiring the strict complementarity condition. Finally, our numerical experiments on sparse logistic regression and an image compression problem indicate that TRSSN is a highly competitive algorithm on both convex and nonconvex problems.

\textbf{Acknowledgements}. The authors would like to thank Xiantao Xiao, Yongfeng Li, Zaiwen Wen, and Liwei Zhang for sharing their code ASSN. A. Milzarek thanks Konstantin Pieper and Florian Mannel for early discussions on normal map-based semismooth Newton approaches.

\bibliographystyle{siam}
\bibliography{sn-references}  

\begin{thebibliography}{100}

\bibitem{AbsMahAnd05}
{\sc P.-A. Absil, R.~Mahony, and B.~Andrews}, {\em Convergence of the iterates
  of descent methods for analytic cost functions}, SIAM J. Optim., 16 (2005),
  pp.~531--547.

\bibitem{ArtBelDonLop14}
{\sc F.~J. Arag\'{o}n~Artacho, A.~Belyakov, A.~L. Dontchev, and M.~L\'{o}pez},
  {\em Local convergence of quasi-{N}ewton methods under metric regularity},
  Comput. Optim. Appl., 58 (2014), pp.~225--247.

\bibitem{AraBarOrb21}
{\sc A.~Y. Aravkin, R.~Baraldi, and D.~Orban}, {\em A proximal quasi-{N}ewton
  trust-region method for nonsmooth regularized optimization}.
\newblock arXiv preprint, arXiv:2103.15993v2, 2021.

\bibitem{ArtGeo14}
{\sc F.~J.~A. Artacho and M.~H. Geoffroy}, {\em Metric subregularity of the
  convex subdifferential in {B}anach spaces}, J. Nonlinear Convex Anal., 15
  (2014), pp.~35--47.

\bibitem{AttBol09}
{\sc H.~Attouch and J.~Bolte}, {\em On the convergence of the proximal
  algorithm for nonsmooth functions involving analytic features}, Math.
  Program., 116 (2009), pp.~5--16.

\bibitem{AttBolRedSou10}
{\sc H.~Attouch, J.~Bolte, P.~Redont, and A.~Soubeyran}, {\em Proximal
  alternating minimization and projection methods for nonconvex problems: an
  approach based on the {K}urdyka-{{\L}}ojasiewicz inequality}, Math. Oper.
  Res., 35 (2010), pp.~438--457.

\bibitem{AttBolSva13}
{\sc H.~Attouch, J.~Bolte, and B.~F. Svaiter}, {\em Convergence of descent
  methods for semi-algebraic and tame problems: proximal algorithms,
  forward-backward splitting, and regularized {G}auss-{S}eidel methods}, Math.
  Program., 137 (2013), pp.~91--129.

\bibitem{bach2011optimization}
{\sc F.~Bach, R.~Jenatton, J.~Mairal, and G.~Obozinski}, {\em Optimization with
  sparsity-inducing penalties}, Found. and Trends{\textregistered} in Mach.
  Learn., 4 (2011), pp.~1--106.

\bibitem{BauCom11}
{\sc H.~H. Bauschke and P.~L. Combettes}, {\em Convex analysis and monotone
  operator theory in {H}ilbert spaces}, CMS Books in Mathematics/Ouvrages de
  Math\'ematiques de la SMC, Springer, New York, NY, USA, 2011.

\bibitem{beck2009fast}
{\sc A.~Beck and M.~Teboulle}, {\em A fast iterative shrinkage-thresholding
  algorithm for linear inverse problems}, SIAM J. Imaging Sci., 2 (2009),
  pp.~183--202.

\bibitem{BecFadOch19}
{\sc S.~Becker, J.~Fadili, and P.~Ochs}, {\em On quasi-{N}ewton
  forward-backward splitting: proximal calculus and convergence}, SIAM J.
  Optim., 29 (2019), pp.~2445--2481.

\bibitem{Ber16}
{\sc D.~P. Bertsekas}, {\em Nonlinear programming}, Athena Scientific
  Optimization and Computation Series, Athena Scientific, Belmont, MA, Nashua,
  NH, USA, third~ed., 2016.

\bibitem{Bis06}
{\sc C.~M. Bishop}, {\em Pattern recognition and machine learning}, Information
  Science and Statistics, Springer, New York, NY, USA, 2006.

\bibitem{BolDanLew06}
{\sc J.~Bolte, A.~Daniilidis, and A.~Lewis}, {\em The \l ojasiewicz inequality
  for nonsmooth subanalytic functions with applications to subgradient
  dynamical systems}, SIAM J. Optim., 17 (2006), pp.~1205--1223.

\bibitem{BolSabTeb14}
{\sc J.~Bolte, S.~Sabach, and M.~Teboulle}, {\em Proximal alternating
  linearized minimization for nonconvex and nonsmooth problems}, Math.
  Program., 146 (2014), pp.~459--494.

\bibitem{BonSha00}
{\sc J.~F. Bonnans and A.~Shapiro}, {\em Perturbation analysis of optimization
  problems}, Springer Series in Operations Research, Springer-Verlag, New York,
  New York, NY, USA, 2000.

\bibitem{BotCurNoc18}
{\sc L.~Bottou, F.~E. Curtis, and J.~Nocedal}, {\em Optimization methods for
  large-scale machine learning}, SIAM Rev., 60 (2018), pp.~223--311.

\bibitem{boulanger2017sparse}
{\sc A.-C. Boulanger and P.~Trautmann}, {\em Sparse optimal control of the
  {K}d{V}-{B}urgers equation on a bounded domain}, SIAM J. Control Optim., 55
  (2017), pp.~3673--3706.

\bibitem{BurQia00}
{\sc J.~V. Burke and M.~Qian}, {\em On the superlinear convergence of the
  variable metric proximal point algorithm using {B}royden and {BFGS} matrix
  secant updating}, Math. Program., 88 (2000), pp.~157--181.

\bibitem{ByrChiNocOzt16}
{\sc R.~H. Byrd, G.~M. Chin, J.~Nocedal, and F.~Oztoprak}, {\em A family of
  second-order methods for convex {$\ell_1$}-regularized optimization}, Math.
  Program., 159 (2016), pp.~435--467.

\bibitem{ByrKhaSch96}
{\sc R.~H. Byrd, H.~F. Khalfan, and R.~B. Schnabel}, {\em Analysis of a
  symmetric rank-one trust region method}, SIAM J. Optim., 6 (1996),
  pp.~1025--1039.

\bibitem{byrd1989tool}
{\sc R.~H. Byrd and J.~Nocedal}, {\em A tool for the analysis of quasi-{N}ewton
  methods with application to unconstrained minimization}, SIAM J. Numer.
  Anal., 26 (1989), pp.~727--739.

\bibitem{byrd1994representations}
{\sc R.~H. Byrd, J.~Nocedal, and R.~B. Schnabel}, {\em Representations of
  quasi-{N}ewton matrices and their use in limited memory methods}, Math.
  Program., 63 (1994), pp.~129--156.

\bibitem{ByrNocYua87}
{\sc R.~H. Byrd, J.~Nocedal, and Y.~X. Yuan}, {\em Global convergence of a
  class of quasi-{N}ewton methods on convex problems}, SIAM J. Numer. Anal., 24
  (1987), pp.~1171--1190.

\bibitem{cai2010singular}
{\sc J.-F. Cai, E.~J. Cand{\`e}s, and Z.~Shen}, {\em A singular value
  thresholding algorithm for matrix completion}, SIAM J. Optim., 20 (2010),
  pp.~1956--1982.

\bibitem{candes2009exact}
{\sc E.~J. Cand{\`e}s and B.~Recht}, {\em Exact matrix completion via convex
  optimization}, Found. Comput. Math., 9 (2009), p.~717.

\bibitem{CarGouToi11}
{\sc C.~Cartis, N.~I.~M. Gould, and P.~L. Toint}, {\em Adaptive cubic
  regularisation methods for unconstrained optimization. {P}art {I}:
  motivation, convergence and numerical results}, Math. Program., 127 (2011),
  pp.~245--295.

\bibitem{CheLiuSunToh16}
{\sc C.~Chen, Y.-J. Liu, D.~Sun, and K.-C. Toh}, {\em A semismooth
  {N}ewton-{CG} based dual {PPA} for matrix spectral norm approximation
  problems}, Math. Program., 155 (2016), pp.~435--470.

\bibitem{chen1997convergence}
{\sc G.~H. Chen and R.~T. Rockafellar}, {\em Convergence rates in
  forward--backward splitting}, SIAM J. Optim., 7 (1997), pp.~421--444.

\bibitem{Che97}
{\sc X.~Chen}, {\em Superlinear convergence of smoothing quasi-{N}ewton methods
  for nonsmooth equations}, J. Comput. Appl. Math., 80 (1997), pp.~105--126.

\bibitem{CheFuk99}
{\sc X.~Chen and M.~Fukushima}, {\em Proximal quasi-{N}ewton methods for
  nondifferentiable convex optimization}, Math. Program., 85 (1999),
  pp.~313--334.

\bibitem{CheQi94}
{\sc X.~Chen and L.~Qi}, {\em A parameterized {N}ewton method and a
  quasi-{N}ewton method for nonsmooth equations}, Comput. Optim. Appl., 3
  (1994), pp.~157--179.

\bibitem{CheYam92}
{\sc X.~Chen and T.~Yamamoto}, {\em On the convergence of some quasi-{N}ewton
  methods for nonlinear equations with nondifferentiable operators}, Computing,
  49 (1992), pp.~87--94.

\bibitem{chen2020trust}
{\sc Z.~Chen, A.~Milzarek, and Z.~Wen}, {\em A trust-region method for
  nonsmooth nonconvex optimization}, arXiv preprint arXiv:2002.08513,  (2020).

\bibitem{ChiHieNghTua19}
{\sc N.~H. Chieu, L.~V. Hien, T.~T.~A. Nghia, and H.~A. Tuan}, {\em Quadratic
  growth and strong metric subregularity of the subdifferential via subgradient
  graphical derivative}, SIAM J. Optim., 31 (2021), pp.~545--568.

\bibitem{ChrDLRMey20}
{\sc C.~Christof, J.~C. De~los Reyes, and C.~Meyer}, {\em A nonsmooth
  trust-region method for locally {L}ipschitz functions with application to
  optimization problems constrained by variational inequalities}, SIAM J.
  Optim., 30 (2020), pp.~2163--2196.

\bibitem{clarke1990optimization}
{\sc F.~H. Clarke}, {\em Optimization and nonsmooth analysis}, vol.~5 of
  Classics in Applied Mathematics, Society for Industrial and Applied
  Mathematics (SIAM), Philadelphia, PA, USA, second~ed., 1990.

\bibitem{ComWaj05}
{\sc P.~L. Combettes and V.~R. Wajs}, {\em Signal recovery by proximal
  forward-backward splitting}, Multiscale Model. Simul., 4 (2005),
  pp.~1168--1200 (electronic).

\bibitem{ConGouToi00}
{\sc A.~R. Conn, N.~I.~M. Gould, and P.~L. Toint}, {\em Trust-region methods},
  MPS/SIAM Series on Optimization, Society for Industrial and Applied
  Mathematics (SIAM); Mathematical Programming Society (MPS), Philadelphia, PA,
  USA, 2000.

\bibitem{cotter2005sparse}
{\sc S.~F. Cotter, B.~D. Rao, K.~Engan, and K.~Kreutz-Delgado}, {\em Sparse
  solutions to linear inverse problems with multiple measurement vectors}, IEEE
  Trans. Signal Process., 53 (2005), pp.~2477--2488.

\bibitem{DeLFacKan96}
{\sc T.~De~Luca, F.~Facchinei, and C.~Kanzow}, {\em A semismooth equation
  approach to the solution of nonlinear complementarity problems}, Math.
  Program., 75 (1996), pp.~407--439.

\bibitem{DenLiTap95}
{\sc J.~E. Dennis, Jr., S.-B.~B. Li, and R.~A. Tapia}, {\em A unified approach
  to global convergence of trust region methods for nonsmooth optimization},
  Math. Program., 68 (1995), pp.~319--346.

\bibitem{DenMor74}
{\sc J.~E. Dennis, Jr. and J.~J. Mor\'{e}}, {\em A characterization of
  superlinear convergence and its application to quasi-{N}ewton methods}, Math.
  Comp., 28 (1974), pp.~549--560.

\bibitem{DenMor77}
\leavevmode\vrule height 2pt depth -1.6pt width 23pt, {\em Quasi-{N}ewton
  methods, motivation and theory}, SIAM Rev., 19 (1977), pp.~46--89.

\bibitem{ding2020spectral}
{\sc C.~Ding, D.~Sun, J.~Sun, and K.-C. Toh}, {\em Spectral operators of
  matrices: semismoothness and characterizations of the generalized jacobian},
  SIAM Journal on Optimization, 30 (2020), pp.~630--659.

\bibitem{DinSunZha17}
{\sc C.~Ding, D.~Sun, and L.~Zhang}, {\em Characterization of the robust
  isolated calmness for a class of conic programming problems}, SIAM J. Optim.,
  27 (2017), pp.~67--90.

\bibitem{dirkse1995path}
{\sc S.~P. Dirkse and M.~C. Ferris}, {\em The path solver: a nonmonotone
  stabilization scheme for mixed complementarity problems}, Optim. Method
  Softw., 5 (1995), pp.~123--156.

\bibitem{donoho2006compressed}
{\sc D.~L. Donoho}, {\em Compressed sensing}, IEEE Trans. Inf. Theory, 52
  (2006), pp.~1289--1306.

\bibitem{DruIof15}
{\sc D.~Drusvyatskiy and A.~D. Ioffe}, {\em Quadratic growth and critical point
  stability of semi-algebraic functions}, Math. Program., 153 (2015),
  pp.~635--653.

\bibitem{DruLew18}
{\sc D.~Drusvyatskiy and A.~S. Lewis}, {\em Error bounds, quadratic growth, and
  linear convergence of proximal methods}, Math. Oper. Res., 43 (2018),
  pp.~919--948.

\bibitem{DruMorNhg14}
{\sc D.~Drusvyatskiy, B.~S. Mordukhovich, and T.~T.~A. Nhgia}, {\em
  Second-order growth, tilt stability, and metric regularity of the
  subdifferential}, J. Convex Anal., 21 (2014), pp.~1165--1192.

\bibitem{facchinei2007finite}
{\sc F.~Facchinei and J.-S. Pang}, {\em Finite-dimensional variational
  inequalities and complementarity problems}, Springer Series in Operations
  Research, Springer-Verlag, New York, NY, USA, 2003.

\bibitem{FerKanMun99}
{\sc M.~Ferris, C.~Kanzow, and T.~S. Munson}, {\em Feasible descent algorithms
  for mixed complementarity problems}, Math. Program., 86 (1999), pp.~475--497.

\bibitem{FerRal95}
{\sc M.~C. Ferris and D.~Ralph}, {\em Projected gradient methods for nonlinear
  complementarity problems via normal maps}, in Recent advances in nonsmooth
  optimization, World Sci. Publ., River Edge, NJ, USA, 1995, pp.~57--87.

\bibitem{FukMin81}
{\sc M.~Fukushima and H.~Mine}, {\em A generalized proximal point algorithm for
  certain nonconvex minimization problems}, Internat. J. Systems Sci., 12
  (1981), pp.~989--1000.

\bibitem{galic2008image}
{\sc I.~Gali{\'c}, J.~Weickert, M.~Welk, A.~Bruhn, A.~Belyaev, and H.-P.
  Seidel}, {\em Image compression with anisotropic diffusion}, J. Math. Imaging
  Vis., 31 (2008), pp.~255--269.

\bibitem{geng2014capacity}
{\sc Y.~Geng and C.~Nair}, {\em The capacity region of the two-receiver
  gaussian vector broadcast channel with private and common messages}, IEEE
  Transactions on Information Theory, 60 (2014), pp.~2087--2104.

\bibitem{GraYuaYua15}
{\sc G.~N. Grapiglia, J.~Yuan, and Y.-x. Yuan}, {\em On the convergence and
  worst-case complexity of trust-region and regularization methods for
  unconstrained optimization}, Math. Program., 152 (2015), pp.~491--520.

\bibitem{GriLor08}
{\sc R.~Griesse and D.~A. Lorenz}, {\em A semismooth {N}ewton method for
  {T}ikhonov functionals with sparsity constraints}, Inverse Problems, 24
  (2008), pp.~035007, 19.

\bibitem{han2015large}
{\sc I.~Han, D.~Malioutov, and J.~Shin}, {\em Large-scale log-determinant
  computation through stochastic chebyshev expansions}, in International
  Conference on Machine Learning, PMLR, 2015, pp.~908--917.

\bibitem{HanSun97}
{\sc J.~Han and D.~Sun}, {\em Newton and quasi-{N}ewton methods for normal maps
  with polyhedral sets}, J. Optim. Theory Appl., 94 (1997), pp.~659--676.

\bibitem{HanPanRan92}
{\sc S.-P. Han, J.-S. Pang, and N.~Rangaraj}, {\em Globally convergent {N}ewton
  methods for nonsmooth equations}, Math. Oper. Res., 17 (1992), pp.~586--607.

\bibitem{HanRaa15}
{\sc E.~Hans and T.~Raasch}, {\em Global convergence of damped semismooth
  {N}ewton methods for {$\ell_1$} {T}ikhonov regularization}, Inverse Problems,
  31 (2015), pp.~025005, 31.

\bibitem{HesSti52}
{\sc M.~R. Hestenes and E.~Stiefel}, {\em Methods of conjugate gradients for
  solving linear systems}, J. Research Nat. Bur. Standards, 49 (1952),
  pp.~409--436 (1953).

\bibitem{ip1992local}
{\sc C.-M. Ip and J.~Kyparisis}, {\em Local convergence of quasi-{N}ewton
  methods for {B}-differentiable equations}, Math. Program., 56 (1992),
  pp.~71--89.

\bibitem{JiaSunToh14}
{\sc K.~Jiang, D.~Sun, and K.-C. Toh}, {\em A partial proximal point algorithm
  for nuclear norm regularized matrix least squares problems}, Math. Program.
  Comput., 6 (2014), pp.~281--325.

\bibitem{JinMok21}
{\sc Q.~Jin and A.~Mokhtari}, {\em Non-asymptotic superlinear convergence of
  standard quasi-{N}ewton methods}.
\newblock arXiv preprint, arXiv:2003.13607v2, 2021.

\bibitem{Kaa88}
{\sc E.~F. Kaasschieter}, {\em Preconditioned conjugate gradients for solving
  singular systems}, vol.~24, 1988, pp.~265--275.
\newblock Iterative methods for the solution of linear systems.

\bibitem{KanLec21}
{\sc C.~Kanzow and T.~Lechner}, {\em Globalized inexact proximal {N}ewton-type
  methods for nonconvex composite functions}, Comput. Optim. Appl., 78 (2021),
  pp.~377--410.

\bibitem{KanQi99}
{\sc C.~Kanzow and H.-D. Qi}, {\em A {QP}-free constrained {N}ewton-type method
  for variational inequality problems}, Math. Program., 85 (1999), pp.~81--106.

\bibitem{kunisch2016time}
{\sc K.~Kunisch, K.~Pieper, and A.~Rund}, {\em Time optimal control for a
  reaction diffusion system arising in cardiac electrophysiology--a monolithic
  approach}, ESAIM Math. Model. Numer. Anal., 50 (2016), pp.~381--414.

\bibitem{kurdyka1998}
{\sc K.~Kurdyka}, {\em On gradients of functions definable in o-minimal
  structures}, in Annales de l'institut Fourier, vol.~48, 1998, pp.~769--783.

\bibitem{lau2022uniqueness}
{\sc C.~W.~K. Lau, C.~Nair, and C.~Yao}, {\em Uniqueness of local maximizers
  for some non-convex log-determinant optimization problems using information
  theory}, in 2022 IEEE International Symposium on Information Theory (ISIT),
  IEEE, 2022, pp.~432--437.

\bibitem{lee2014proximal}
{\sc J.~D. Lee, Y.~Sun, and M.~A. Saunders}, {\em Proximal {N}ewton-type
  methods for minimizing composite functions}, SIAM J. Optim., 24 (2014),
  pp.~1420--1443.

\bibitem{LemSag96}
{\sc C.~Lemar{\'e}chal and C.~Sagastiz{\'a}bal}, {\em More than first-order
  developments of convex functions: primal-dual relations}, J. Convex Anal., 3
  (1996), pp.~255--268.

\bibitem{LewOve13}
{\sc A.~S. Lewis and M.~L. Overton}, {\em Nonsmooth optimization via
  quasi-{N}ewton methods}, Math. Program., 141 (2013), pp.~135--163.

\bibitem{LiYamFuk01}
{\sc D.~H. Li, N.~Yamashita, and M.~Fukushima}, {\em Nonsmooth equation based
  {BFGS} method for solving {KKT} systems in mathematical programming}, J.
  Optim. Theory Appl., 109 (2001), pp.~123--167.

\bibitem{LiSunToh18}
{\sc X.~Li, D.~Sun, and K.-C. Toh}, {\em A highly efficient semismooth {N}ewton
  augmented {L}agrangian method for solving lasso problems}, SIAM J. Optim., 28
  (2018), pp.~433--458.

\bibitem{lojasiewicz1963}
{\sc S.~{\L}ojasiewicz}, {\em Une propri{\'e}t{\'e} topologique des
  sous-ensembles analytiques r{\'e}els}, Les {\'e}quations aux d{\'e}riv{\'e}es
  partielles, 117 (1963), pp.~87--89.

\bibitem{lojasiewicz1993}
\leavevmode\vrule height 2pt depth -1.6pt width 23pt, {\em Sur la
  g{\'e}om{\'e}trie semi- et sous-analytique}, Annales de l'institut Fourier
  (Grenoble), 43 (1993), pp.~1575--1595.

\bibitem{mai2019anderson}
{\sc V.~Mai and M.~Johansson}, {\em Anderson acceleration of proximal gradient
  methods}, in Proceedings of the 37th International Conference on Machine
  Learning, Virtual, 13--18 Jul 2020, PMLR, pp.~6620--6629.

\bibitem{MaiBruWeiFor11}
{\sc M.~Mainberger, A.~Bruhn, J.~Weickert, and S.~Forchhammer}, {\em Edge-based
  compression of cartoon-like images with homogeneous diffusion}, Pattern
  Recognition, 44 (2011), pp.~1859--1873.

\bibitem{mairal2009online}
{\sc J.~Mairal, F.~Bach, J.~Ponce, and G.~Sapiro}, {\em Online dictionary
  learning for sparse coding}, in Proceedings of the 26th International
  Conference on Machine Learning, New York, NY, USA, 14--18 Jun 2009,
  Association for Computing Machinery, pp.~689--696.

\bibitem{Man21}
{\sc F.~Mannel}, {\em Convergence properties of the broyden-like method for
  mixed linear--nonlinear systems of equations}, Numer. Algorithms,  (2021).

\bibitem{ManRun20}
{\sc F.~Mannel and A.~Rund}, {\em A hybrid semismooth quasi-newton method for
  nonsmooth optimal control with pdes}, Opt. Eng.,  (2020).

\bibitem{mannelhybrid}
\leavevmode\vrule height 2pt depth -1.6pt width 23pt, {\em A hybrid semismooth
  quasi-{N}ewton method for structured nonsmooth operator equations in banach
  spaces}.
\newblock preprint, available at
  \url{https://imsc.uni-graz.at/mannel/sqn1.pdf}, 2021.

\bibitem{MarQi95}
{\sc J.~M. Mart\'{\i}nez and L.~Q. Qi}, {\em Inexact {N}ewton methods for
  solving nonsmooth equations}, vol.~60(1-2), 1995, pp.~127--145.
\newblock Linear/nonlinear iterative methods and verification of solution
  (Matsuyama, 1993).

\bibitem{meier2008group}
{\sc L.~Meier, S.~Van De~Geer, and P.~B{\"u}hlmann}, {\em The group lasso for
  logistic regression}, J. R. Stat. Soc. Ser. B-Stat. Methodol., 70 (2008),
  pp.~53--71.

\bibitem{MenSunZha05}
{\sc F.~Meng, D.~Sun, and G.~Zhao}, {\em Semismoothness of solutions to
  generalized equations and the {M}oreau-{Y}osida regularization}, Math.
  Program., 104 (2005), pp.~561--581.

\bibitem{Mif77}
{\sc R.~Mifflin}, {\em Semismooth and semiconvex functions in constrained
  optimization}, SIAM J. Control Optim., 15 (1977), pp.~959--972.

\bibitem{milzarek2016numerical}
{\sc A.~Milzarek}, {\em Numerical methods and second order theory for nonsmooth
  problems}, PhD thesis, Technische Universit{\"a}t M{\"u}nchen, 2016.

\bibitem{milzarek2014semismooth}
{\sc A.~Milzarek and M.~Ulbrich}, {\em A semismooth newton method with
  multidimensional filter globalization for l\_1-optimization}, SIAM J. Optim.,
  24 (2014), pp.~298--333.

\bibitem{MohMorSar19}
{\sc A.~Mohammadi, B.~S. Mordukhovich, and M.~E. Sarabi}, {\em Parabolic
  regularity in geometric variational analysis}, Trans. Amer. Math. Soc., 374
  (2021), pp.~1711--1763.

\bibitem{MohSar20}
{\sc A.~Mohammadi and M.~E. Sarabi}, {\em Twice {EPI}-differentiability of
  extended-real-valued functions with applications in composite optimization},
  SIAM J. Optim., 30 (2020), pp.~2379--2409.

\bibitem{Mor65}
{\sc J.-J. Moreau}, {\em Proximit\'e et dualit\'e dans un espace hilbertien},
  Bull. Soc. Math. FR., 93 (1965), pp.~273--299.

\bibitem{MunFacFerFisKan01}
{\sc T.~S. Munson, F.~Facchinei, M.~C. Ferris, A.~Fischer, and C.~Kanzow}, {\em
  The semismooth algorithm for large scale complementarity problems}, INFORMS
  J. Comput., 13 (2001), pp.~294--311.

\bibitem{NocWri06}
{\sc J.~Nocedal and S.~J. Wright}, {\em Numerical Optimization}, Springer
  Series in Operations Research and Financial Engineering, Springer, New York,
  second~ed., 2006.

\bibitem{Nol10}
{\sc D.~Noll}, {\em Cutting plane oracles to minimize non-smooth non-convex
  functions}, Set-Valued Var. Anal., 18 (2010), pp.~531--568.

\bibitem{NolRon13}
{\sc D.~Noll and A.~Rondepierre}, {\em Convergence of linesearch and
  trust-region methods using the {K}urdyka-{{\L}}ojasiewicz inequality}, in
  Computational and analytical mathematics, vol.~50 of Springer Proc. Math.
  Stat., Springer, New York, NY, USA, 2013, pp.~593--611.

\bibitem{OchCheBroPoc14}
{\sc P.~Ochs, Y.~Chen, T.~Brox, and T.~Pock}, {\em i{P}iano: inertial proximal
  algorithm for nonconvex optimization}, SIAM J. Imaging Sci., 7 (2014),
  pp.~1388--1419.

\bibitem{parikh2014proximal}
{\sc N.~Parikh and S.~Boyd}, {\em Proximal algorithms}, Found. and
  Trends{\textregistered} in Optim., 1 (2014), pp.~127--239.

\bibitem{PatBem13}
{\sc P.~Patrinos and A.~Bemporad}, {\em Proximal {N}ewton methods for convex
  composite optimization}, in 52nd IEEE Conference on Decision and Control,
  2013, pp.~2358--2363.

\bibitem{PatSteBem14}
{\sc P.~Patrinos, L.~Stella, and A.~Bemporad}, {\em Forward-backward truncated
  {N}ewton methods for convex composite optimization}.
\newblock ArXiv:1402.6655, 2 2014.

\bibitem{pieper2015finite}
{\sc K.~Pieper}, {\em Finite element discretization and efficient numerical
  solution of elliptic and parabolic sparse control problems}, PhD thesis,
  Technische Universit{\"a}t M{\"u}nchen, 2015.

\bibitem{PolRoc92}
{\sc R.~A. Poliquin and R.~T. Rockafellar}, {\em Amenable functions in
  optimization},  (1992), pp.~338--353.

\bibitem{PolRoc93}
\leavevmode\vrule height 2pt depth -1.6pt width 23pt, {\em A calculus of
  epi-derivatives applicable to optimization}, Canad. J. Math., 45 (1993),
  pp.~879--896.

\bibitem{PolRoc96-2}
\leavevmode\vrule height 2pt depth -1.6pt width 23pt, {\em Generalized
  {H}essian properties of regularized nonsmooth functions}, SIAM J. Optim., 6
  (1996), pp.~1121--1137.

\bibitem{PolRoc96}
\leavevmode\vrule height 2pt depth -1.6pt width 23pt, {\em Prox-regular
  functions in variational analysis}, Trans. Amer. Math. Soc., 348 (1996),
  pp.~1805--1838.

\bibitem{Pow74}
{\sc M.~J.~D. Powell}, {\em Convergence properties of a class of minimization
  algorithms}, in Nonlinear programming, 2 ({P}roc. {S}ympos. {S}pecial
  {I}nterest {G}roup on {M}ath. {P}rogramming, {U}niv. {W}isconsin, {M}adison,
  {W}is., 1974), 1974, pp.~1--27.

\bibitem{Pow84}
\leavevmode\vrule height 2pt depth -1.6pt width 23pt, {\em On the global
  convergence of trust region algorithms for unconstrained minimization}, Math.
  Program., 29 (1984), pp.~297--303.

\bibitem{Pow10}
\leavevmode\vrule height 2pt depth -1.6pt width 23pt, {\em On the convergence
  of a wide range of trust region methods for unconstrained optimization}, IMA
  J. Numer. Anal., 30 (2010), pp.~289--301.

\bibitem{qi1993convergence}
{\sc L.~Qi}, {\em Convergence analysis of some algorithms for solving nonsmooth
  equations}, Math. Oper. Res., 18 (1993), pp.~227--244.

\bibitem{Qi97}
\leavevmode\vrule height 2pt depth -1.6pt width 23pt, {\em On superlinear
  convergence of quasi-{N}ewton methods for nonsmooth equations}, Oper. Res.
  Lett., 20 (1997), pp.~223--228.

\bibitem{QiJia97}
{\sc L.~Qi and H.~Jiang}, {\em Semismooth {K}arush-{K}uhn-{T}ucker equations
  and convergence analysis of {N}ewton and quasi-{N}ewton methods for solving
  these equations}, Math. Oper. Res., 22 (1997), pp.~301--325.

\bibitem{QiSun99}
{\sc L.~Qi and D.~Sun}, {\em A survey of some nonsmooth equations and smoothing
  {N}ewton methods}, in Progress in {O}ptimization, vol.~30 of Appl. Optim.,
  Kluwer Acad. Publ., Dordrecht, NL, 1999, pp.~121--146.

\bibitem{QiSun93}
{\sc L.~Qi and J.~Sun}, {\em A nonsmooth version of {N}ewton's method}, Math.
  Program., 58 (1993), pp.~353--367.

\bibitem{QiSun94}
{\sc L.~Q. Qi and J.~Sun}, {\em A trust region algorithm for minimization of
  locally {L}ipschitzian functions}, Math. Program., 66 (1994), pp.~25--43.

\bibitem{ralph1994global}
{\sc D.~Ralph}, {\em Global convergence of damped {N}ewton's method for
  nonsmooth equations via the path search}, Math. Oper. Res., 19 (1994),
  pp.~352--389.

\bibitem{RauFuk00}
{\sc A.~I. Rauf and M.~Fukushima}, {\em Globally convergent {BFGS} method for
  nonsmooth convex optimization}, J. Optim. Theory Appl., 104 (2000),
  pp.~539--558.

\bibitem{ren1983convergence}
{\sc G.~Ren-Pu and M.~J. Powell}, {\em The convergence of variable metric
  matrices in unconstrained optimization}, Math. Program., 27 (1983), p.~123.

\bibitem{Rit79}
{\sc K.~Ritter}, {\em Local and superlinear convergence of a class of variable
  metric methods}, Computing, 23 (1979), pp.~287--297.

\bibitem{Rit81}
{\sc K.~Ritter}, {\em Global and superlinear convergence of a class of variable
  metric methods}, Math. Program. Stud.,  (1981), pp.~178--205.

\bibitem{robinson1992normal}
{\sc S.~M. Robinson}, {\em Normal maps induced by linear transformations},
  Math. Oper. Res., 17 (1992), pp.~691--714.

\bibitem{rockafellar1970convex}
{\sc R.~T. Rockafellar}, {\em Convex analysis}, Princeton Mathematical Series,
  No. 28, Princeton University Press, Princeton, NJ, USA, 1970.

\bibitem{Roc88}
{\sc R.~T. Rockafellar}, {\em First- and second-order epi-differentiability in
  nonlinear programming}, Trans. Amer. Math. Soc., 307 (1988), pp.~75--108.

\bibitem{rockafellar2009variational}
{\sc R.~T. Rockafellar and R.~J.-B. Wets}, {\em Variational analysis}, vol.~317
  of Grundlehren der Mathematischen Wissenschaften [Fundamental Principles of
  Mathematical Sciences], Springer-Verlag, Berlin, DE, 2009.

\bibitem{RodNes21-1}
{\sc A.~Rodomanov and Y.~Nesterov}, {\em New {R}esults on {S}uperlinear
  {C}onvergence of {C}lassical {Q}uasi-{N}ewton {M}ethods}, J. Optim. Theory
  Appl., 188 (2021), pp.~744--769.

\bibitem{RodNes21}
\leavevmode\vrule height 2pt depth -1.6pt width 23pt, {\em Rates of superlinear
  convergence for classical quasi-newton methods}, Math. Program.,  (2021).

\bibitem{RunAigKunSto18a}
{\sc A.~Rund, C.~S. Aigner, K.~Kunisch, and R.~Stollberger}, {\em Magnetic
  resonance {RF} pulse design by optimal control with physical constraints},
  IEEE Trans. Med. Imaging, 37 (2018), pp.~461--472.

\bibitem{Sac85}
{\sc E.~Sachs}, {\em Convergence rates of quasi-{N}ewton algorithms for some
  nonsmooth optimization problems}, SIAM J. Control Optim., 23 (1985),
  pp.~401--418.

\bibitem{schmaltz2009beating}
{\sc C.~Schmaltz, J.~Weickert, and A.~Bruhn}, {\em Beating the quality of
  {JPEG} 2000 with anisotropic diffusion}, in Joint Pattern Recognition
  Symposium, Berlin, DE, 2009, Springer, pp.~452--461.

\bibitem{shalev2014understanding}
{\sc S.~Shalev-Shwartz and S.~Ben-David}, {\em Understanding machine learning:
  From theory to algorithms}, Cambridge University Press, Cambridge, UK, 2014.

\bibitem{Sha03}
{\sc A.~Shapiro}, {\em On a class of nonsmooth composite functions}, Math.
  Oper. Res., 28 (2003), pp.~677--692.

\bibitem{shevade2003simple}
{\sc S.~K. Shevade and S.~S. Keerthi}, {\em A simple and efficient algorithm
  for gene selection using sparse logistic regression}, Bioinformatics, 19
  (2003), pp.~2246--2253.

\bibitem{SolSva01}
{\sc M.~V. Solodov and B.~F. Svaiter}, {\em A unified framework for some
  inexact proximal point algorithms}, Numer. Funct. Anal. Optim., 22 (2001),
  pp.~1013--1035.

\bibitem{steihaug1983conjugate}
{\sc T.~Steihaug}, {\em The conjugate gradient method and trust regions in
  large scale optimization}, SIAM J. Numer. Anal., 20 (1983), pp.~626--637.

\bibitem{SteThePat17}
{\sc L.~Stella, A.~Themelis, and P.~Patrinos}, {\em Forward--backward
  quasi-{N}ewton methods for nonsmooth optimization problems}, Comput. Optim.
  Appl., 67 (2017), pp.~443--487.

\bibitem{Sto84}
{\sc J.~Stoer}, {\em The convergence of matrices generated by rank-{$2$}
  methods from the restricted {$\beta $}-class of {B}royden}, Numer. Math., 44
  (1984), pp.~37--52.

\bibitem{sun2006strong}
{\sc D.~Sun}, {\em The strong second-order sufficient condition and constraint
  nondegeneracy in nonlinear semidefinite programming and their implications},
  Mathematics of Operations Research, 31 (2006), pp.~761--776.

\bibitem{SunHan97}
{\sc D.~Sun and J.~Han}, {\em Newton and quasi-{N}ewton methods for a class of
  nonsmooth equations and related problems}, SIAM J. Optim., 7 (1997),
  pp.~463--480.

\bibitem{SunSun02}
{\sc D.~Sun and J.~Sun}, {\em Strong semismoothness of eigenvalues of symmetric
  matrices and its application to inverse eigenvalue problems}, SIAM J. Numer.
  Anal., 40 (2002), pp.~2352--2367 (2003).

\bibitem{TheStePat18}
{\sc A.~Themelis, L.~Stella, and P.~Patrinos}, {\em Forward-backward envelope
  for the sum of two nonconvex functions: further properties and nonmonotone
  linesearch algorithms}, SIAM J. Optim., 28 (2018), pp.~2274--2303.

\bibitem{tibshirani1996regression}
{\sc R.~Tibshirani}, {\em Regression shrinkage and selection via the lasso}, J.
  R. Stat. Soc. Ser. B-Stat. Methodol., 58 (1996), pp.~267--288.

\bibitem{ulbrich2001nonmonotone}
{\sc M.~Ulbrich}, {\em Nonmonotone trust-region methods for bound-constrained
  semismooth equations with applications to nonlinear mixed complementarity
  problems}, SIAM J. Optim., 11 (2001), pp.~889--917.

\bibitem{ulbrich2011semismooth}
\leavevmode\vrule height 2pt depth -1.6pt width 23pt, {\em Semismooth {N}ewton
  methods for variational inequalities and constrained optimization problems in
  function spaces}, vol.~11 of MOS-SIAM Series on Optimization, Society for
  Industrial and Applied Mathematics (SIAM); Mathematical Optimization Society,
  Philadelphia, PA, USA, 2011.

\bibitem{WenYinGolZha10}
{\sc Z.~Wen, W.~Yin, D.~Goldfarb, and Y.~Zhang}, {\em A fast algorithm for
  sparse reconstruction based on shrinkage, subspace optimization, and
  continuation}, SIAM J. Sci. Comput., 32 (2010), pp.~1832--1857.

\bibitem{Wer79}
{\sc J.~Werner}, {\em \"{U}ber die globale {K}onvergenz von
  {V}ariable-{M}etrik-{V}erfahren mit nicht-exakter {S}chrittweitenbestimmung},
  Numer. Math., 31 (1978/79), pp.~321--334.

\bibitem{WriNowFig09}
{\sc S.~J. Wright, R.~D. Nowak, and M.~A.~T. Figueiredo}, {\em Sparse
  reconstruction by separable approximation}, IEEE Trans. Signal Process., 57
  (2009), pp.~2479--2493.

\bibitem{xiao2018regularized}
{\sc X.~Xiao, Y.~Li, Z.~Wen, and L.~Zhang}, {\em A regularized semi-smooth
  {N}ewton method with projection steps for composite convex programs}, J. Sci.
  Comput., 76 (2018), pp.~364--389.

\bibitem{YanSunToh15}
{\sc L.~Yang, D.~Sun, and K.-C. Toh}, {\em {${\rm SDPNAL}+$}: a majorized
  semismooth {N}ewton-{CG} augmented {L}agrangian method for semidefinite
  programming with nonnegative constraints}, Math. Program. Comput., 7 (2015),
  pp.~331--366.

\bibitem{YuLiPon19}
{\sc P.~Yu, G.~Li, and T.~K. Pong}, {\em {Kurdyka}-{{\L}}ojasiewicz exponent
  via inf-projection}, Found. Comput. Math.,  (2021).

\bibitem{yu2013decomposing}
{\sc Y.-L. Yu}, {\em On decomposing the proximal map}, in Advances in Neural
  Information Processing Systems, 2013, pp.~91--99.

\bibitem{yuan2006model}
{\sc M.~Yuan and Y.~Lin}, {\em Model selection and estimation in regression
  with grouped variables}, J. R. Stat. Soc. Ser. B-Stat. Methodol., 68 (2006),
  pp.~49--67.

\bibitem{Yua85}
{\sc Y.~Yuan}, {\em Conditions for convergence of trust region algorithms for
  nonsmooth optimization}, Math. Program., 31 (1985), pp.~220--228.

\bibitem{ZhaHag04}
{\sc H.~Zhang and W.~W. Hager}, {\em A nonmonotone line search technique and
  its application to unconstrained optimization}, SIAM J. Optim., 14 (2004),
  pp.~1043--1056 (electronic).

\bibitem{ZhaSunToh10}
{\sc X.-Y. Zhao, D.~Sun, and K.-C. Toh}, {\em A {N}ewton-{CG} augmented
  {L}agrangian method for semidefinite programming}, SIAM J. Optim., 20 (2010),
  pp.~1737--1765.

\bibitem{ZhoTohSun03}
{\sc G.~Zhou, K.~C. Toh, and D.~Sun}, {\em Globally and quadratically
  convergent algorithm for minimizing the sum of {E}uclidean norms}, J. Optim.
  Theory Appl., 119 (2003), pp.~357--377.

\end{thebibliography}
\appendix
\section{Proofs of Auxiliary Results}
\subsection{Proof of \cref{lemma3-8}} \label{sec:app:pf-38}

\begin{proof} \cref{algo1} obviously returns $q$ with $\|q\| \leq \Delta$. Moreover, using $D^\top = \Lambda D \Lambda^{-1}$, we have $\mathcal R(S) \subseteq \mathcal R(\Lambda D)$, $\mathrm{dim}~\mathcal R(S) \leq m$, and $g \in \mathcal R(\Lambda D)$. Let us consider the iteration $i = m-1$ and suppose that \cref{algo1} does not terminate in step 6 or 10. Following \cite{HesSti52,NocWri06}, it can be shown that the CG-method has generated a sequence of vectors $\{r_0,r_1,...,r_{m-1}\}$ and $\{p_0,p_1,...,p_{m-1}\}$ with the properties
\be \label{eq:cg-prop} \iprod{p_\ell}{Sp_j} = 0, \quad \iprod{r_\ell}{p_j} = 0, \quad \mathrm{span}\{r_0,...,r_{\ell}\} = \mathrm{span}\{p_0,...,p_{\ell}\} = \mathcal K^\ell(S,r_0), \ee
for all $j = 0,...,\ell - 1$ and all $\ell = 1,...,m-1$. Here, $\mathcal K^\ell(S,r_0)$ denotes the Krylov space $\mathrm{span}\{r_0, Sr_0,...,S^\ell r_0\}$. By construction, we have $p_\ell, r_\ell \in \mathcal R(\Lambda D)$ and $\mathcal K^\ell(S,r_0) \subseteq \mathcal R(\Lambda D)$ for all $\ell = 0,...,m-1$. Moreover, due to the $S$-orthogonality and $\iprod{p_\ell}{Sp_\ell} > 0$, $\ell = 0,...,m-1$, the vectors $\{p_0,p_1,...,p_{m-1}\}$ are linearly independent and hence, it follows 
\[ \mathrm{span}\{p_0,p_1,...,p_{m-1}\} = \mathcal R(\Lambda D). \]
Using the second equality in \cref{eq:cg-prop}, we can infer $\mathcal R(\Lambda D) \ni r_m \, \bot \, \mathrm{span}\{p_0,p_1,...,p_{m-1}\} = \mathcal R(\Lambda D)$ which implies $r_m = r_0 + S q_m = 0$. Consequently, \cref{algo1} would stop at iteration $m$ with $q = q_m$ satisfying $Sq = - g$ and $\|q\| \leq \Delta$. This finishes the proof of the first part. In order to prove the second part, we need to verify that \cref{algo1} never terminates in step 6 or 10. If $S$ is positive semidefinite and $M$ is invertible, the condition $\iprod{p_i}{Sp_i} \leq 0$ and the proof of \cref{lemma3-7} imply $Dp_i = 0$. However, by \citep[Lemma 2.2]{steihaug1983conjugate}, we have $\|r_i\|^2 = -\iprod{r_i}{p_i} = -\iprod{r_0}{p_i} = - \iprod{\Fnor{z}}{Dp_i} = 0$. Hence, \cref{algo1} would exit at the $(i-1)$-th iteration when checking the norm of the residual $\|r_i\|$ and therefore the method can not terminate at step 6. Now, due to the invertibility of $M$, we have $\mathcal R(S) = \mathcal R(\Lambda D)$, $g = r_0 \in \mathcal R(S)$, and $r_\ell, p_\ell, q_\ell \in \mathcal R(S)$ for all $\ell$. According to part (i), \cref{algo1} would generate a solution $q = q_m$ satisfying $Sq = -g$ with $q \in \mathcal R(S)$ unless it stops earlier in step 10. Since $q$ lies in $\mathcal R(S)$, it has to coincide with the minimum norm solution $\hat q = -S^+g$ where $S^+$ denotes the Moore-Penrose inverse of $S$, see, e.g., \cite{Kaa88} for comparison. Moreover, due to \citep[Theorem 2.1]{steihaug1983conjugate}, it follows 
\[ 0 = \|q_0\| < \|q_1\| < ... < \|q_j\| < ... < \|q_m\| = \|\hat q\| \leq \|M^{-1}\Fnor{z}\| \leq \Delta. \]
This shows that \cref{algo1} can only terminate in step 14 which completes the proof.
\end{proof}

\subsection{Proof of \cref{lemma:matrix_extension}} \label{sec:app:mat-ex}

\begin{proof} Let $U \in \R^{n \times r}$ and $V \in \R^{n \times n-r}$ be two orthonormal basis matrices of the subspaces $\mathcal W$ and $\mathcal W^\bot$, respectively. Thus, by definition, we have $U^\top V = 0$ and introducing the orthogonal matrix $Q = (U,V) \in \R^{n \times n}$, we can write
\[ H(x) = Q \begin{pmatrix} U^\top H(x) U & U^\top H(x) V \\ V^\top H(x) U & V^\top H(x) V \end{pmatrix} Q^\top. \]
Our assumption then implies that the matrix $U^\top H(x) U$ is positive definite with $\lambda_{\min}(U^\top H(x) U)\geq \delta$ for all $x \in B_\epsilon(w)$. We now define $G : \Rn \to \mathbb S^n$ via 
\[ G(x) := Q\begin{pmatrix} U^\top H(x) U & U^\top H(x) V \\ V^\top H(x) U & \gamma(x) I \end{pmatrix}Q^\top, \quad \gamma(x) := \frac{2}{\delta}\|U^\top H(x) V\|^2+\frac{\delta}{2}. \]
Now, for every $d \in \mathcal W$ there exists $\xi \in \R^r$ such that $d = U \xi$ and thus, we obtain
\[ G(x) d = Q \begin{pmatrix} U^\top H(x) U & U^\top H(x) V \\ V^\top H(x) U & \gamma(x) I \end{pmatrix}\begin{pmatrix} \xi \\ 0 \end{pmatrix} = Q \begin{pmatrix} U^\top H(x)d \\ V^\top H(x)d \end{pmatrix} = H(x)d. \]
Next, let $d = UU^\top d + VV^\top d \in \Rn$ be arbitrary. Using the symmetry of $H(x)$ and Young's inequality, it holds that:
\begin{align*}
d^\top G(x) d &=(U^\top d)^\top [U^\top H(x) U] U^\top d + 2 (U^\top d)^\top [ U^\top H(x) V] V^\top d + \gamma(x) \|V^\top d\|^2 \\ &\hspace{-4ex}\geq  \delta\| U^\top d\|^2-\frac{\delta}{2}\|U^\top d\|^2-\frac{2}{\delta}\|U^\top H(x) V\|^2\|V^\top d \|^2+\gamma(x) \|V^\top d\|^2 = \frac{\delta}{2} [ \|U^\top d\|^2+\|V^\top d\|^2] = \frac{\delta}{2} \|d\|^2.
\end{align*}
Finally, let us suppose that $H$ is Lipschitz continuous on $B_\epsilon(w)$ with constant $L_\mathcal H$. Since $H$ is continuous, there is a constant $C_{\mathcal H} > 0$ such that $\sup_{x \in B_\epsilon(w)} \|H(x)\| \leq C_{\mathcal H}$. Thus, for all $d \in \Rn$ and $x,y \in B_\epsilon(w)$ and using $\|UU^\top\| = \|VV^\top\| = \|U\| = \|V\| = 1$, we obtain
\begin{align*}
\|[G(x)-G(y)]d\| & = \left \|Q \begin{pmatrix} U^\top[H(x)-H(y)] U & U^\top[H(x)-H(y)] V \\ V^\top[H(x)-H(y)] U & [\gamma(x)-\gamma(y)] I \end{pmatrix} \begin{pmatrix} U^\top d \\ V^\top d \end{pmatrix} \right\| \\ & = \|UU^\top[H(x)-H(y)]d + VV^\top [H(x)-H(y)]UU^\top d + [\gamma(x)-\gamma(y)]VV^\top d\| \\ &\leq 2L_{\mathcal H} \|x-y\| \|d\| + \frac{2}{\delta}\|U^\top[H(x)-H(y)]V\| (\|U^\top H(x) V\| + \|U^\top H(y)V\|) \|d\|. \end{align*}
This shows that $G$ is Lipschitz continuous on $B_\epsilon(w)$ with constant $\left(2 + \frac{4C_{\mathcal H}}{\delta}\right)L_{\mathcal H}$. 
\end{proof}



\end{document}